\documentclass[headings=small,parskip=half]{scrartcl}

\usepackage[utf8]{inputenc}
\usepackage[a4paper, top=1in, bottom=1in, left=0.9in, right=0.9in]{geometry} 
\usepackage{microtype} 

\usepackage{mathtools}
\usepackage{tikz-cd}
\usepackage{xparse}
\usepackage[numbers]{natbib}
\usepackage{authblk}
\usepackage{amsmath,amssymb}
\usepackage{lipsum}
\usepackage{dsfont}
\usepackage{subcaption}
\usepackage{enumerate}
\usepackage{hyperref} 
\usepackage{multicol}
\usepackage[dvipsnames]{xcolor}
\usepackage{MnSymbol}
\usepackage{amsthm}
\usepackage{bm}
\usepackage{float}
\usepackage{graphicx}
\graphicspath{ {./} }

\newcounter{mycounter} 

\newcommand\showmycounter{\stepcounter{mycounter}\themycounter}

\newcommand{\const}{C_{\showmycounter}}

\setcitestyle{square,numbers}

\renewcommand{\P}{\mathbb{P}}

\numberwithin{equation}{section} 

\DeclareMathOperator*{\esssup}{ess\,sup} 
\DeclareMathOperator*{\essinf}{ess\,inf} 

\DeclarePairedDelimiterX{\Norm}[1]{\|}{\|}{\Normargs{#1}}
\NewDocumentCommand{\Normargs}{>{\SplitArgument{1}{;}}m}
{\Normargsaux#1}
\NewDocumentCommand{\Normargsaux}{mm}
{\IfNoValueTF{#2}{#1} {#1\nonscript\:\delimsize\vert\allowbreak\nonscript\:\mathopen{}#2}}%
\def\norm{\Norm*}%

\DeclarePairedDelimiterX{\set}[1]{\{}{\}}{\setargs{#1}}
\NewDocumentCommand{\setargs}{>{\SplitArgument{1}{;}}m}
{\setargsaux#1}
\NewDocumentCommand{\setargsaux}{mm}
{\IfNoValueTF{#2}{#1} {#1\nonscript\:\delimsize\vert\allowbreak\nonscript\:\mathopen{}#2}}%
\def\Set{\set*}%

\providecommand{\keywords}[1]{\textbf{Keywords. } #1}

\providecommand{\MSC}[1]{\textbf{MSC (2020). } #1}

\theoremstyle{plain}
\newtheorem{theorem}{Theorem}[section] 
\newtheorem{lemma}[theorem]{Lemma}
\newtheorem{proposition}[theorem]{Proposition}
\newtheorem{corollary}[theorem]{Corollary}

\theoremstyle{definition}
\newtheorem{definition}[theorem]{Definition} 
\newtheorem{example}[theorem]{Example}
\newtheorem{remark}[theorem]{Remark}
\newtheorem{assumption}{Assumption}

\newcommand{\R}{\mathbb{R}}
\newcommand{\N}{\mathbb{N}}

\renewcommand{\d}{\,\mathrm{d}}

\date{}

\begin{document}

\title{Well-posedness of stochastic parabolic equations with gradient nonlinearities and applications to phase-field models}

\author{Amjad Saef}
\affil{\small{
  Institute of Mathematics, Technische Universität Berlin,\linebreak
  Straße des 17.\ Juni 136, 10623 Berlin, Germany,\linebreak
  e-mail: \href{mailto:saef@math.tu-berlin.de}{saef@math.tu-berlin.de},\linebreak
  ORCID: \href{https://orcid.org/0009-0006-9503-3367}{0009-0006-9503-3367}}}
\maketitle
\vspace{-1.7cm}

\begin{abstract}
We study well-posedness of stochastic 
parabolic equations with gradient nonlinearities. Our analysis is based on recent maximal-regularity frameworks for nonlinear stochastic parabolic equations 
in critical spaces. We extend the existing results by controlling drift and noise coefficient separately. This way we can allow for less regular driving noise in case of subcritical dispersion 
coefficients. Our approach, based on gluings of local solutions, moreover implies new continuation criteria. We then apply our existence result and the continuation criteria to show global well-posedness of stochastic phase-field models of moving boundary problems.
\end{abstract}
\keywords{stochastic partial differential equations,  stochastic maximal regularity, blow-up criteria, weighted spaces, phase field model}

\MSC{35K55, 60H15, 58D25, 76M35, 80A22} 

\section{Introduction}
In this article, we study local well-posedness of systems of semilinear parabolic stochastic partial differential equations
\begin{equation}\label{eq:StrongSolIntroSPDE}
\d {\bm z}_t = \left(\Delta {\bm z}_t + F({\bm z}_t) \right)\d t + G({\bm z}_t)\d {\bm W}_t
\end{equation}
on the $d$-dimensional torus $\mathbb T^d$, $d \geq 2$, with initial datum $\bm z_0$.

Let $2d >q > d \geq 2$, $1 > \delta > d/q$ and let $m$ denote the dimensionality of the system. Let $H^{s,q}(\mathbb T^d)$ denote the Sobolev space with smoothness $s$ and integrability $q$. We assume that both the drift $$F \colon (H^{1,q}(\mathbb T^d))^m \rightarrow (L^{q/2}(\mathbb T^d))^m$$ and dispersion coefficient $$G \colon (H^{\delta,q}(\mathbb T^d))^m \rightarrow \gamma(U,(H^{\sigma,q}(\mathbb T^d))^m)$$ are locally Lipschitz and of quadratic growth in the input variable. Correspondingly, the noise $ {\bm W}_t$ in this equation is a cylindrical Wiener process on some Hilbert space $U$. Here, and in subsequent parts of this section, $\gamma(U, \mathcal B)$ denotes the space of $\gamma$-radonifying operators between $U$ and some Banach space $\mathcal B$.  This is a generalisation of the space of Hilbert--Schmidt operators; for more details, see e.g. \cite{VanNeervenVeraarStochIntUMD}. 

We develop a local-in-time solution theory at critical regularity for the initial datum. This can be viewed as an extension of the abstract framework developed by \citet{AgrestiVeraarNonLinParab} in a specific example. A novelty of our approach is that for subcritical regularity of the dispersion coefficient $G$, we can relax the spatial regularity required of the stochastic forcing. Moreover, we derive maximality of solutions and blow-up criteria. This is the content of our first main result, Theorem \ref{thm:localexistence}. 

\begin{remark}
In the remaining part of this section, we will omit the dimensionality $m$ of the system for readability's sake. We will instead accentuate vector-valued variables through boldface notation.
\end{remark}

We then apply the developed theory to phase-field models of cell motility with singular diffusion terms. To this end, we reduce such systems to reaction-diffusion systems with gradient nonlinearities, by means of a logarithmic transform. The result of our analysis, Theorem \ref{thm:phasefieldglobalstrongex}, provides conditions which ensure global well-posedness of such systems for driving noise with nearly minimal spatial roughness, in a non-renormalised context. To the best of the author's knowledge, these results are not yet part of the literature.

The principal application of the abstract well-posedness result is the case of an interface growth type nonlinearity $F(z) = |\nabla z|^2$ and a Nemytskii operator $G(z) = z^2$. Such equations generally do not satisfy the usual coercivity conditions of variational frameworks, such as \cite{RoecknerFullyLocallyMonotone, GoodairVariational}. Though in the special case $F = |\nabla z_t|^2$, a Cole-Hopf transform reasonably simplies such an equation, this approach already fails in the deterministic case for systems such as
\begin{equation} \label{eq:ColeHopfExample}
\begin{cases}
\partial_t z_t = \Delta z_t + |\nabla z_t|^2 - \nabla z_t \nabla c_t \\
\partial c_t = \Delta c_t + \nabla z_t \nabla c_t.
\end{cases}
\end{equation}
There, the Cole-Hopf transform $\phi = e^u$ results in equations containing logarithmic derivatives
\begin{equation} \label{eq:TransformedColeHopfExample}
\begin{cases}
\partial_t \phi_t = \Delta z_t - \nabla \phi\nabla c_t \\
\partial c_t = \Delta c_t + \frac{\nabla \phi_t}{\phi_t} \nabla c_t.
\end{cases}
\end{equation}
If the initial datum $\phi_0$ is not strictly bounded away from $0$, the logarithmic derivative in the second equation potentially becomes singular. This is a relevant obstruction, since our interest in \eqref{eq:StrongSolIntroSPDE} actually derives from stochastic phase-field models of cell motility the form \begin{equation} 
\label{IntroABSModel}
\left\{ \begin{aligned}
&\partial_t \phi_t = \gamma \Delta \phi_t + g(\phi_t, c_t) + \Psi(\phi_t, c_t)|\nabla \phi_t|, \\
&\mathrm d c_t = \left( D\Delta c_t + D\frac{\nabla \phi_t  \nabla  c_t}{\phi_t}  +  f(\phi_t,  c_t) \right) \mathrm d t  +  b(\phi_t,c_t)\mathrm d  W,
\end{aligned} \right.,
\end{equation}
where the zero level-set signifies the cell exterior.  Here, the domain $\mathcal D$ is assumed to be the $d$-dimensional torus $\mathbb T^d$ for $d \geq 2$, and $\gamma, D > 0$ are diffusion coefficients. The operators $f$ and $g$ denote nonlocal reaction terms, the dispersion coefficient $b$ is a Nemytskii operator, $\Psi$ is a possibly nonlocal Nemytskii operator and $W$ is coloured-in-space, white-in-time Wiener noise. 

The correspondence of \eqref{eq:TransformedColeHopfExample} with \eqref{eq:ColeHopfExample} motivated the investigations in the present work. We observe that \eqref{IntroABSModel} contains exactly the type of singular diffusion term encountered in \eqref{eq:TransformedColeHopfExample}: This term penalises diffusion into the zero level-set, i.e. the cell exterior. 

This type of phase-field model was first applied in biophysical modelling in \cite{ABS}. There, the space-time \textit{order parameter} $\phi(t, \bm x)$, $t \geq 0$, $\bm x \in \mathbb T^d$, changes rapidly but continuously between two equilibria, e.g. $\phi \equiv 0$ and $\phi \equiv 1$. In cell motility models, one may then interpret the region $$\Set{\bm x \in \mathcal D:\phi(t,\bm x) \approx 0}$$ as the cell exterior at time $t$, while the transition region to $\Set{\phi \approx 1}$ represents the cell membrane or a diffuse approximation thereof. Other subject areas where singular diffusion terms of the form $$L_\phi c \coloneqq \Delta c + \frac{\nabla \phi}{\phi} \nabla c$$
are encountered include population ecology \cite{KirkpatrickModel} and stochastic mechanics \cite{CarlenDiff}.

In a previous work \cite{Ich2}, existence of martingale solutions of \eqref{IntroABSModel}, both in the classical sense and in a weighted sense, was shown under rather mild conditions. In particular, the condition $\log \phi_0 \in L^1$ was derived as sufficient, and in a sense optimal, for existence of analytically weak martingale solutions of \eqref{IntroABSModel}; see in particular \cite[Corollary 4.6 and Remark 4.12]{Ich2}. For a more detailed overview of applications and mathematical analysis of such phase-field models, we also refer the interested reader to the aforementioned work.

However, in that work, uniqueness and measurable dependence on driving noise of solutions had not been shown. A central question behind the present work was whether uniqueness of solutions can be derived even when the initial data are not bounded away from $0$. As $\phi \approx 0$ signifies the cell exterior in applications of \eqref{IntroABSModel}, this is a desirable stability property. Theorem \ref{thm:phasefieldglobalstrongex} answers this question positively, since $\log \phi_0$ need only take values in Besov spaces of critical order; see also Remark \ref{remark:MorreyRemark}.

\subsection{Overview of previous approaches}
We first provide a short discussion of previous approaches to this problem. Consider the equation 
\begin{equation} \label{eq:StrongSolIntroColeHopf}
\partial_t z_t = \Delta z_t + |\nabla z_t|^2
\end{equation}
Notably, \citet{BENARTZI2002343} solved deterministic variants of \eqref{eq:StrongSolIntroColeHopf} using fixed-point arguments in a weighted sup-norm in an auxiliary space with polynomial time-weight $t^\alpha$, $\alpha >0$. This approach is sometimes called \textit{Kato's method} \cite{KunstmannHaakKato}, after the seminal work of \citet{FujitaKato}. Here, appropriate spaces of initial conditions and the corresponding weights are determined by the scaling properties of solutions of \eqref{eq:StrongSolIntroColeHopf}. A crucial observation in this specific example is that the parabolic scaling $(t,x) \mapsto (\lambda^2 t, \lambda x)$ leaves the solution space to this equation invariant. 

Finding sufficient regularity conditions on initial data using scaling arguments goes back at least to the already mentioned work by Fujita and Kato. Namely, for the incompressible Navier--Stokes equations, the parabolic scaling
$$
u(t,x)\mapsto u_\lambda(t,x) \coloneqq \lambda u(\lambda^2t,\lambda x)
$$
leaves invariant the homogeneous Sobolev norm
$
\norm{u_0}_{\dot H^{d/2-1}},
$, $d \geq 2$,
and this regularity class is the appropriate class of initial conditions to apply their fixed-point argument. In this sense, a space of initial data is called \textit{critical} if its norm is invariant under the natural scaling of the equation. Spaces of higher regularity are subcritical, while spaces of lower regularity are supercritical. The notion of a critical space was subsequently narrowed down in works by \citet{WeisslerCriticalSchroedinger} and \citet{DanchinCriticalViscous}.

In the scale of Lebesgue norms, the natural scaling of \eqref{eq:StrongSolIntroColeHopf} leaves invariant only $L^\infty$-type norms of $z$.  However, if one additionally measures derivatives, then in the homogeneous Sobolev scale one locally has the Euclidean scaling relationship
$$
\norm{z_0(\lambda\cdot)}_{\dot H^{s,q}}
\sim
\lambda^{s-d/q}\norm{z_0}_{\dot H^{s,q}},
$$
and therefore the critical Sobolev index of initial data is $s_c=\frac dq$. A similar argument works for Besov spaces $B^s_{q,p}$. Similarly, appropriate weighted auxiliary space-time norms are determined by the natural scaling of the equation. For $d<q$, let $\alpha = \frac12 - \frac{d}{2q}$ and 
$$
\mathcal N^{\alpha,\infty}(T,z) =\sup_{t\in (0,T]}
t^{\alpha}
\norm{\nabla z_t}_{L^q}
$$
Then, after rescaling
$
z_\lambda(t,x)=z(\lambda^2t,\lambda x)
$, one locally has the Euclidean scaling relationship
$
\mathcal N^{\alpha,\infty}(T/\lambda^2,z_\lambda)
= \mathcal N^{\alpha,\infty}(T,z)$.
These ideas lie at the core of the approach laid out in \citet{BENARTZI2002343}. Analyticity and hypercontractivity of the heat semigroup $(S(t))_{t \geq 0}$ then show that the map 
$z \mapsto   S\ast |\nabla z|^2 =\int_0^\cdot S(\cdot-s) |\nabla z_s|^2 \d s$ satisfies the local Lipschitz property 
\begin{equation}\label{eq:PinfinityLocalLipschitz}
\mathcal N^{\alpha,\infty}(T,S\ast (|\nabla z_1|^2-|\nabla z_2|^2)) \leq (\mathcal N^{\alpha,\infty}(T,z_1)+\mathcal N^{\alpha,\infty}(T,z_2))\,\mathcal N^{\alpha,\infty}(T,z_1-z_2)
\end{equation}
The use of time-weighted solution spaces is in line with the observation that parabolic smoothing regularises singularities instantaneously. Moreover, this heuristic is widely used with considerable success in isolating regularity conditions on initial data to ensure existence and uniqueness of solutions of nonlinear (S)PDEs. Thus, these methods suggest themselves as appropriate to study well-posedness of the transformed phase-field models.

However, weighted $L^\infty$-in-time approaches do not translate well to stochastic equations. The main reason is that stochastic integrals do not generally interact well with weighted supremum norms, since the additional regularity encoded by the time weight cannot be incorporated directly into Burkholder--Davis--Gundy type estimates.

It turns out that this obstruction can be circumvented if one switches to a weighted $L^p([0,T];t^\beta \d t)$ framework for $p < \infty$ and $\beta \in \mathbb R$. In this case, \citet{PruessSimonettWilke} developed a quite general deterministic theory of weighted maximal regularity in abstract critical spaces, which is meant to unify the previously developed scaling heuristic across a multitude of examples. They give an explicit definition of what constitutes a critical nonlinearity with respect to a given space and embed arguments such as \eqref{eq:PinfinityLocalLipschitz} into their maximal regularity framework.

We remark that the weight $\beta$ is not generally equal to $\alpha$; rather, $\beta = \alpha p -1$. The definition of the corresponding norm $\mathcal N^{\alpha,p}$ follows the pattern of abstract interpolation spaces (cf. Prop. \ref{prop:InterpolationIdentity}): if $p < \infty$, then 
$$\mathcal N^{\alpha,p} \coloneqq \left(\int_0^T \left(t^\alpha \norm{z_t}_{H^{1,q}}\right)^p\frac{\d t}{t}\right)^\frac{1}{p}$$ 
is the appropriate fixed-point norm. In the following, we suppress reference to $T$ for readability. By weighted maximal regularity theory and embeddings of weighted Sobolev spaces, we find that $$\mathcal N^{\alpha,2p}(S\ast (|\nabla z_1|^2-|\nabla z_2|^2)) \lesssim \mathcal N^{\alpha,p}(|\nabla z_1|^2-|\nabla z_2|^2) \lesssim (\mathcal N^{\alpha,2p}(z_1)+\mathcal N^{\alpha,2p}(z_2))\mathcal N^{\alpha,2p}(z_1- z_2)$$
The important contribution to transfer such estimates to the stochastic setting was made by \citet[Sec. 7]{AgrestiVeraarStability}. There, they showed that for suitably regularising generators, stochastic convolutions satisfy a weighted maximal regularity property. One notable consequence of their results is that for suitable weights $\beta$ and stochastically integrable $G \in L^p([0,t],t^\beta;\gamma(U,H^{s,q}))$, the stochastic convolution $S \diamond G$ satisfies
$$\norm{S \diamond G}_{L^p(\Omega \times [0,T],\d \mathbb P \otimes t^\beta \d t;H^{s+1,q}}) \leq \norm{G}_{L^p(\Omega \times [0,T],\d \mathbb P \otimes t^\beta \d t;\gamma(U,H^{s,q}))}.$$ Subsequently, they carefully and quite generally transferred the abstract critical framework by Prüss, Simonett and Wilke to the stochastic case \cite{AgrestiVeraarNonLinParab}. For detailed surveys of the deterministic and stochastic theory, see \cite{WilkeSurvey} and \cite{VeraarSurvey}, respectively.

\section{Methodology and main results}

Throughout this section, we fix some separable Hilbert space $U$ and a cylindrical Wiener process ${\bm W}_t$ on $U$, with respect to a filtered probability space $(\Omega, \mathcal F, (\mathcal F_t)_{t\geq 0}, \mathbb P)$ satisfying the usual conditions. 
Further, we assume that the dimension $d$ and the integrability exponent $q$ satisfy $d \geq 2$ and $q \in (d,2d)$. Further, let $\sigma > 0$ denote a smoothness parameter and denote by $S = (S(t))_{t \geq 0}$ the analytic semigroup generated by the periodic 
Laplacian on $(L^q(\mathbb T^d))^m$. In this section, we establish local existence of probabilistically strong, analytically mild solutions of the \eqref{eq:StrongSolIntroSPDE}. Ergo, these solutions should satisfy 
\begin{equation}
{\bm z}_t
=
S(t){\bm z}_0
+\int_0^{t} S(t-s)F({\bm z}_s) \d s
+\int_0^{t} S(t-s)G({\bm z}_s)\d {\bm W}_s,
\end{equation}
in a suitable sense. 
We impose the following assumptions on the nonlinearities. 
\begin{assumption} \label{ass:DriftNonlin} For any ${\bm u},{\bm v} \in H^{1,q}(\mathbb T^d)$, $F \colon H^{1,q}(\mathbb T^d) \rightarrow L^{q /2}(\mathbb T^d)$ satisfies
$$\norm{F({\bm u}) - F({\bm v})}_{L^{q/ 2}(\mathbb T^d)} \leq \const \left(1+\norm{{\bm u}}_{H^{1,q}(\mathbb T^d)}+\norm{{\bm v}}_{H^{1,q}(\mathbb T^d)}\right)\norm{{\bm u}-{\bm v}}_{H^{1,q}(\mathbb T^d)}.$$
\end{assumption}
\begin{example}
This holds, for example, when 
$$
|F({\bm u}) - F({\bm v})| \leq \const (1+|{\bm u}|+|{\bm v}|+|\nabla {\bm u}| + |\nabla {\bm v}|)(|{\bm u}-{\bm v}|+|\nabla {\bm u}-\nabla {\bm v}|)
$$
pointwise. However, this condition also allows for nonlocal nonlinearities. 
\end{example}

\begin{assumption} \label{ass:nemytskiiLipschitz}
The dispersion coefficient $G \colon H^{\delta,q}(\mathbb T^d) \rightarrow \gamma(U,H^{\sigma,q}(\mathbb T^d))$ satisfies $$\norm{G({\bm u})-G({\bm v})}_{\gamma(U,H^{\sigma,q}(\mathbb T^d))} \leq \const \left(1+\norm{{\bm u}}_{H^{\delta,q}(\mathbb T^d)}+\norm{{\bm v}}_{H^{\delta,q}(\mathbb T^d)}\right)\norm{{\bm u} - {\bm v}}_{H^{\delta,q}(\mathbb T^d)}$$
for some $\delta \in \left(\frac dq, 1\right)$ and  $\sigma > \delta-\frac dq > 0$.
\end{assumption}

Let $B^s_{q,p}(\mathbb T^d)$ denote the Besov space with smoothness $s$, integrability $q$ and fine index $p$ on $\mathbb T^d$. The following theorem constitutes the main result of this subsection.
\begin{theorem}[Local existence of mild solutions] \label{thm:localexistence} Let $d \geq 2$, $q \in (d,2d)$, $\delta \in \left(\frac dq, 1\right)$ and $\sigma >\delta-\frac dq$. Assume that $F$ and $G$ satisfy Assumptions \ref{ass:DriftNonlin} and \ref{ass:nemytskiiLipschitz}. Then, for any $p \in \left[\frac{q}{\delta q-d},\infty\right)$ and $\mathcal F_0$-measurable ${\bm z}_0 \in L^0(\Omega;B^{d/q}_{q,2p}(\mathbb T^d))$, there exists a unique maximal local solution $({\bm z},\tau)$ of \eqref{eq:FullNoiseSPDE} in the sense of Definition \ref{def:localsol}, with
\begin{enumerate}
\item (Maximality) If $(\widetilde {\bm z},\widetilde\tau)$ is any other local solution, then $\widetilde\tau\leq \tau_{\max}$ a.s. and
$\widetilde {\bm z}_t={\bm z}_t$ for all $t<\widetilde\tau$ a.s.
\item (Blow-up criterion) For any stopping time $\tau \leq \tau_{\mathrm{max}}$,
\begin{equation}
\limsup_{t\nearrow \tau_{\max}}   ~\norm{{\bm z}}_{L^\infty([\tau,t];B^{d/q}_{q,2p}(\mathbb T^d))}+ \norm{{\bm z}}_{L^{2}([\tau,t];H^{1,q}(\mathbb T^d))} + \norm{{\bm z}}_{L^4([\tau,t];H^{\delta,q}(\mathbb T^d))} =\infty.
\end{equation}
$\mathbb P$-almost surely on the event $\{\tau < \tau_{\max}<T\}$,
\end{enumerate}
\end{theorem}
For dispersion coefficients of critical growth (i.e. $\delta = 1$), existence and uniqueness of an analytically weak notion of solution is already implied by the local existence theory in \cite{AgrestiVeraarNonLinParab} under our conditions. Though these results extend to $\delta < 1$ by the embedding $H^{1,q} \hookrightarrow H^{\delta,q}$, the theory of \cite{AgrestiVeraarNonLinParab} continues to require $\sigma \geq 1-\frac{d}{q}$ in this stronger setting. This is particularly restrictive in the case of Nemytskii operators, which often satisfy local Lipschitz estimates in $H^{\delta,q}(\mathbb T^d)$  for any $\delta > \frac dq$.

We then apply this result to the phase-field model \eqref{eq:StrongPhaseFieldABS}. Importantly, the particular case of Nemytskii operators permits any spatial regularity $\sigma > 0$.  We provide a detailed study of conditions on coefficients and initial data under which global existence and uniqueness of probabilistically strong solutions to \eqref{eq:StrongPhaseFieldABS} follows. The following statement contains the main result of our analysis, see also Theorem \ref{thm:phasefieldglobalstrongex}.

\begin{theorem} Let the coefficients of \eqref{IntroABSModel} satisfy assumptions \ref{ass:ColeHopfGLipschitz} to \ref{ass:CriticalNemytskiiDispersion}. Let $W^Q$ take values in $H^{\sigma,q}$ for some arbitrary $\sigma > 0$ and $q \in (d,2d)$. Then, for suitable $p$ and $\mathcal F_0$-measurable $\log \phi_0, c_0 \in B^{d/q}_{q,2p}$ with $0 \leq \phi_0, c_0 \leq 1$ almost surely, there exists a unique global-in-time solution $(\phi, c)$ of \eqref{IntroABSModel} with initial data $\phi_0, c_0$, in the sense of Def. \ref{def:localsol}. In particular, $\mathbb P$-a.s.,
$$
\phi \in C([0,T];B^{d/q}_{q,2p}) \cap C((0,T];H^{1,q}),~ c \in L^4([0,T];H^{1,q})
$$
with $$\sup_{t\in (0,T]} t^{\frac12 - \frac{d}{2q}} \norm{\phi}_{H^{1,q}} < \infty$$ and
$0 \leq \phi_t, c_t \leq 1$ for all $t \in [0,T]$. Moreover, $(\phi,c)$ is the unique solution with these regularities. 
\end{theorem}

\begin{remark}
After the contents of this manuscript were largely completed, the author became aware of the recent preprint \cite[Theorem 4.5]{agresti2026optimallocaltheoryreactiondiffusion}, where a more general well-posedness condition than Theorem \ref{thm:localexistence} is announced. The proof of that statement is not included there and is indicated to appear in forthcoming work. The present paper gives a complete proof of the special case stated above, using the methods developed below. 
\end{remark}

Thus, for subcritical dispersion coefficients, relative to the deterministic nonlinearity $F$, we can allow for considerably rougher forcings than the results of \cite{AgrestiVeraarNonLinParab} allow for, under comparable conditions on the initial data. To achieve this, we isolate deterministic and stochastic components of the equations in the fixed point argument. Writing the solution as $z=u+v$, we let $u$ carry the initial trace and solve a deterministic convolution equation in a weighted deterministic maximal-regularity space $\mathbb X_1$, while $v$ is treated as the zero-trace stochastic convolution in a different weighted stochastic maximal-regularity space $\mathbb Y_\delta$, where $\delta$ indicates the lower regularity needed to control the dispersion coefficient. 

The crucial insight is contained in Lemma \ref{lemma:XYembeddings}. Namely, even if the weights of the stochastic and deterministic components are different, the respective maximal regularity spaces embed into an intermediate scale of weighted Lebesgue spaces under which the fixed point argument closes. This in particular needs particular care to ensure validity of weighted space-time Sobolev embeddings and introduces the lower bound $\sigma > \delta -\frac{d}{q}$ on the Sobolev regularity of the driving noise process.

\begin{figure}[H]
    \centering
    \includegraphics[width=0.9\textwidth]{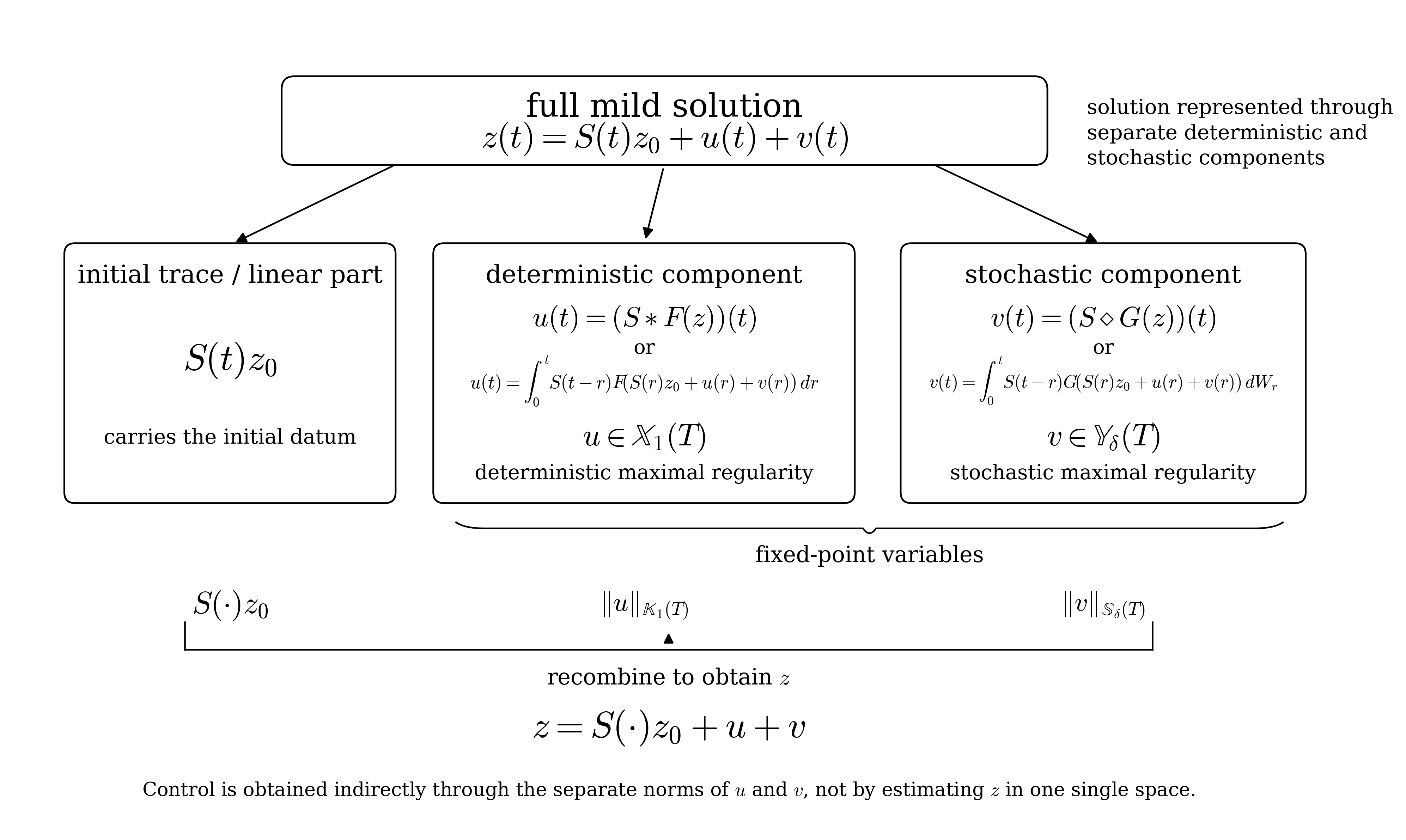}
    \caption{Illustration of decomposition of mild solutions}
    \label{fig:Strategy}
\end{figure}

A further methodological distinction lies in the consequent use of the semigroup formalism, compared to the analytically weak formalism in \cite{AgrestiVeraarNonLinParab}. To this end, we introduce a notion of local mild solution with inbuilt local regularity assumptions (cf. Definition \ref{def:localsol}). We then prove uniqueness of maximal solutions in this mild sense, together with a blow-up criterion.
Central to this approach is an extension procedure for local solutions: once a solution is controlled up to a stopping time, one may restart the dynamics from the random terminal value, driven by the randomly shifted Wiener process, and then concatenate the old and new piecewise solutions. This is technically different, though conceptually comparable to the approach in e.g. \citet{AgrestiVeraarBlowup}.

The resulting blow-up criterion is determined by conditions which ensure well-definedness of the deterministic and stochastic convolution against the heat semigroup, and is generally incomparable to the criteria laid out in \cite{VeraarSurvey}. In special cases, it provides slightly weaker conditions to exclude blow-up.

\begin{figure}[H] 
    \centering
    \includegraphics[width=0.7\textwidth]{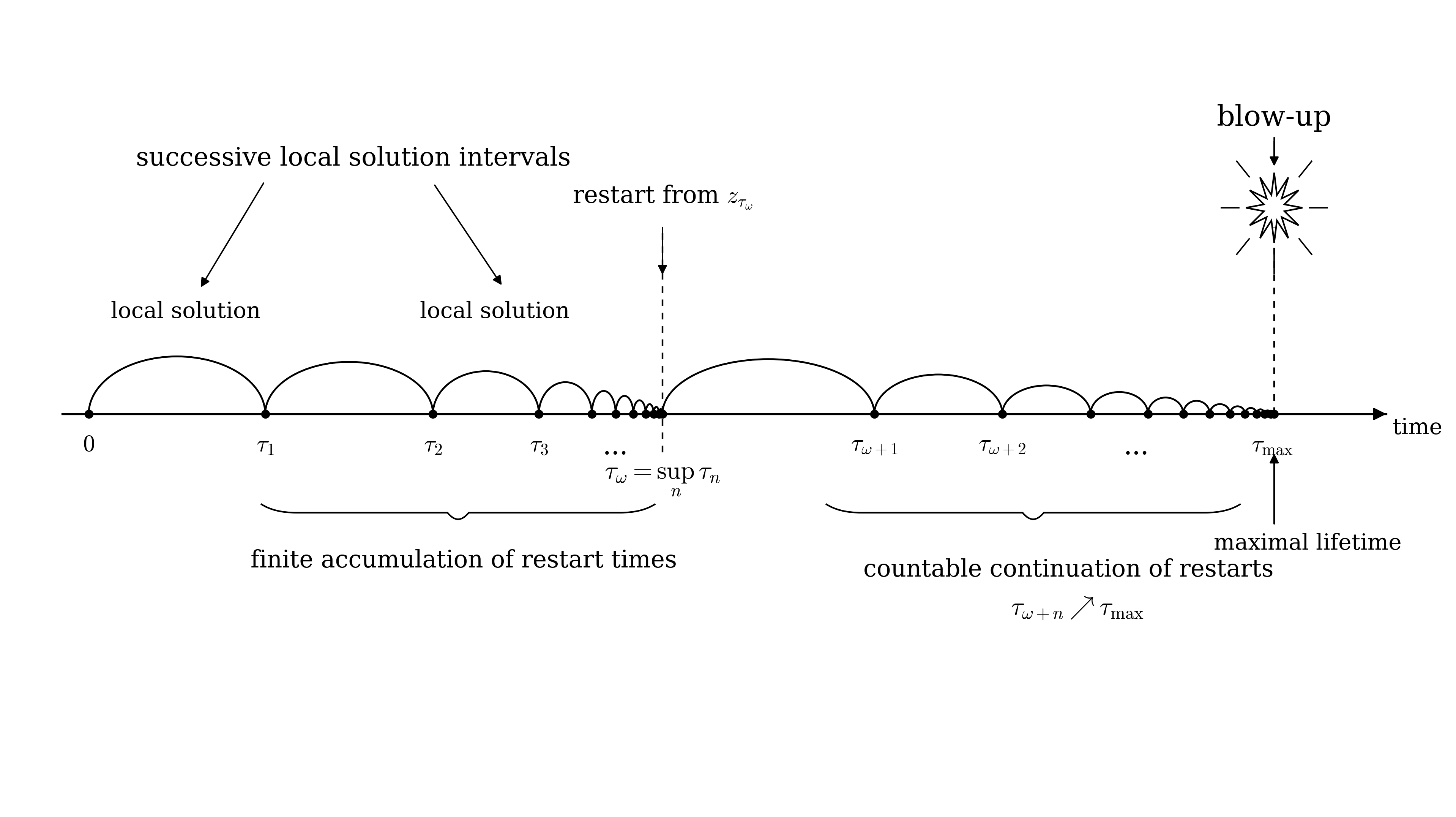}
    \caption{Illustration of gluings of solutions}
    \label{fig:LocalSolution}
\end{figure}

\subsection{Contribution to the existing maximal regularity theory in critical spaces}
We now lay out how our results contribute to this special case of the critical theory of \citet{AgrestiVeraarNonLinParab}. To this end, we need to first study the implications of the existence- and uniqueness result \cite[Theorem 4.8]{AgrestiVeraarNonLinParab} in the setting of this chapter.

Introduce the spaces 
$X_0 \coloneqq H^{-s,q}(\mathbb T^d)$ and $X_1 \coloneqq H^{2-s,q}(\mathbb T^d)$, so that the intermediate complex interpolation space is given by
$X_{1/2} \coloneqq [X_0,X_1]_{1/2} = H^{1-s,q}(\mathbb T^d)$. In the following, it is convenient to
write $p \in (2,\infty)$ for the time-integrability exponent of the maximal regularity theorem and to set
$$\alpha \coloneqq \frac{1+\kappa}{p}$$ given some $\kappa \in [0,\frac{p}{2}-1)$. 
In our case, the results in \cite{AgrestiVeraarNonLinParab} imply existence  of local-in-time solutions of \eqref{eq:StrongSolIntroSPDE} for initial conditions in the real interpolation space
$X_{1-\alpha,p} \coloneqq  (X_0,X_1)_{1-\alpha,p} = B^{2-s-2\alpha}_{q,p}(\mathbb T^d)$, whenever drift and dispersion coefficient satisfy $$\norm{F(\bm u) - F(\bm v)}_{X_0}  \leq \const \left(1+\norm{\bm u}^{\rho_1}_{X_{\varphi_1}}+\norm{\bm v}^{\rho_1}_{X_{\varphi_1}}\right)\norm{\bm u - \bm v}_{X_{\beta_1}}$$
and $$\norm{G(\bm u) - G(\bm v)}_{\gamma(U,X_{1/2})}  \leq \const \left(1+\norm{\bm u}^{\rho_2}_{X_{\varphi_2}}+\norm{\bm v}^{\rho_2}_{X_{\varphi_2}}\right)\norm{\bm u - \bm v}_{X_{\beta_2}},$$
for \begin{equation} \label{eq:CriticalRelation}
\rho_i\left(\varphi_i-1+\alpha\right)+\beta_i \leq 1,\quad i = 1,2. 
\end{equation}
Since the drift $F$ is modelled on the quadratic gradient term $|\nabla z|^2$,  the natural choice in the Lipschitz condition is $\rho = 1$ and $\varphi_1 = \beta_1$, together with $X_{\beta_1} = H^{1,q}(\mathbb T^d)$. Since
$$
X_{\beta_1} \coloneqq  [X_0,X_1]_{\beta_1} = H^{-s+2\beta_1,q}(\mathbb T^d),
$$
this forces $\beta_1 = \frac{1+s}{2}$. Moreover, in the present situation the map $F \colon H^{1,q} \to L^{q/2}$ can be regarded as $X_0$-valued only through the embedding $L^{q/2}(\mathbb T^d) \hookrightarrow H^{-s,q}(\mathbb T^d)$, and this identifies the loss of smoothness as $s = \frac{d}{q}$. Consequently, the strict criticality relation imposed by \eqref{eq:CriticalRelation} becomes
$$
\alpha = 2(1-\beta_1) = 1-\frac{d}{q}.
$$
Note that $\kappa < p/2-1$ iff $\alpha < \frac12$ iff $q <2d$.
Therefore, we can compute the critical trace space as
$$
X_{1-\alpha,p} = B^{2-s-2\alpha}_{q,p}(\mathbb T^d) = B^{d/q}_{q,p}(\mathbb T^d).
$$
Observe that a restriction on $p$ enters through these conditions; namely, $\kappa \geq 0$ is equivalent to $$\alpha \coloneqq \frac{1+\kappa}{p}\geq \frac{1}{p} \iff 1-\frac{d}{q} \geq \frac{1}{p} \iff p \geq \frac{q}{q-d}.$$
So if $p \geq \frac{q}{q-d}$ and dispersion coefficient $G$ also satisfies the above condition, a direct application of the Agresti--Veraar theorem yields local maximal solutions for $\mathcal F_0$-measurable
$
{\bm z}_0 \in  B^{d/q}_{q,p}(\mathbb T^d)
$, 
since the Laplacian satisfies the stochastic maximal regularity hypothesis on $\mathbb T^d$. 

A few restrictions are worth noting. 
\begin{enumerate}[(i)]
    \item First, the endpoint $q=d$ is excluded in the setup of \cite{AgrestiVeraarNonLinParab}, since in that case $s=\frac{d}{q}=1$ and therefore $\beta = \frac{1+s}{2}=1$, whereas the strict theory requires $\beta<1$. Hence one needs $q>d$ for the above choice to fit into the framework.

    The endpoint $q=2d$ can however occur in the exceptional case $p=2$ with $\kappa=0$. We do not cover this case in this article. This is not problematic, since our conditions are not subsumed in this special case. The reason is that the Besov embeddings do not yield that $B^{d/q}_{q,r} \hookrightarrow B^{1/2}_{2d,2}$ for $r > 2$.
    \item Finally, even if the dispersion estimate is stronger than what is needed to verify the hypothesis of the framework in \cite{AgrestiVeraarNonLinParab}, one cannot relax the spatial regularity required of the noise in the theorem.
Indeed, the stochastic term is always measured in
$\gamma(U,X_{1/2})$, which is already determined by the deterministic nonlinearity.  
\end{enumerate}

The latter observation lies at the basis of our contributions to the solution theory of such equations at critical regularity. Since the critical space of initial conditions is constrained by the deterministic theory, there is little space for improvement (see e.g. \cite{SimonettPruess} and subsequent works).

However, we demonstrate in this article that tighter estimates on the dispersion coefficient can be exploited in order to relax the spatial regularity required of the stochastic forcing.

\begin{remark}
For $\delta = 1$, one can see that our conditions on initial data are essentially equivalent to the condition in \cite{AgrestiVeraarNonLinParab}. Except for the edge case $\sigma = 1-d/q$, our result applies to this class of initial conditions, too. This follows from the Besov embedding $B^{d/q}_{q,p}(\mathbb T^d) \hookrightarrow B^{d/q}_{q,2p}(\mathbb T^d)$.  On the other hand, if $p \geq \frac{q}{q-d}$, then trivially $2p  \geq \frac{q}{q-d}$, so the existence theorem applies to ${\bm z_0} \in B^{d/q}_{q,2p}(\mathbb T^d)$. We note though that the continuity properties of the respective solutions are different, since we consider weakly continuous solutions; see also Remark \ref{remark:improvedcontinuity}. Moreover, the differing conditions on the fine structure parameter will start to make a difference when we compare blow-up conditions, see also the discussion below.
\end{remark}

Now consider an $\mathcal F_0$-measurable $z_0 \in B^{d/q}_{q,2p}$ for $p \geq \frac{q}{\delta q-d}$. It is worth comparing our blow-up criterion with the criteria in \cite[Section 5, Thm. 1]{VeraarSurvey}. Since we do not obtain the first blow-up criterion $$\mathbb P\left(\tau < T,B^{d/q}_{q,2p}-\lim_{t \nearrow \tau} {\bm z}_t \text{ exists}\right) = 0$$ in this article, we compare our result to the criterion 
\begin{equation} \label{eq:avblowup}
\mathbb P\left(\tau < T,\norm{{\bm z}}_{L^\infty([0,t];X_{1-\frac{1+\kappa}{2p},2p})}
+
\norm{{\bm z}}_{L^{2p}([0,\tau];X_{1-\frac{\kappa}{2p}})}
<\infty \right) = 0.
\end{equation}
The above identity is the content of  \cite[Theorem 5.1, claim (2)]{VeraarSurvey}.

\begin{enumerate}[(i)]
    \item Recall that the strict criticality relation imposes the condition $\frac{1+\kappa}{2p} = 1-\frac{d}{q}$, so $X_{1-\frac{1+\kappa}{2p},2p} = B^{d/q}_{q,2p}$ and more importantly, $X_{1-\frac{\kappa}{2p}} = H^{\frac{d}{q} + \frac{1}{p},q} \hookrightarrow H^{\delta,q}$, so the criterion excludes blow-up whenever \begin{equation}\label{eq:AgrestiVeraarBlowup}
    \norm{{\bm z}}_{L^\infty([0,t];B^{d/q}_{q,2p})}
    +
    \norm{{\bm z}}_{L^{2p}([0,\tau];H^{d/q + 1/p,q})}
    <\infty,
    \end{equation}
    where in particular $d/q+1/p = \delta$ if $p = \frac{q}{\delta q-d}$.
    \item By contrast, our continuation criterion is based on the finiteness of
    $$
    \Xi_t({\bm z})
    =
    \norm{{\bm z}}_{L^\infty([0,t];B^{d/q}_{q,2p})}
    +
    \norm{{\bm z}}_{L^2([0,t];H^{1,q})}
    +
    \norm{{\bm z}}_{L^4([0,t];H^{\delta,q})}.
    $$
    \item We remark that this expression simplifies for $\delta < \frac{1}{2}+\frac{d}{2q}$, since then$$L^\infty([0,t];B^{d/q}_{q,2p}) \cap L^2([0,t];H^{1,q}) \hookrightarrow L^4([0,t];H^{\delta,q})$$
    \item Our condition is not subsumed by \eqref{eq:avblowup}: even in the special case $\delta = 1$, the condition $q \in (d,2d)$ results in
    $$L^\infty([0,t];H^{d/q,q}) \cap L^4([0,t];H^{1,q}) \not \hookrightarrow L^{2r}([0,T];H^{d/q+1/r,q})$$ for any $r \geq \frac{q}{q-d}$.
    \item Meanwhile, condition \eqref{eq:avblowup} also is not subsumed by ours. Evidently, criterion \eqref{eq:avblowup} requires less smoothness for $p > \frac{q}{q-d}$. Moreover, since $2p \geq \frac{2q}{\delta q-d}  > 4$ if $q \in (d,2d)$, even for $\delta = 1$, condition \eqref{eq:avblowup} allows to account for lower Besov fine indices in the range $[\frac{q}{q-d},\frac{2q}{q-d}]$.
    \item In the main comparable case $\delta=1$ and $p = \frac{q}{q-d}$, our blow-up condition collapses to 
    $$\norm{{\bm z}}_{L^\infty([0,t];B^{d/q}_{q,2p})}
    +
    \norm{{\bm z}}_{L^4([0,t];H^{1,q})}.
    $$
    Then, $d/q+1/p = 1$ and condition \eqref{eq:avblowup} reads 
    $$
    \norm{{\bm z}}_{L^\infty([0,t];B^{d/q}_{q,2p})}
    +
    \norm{{\bm z}}_{L^{2p}([0,\tau];H^{1,q})}
    <\infty,
    $$
    which is strictly stronger than our condition since then $2p = \frac{2q}{q-d}>4$. Thus, in some instances, our criterium provides less strict conditions to exclude blow-up.
\end{enumerate}
\section{Well-posedness of stochastic parabolic equations with gradient nonlinearities}
In this section, we derive the local-in-time well-posedness theory for stochastic equations of the type 
$$
\d {\bm z}_t = \left(\Delta {\bm z}_t + F({\bm z}_t) \right)\d t + G({\bm z}_t)\d {\bm W}_t,
$$
where $F$ and $G$ satisfy Assumptions \ref{ass:DriftNonlin} and \ref{ass:nemytskiiLipschitz}, respectively, for $q \in (d,2d)$, $d \geq 2$, $\delta \in (\frac{d}{q},1)$ and $\sigma > \delta - \frac{d}{q}$. The main results of this section, Prop. \ref{prop:localisation} and Thm. \ref{thm:maximallocalsolution}, show existence of maximal local solutions and establish the blow-up criteria laid out in Theorem \ref{thm:localexistence}.
\subsection{Weighted vector-valued Sobolev spaces}
Let $\mathcal B$ denote a separable, type $2$, UMD Banach space and let $0 < T \leq \infty$ be given. In this subsection, we introduce weighted vector-valued Sobolev spaces and state embedding theorems between such spaces. These embeddings will be the main tools to close fixed point arguments.

The contents of this subsection are mainly based on \citet[Section 2.2]{AgrestiVeraarNonLinParab}, where the proofs of these theorems and related results in the literature are discussed in detail. We will not include the proofs here and instead refer the reader to this article.

From hereon, let $w_{\beta}(t)\coloneqq t^{\beta}$, $t\in [0,T]$, denote a power-type weight. By
$L^p([0,T],w_{\beta};\mathcal B)$, we denote  the Banach space of all strongly measurable functions $u:[0,T]\to \mathcal B$ for which
$$
\norm{u}_{L^p([0,T],w_{\beta};\mathcal B)}^p\coloneqq \int_0^T\norm{u_t}^p_{\mathcal B} w_{\beta}(t)\d t<\infty.
$$
\begin{remark}
Assume that $-1 <\beta < \frac{p_0}{2}-1$, and $p_1 < \frac{p_0}{1+\beta}$. Then for each $T< \infty$, $$L^{p_0}([0,T],w_\beta;\mathcal B) \rightarrow L^{p_1}([0,T];\mathcal B)$$ by Hölder's inequality. In particular, this applies to $p_1 = 2$ by assumption.
\end{remark}
\begin{definition}[Weighted $\mathcal B$-valued Sobolev space]

For $k \geq 1$, we denote by $W^{k,p}([0,T],w_{\beta};\mathcal B)$ the set of all $f\in L^p([0,T],w_{\beta};\mathcal B)$ such that their $j$-th distributional derivatives, $0 \leq j \leq k$, have representatives $f^{(j)} \in L^p([0,T],w_{\beta};\mathcal B)$. When endowed with the norm
$$
\norm{u}_{W^{k,p}([0,T],w_{\beta};\mathcal B)}\coloneqq\sum_{j=0}^k\norm{u^{(j)}}_{L^p([0,T],w_{\beta};\mathcal B)},
$$
this becomes a Banach space.
\end{definition}
As stated in \cite{AgrestiVeraarNonLinParab}, the trace map $f\mapsto f(0)$ is a bounded mapping from $W^{1,p}([0,T],w_{\beta};\mathcal B)$ into $\mathcal B$ for $\beta \in (-1,p-1)$.  We can then define the closed subspace of zero-trace functions as
$$
{_0W}^{1,p}([0,T],w_{\beta};\mathcal B) = \left\{f\in W^{1,p}([0,T],w_{\beta};\mathcal B): f(0)=0 \right\}.
$$
Fractional order Sobolev spaces will now be defined via complex interpolation between $L^p([0,T],w_{\beta};\mathcal B)$ and ${_0W}^{1,p}([0,T],w_{\beta};\mathcal B)$. 
\begin{definition}
We define 
$$
H^{\theta,p}([0,T],w_{\beta};\mathcal B) \coloneqq [L^p([0,T],w_{\beta};\mathcal B),W^{1,p}([0,T],w_{\beta};\mathcal B)]_{\theta}
$$ 
and 
$$ 
{_0H}^{\theta,p}([0,T],w_{\beta};\mathcal B)\coloneqq[L^p([0,T],w_{\beta};\mathcal B),{_0W}^{1,p}([0,T],w_{\beta};\mathcal B)]_{\theta}.
$$
through complex interpolation.
\end{definition}
\begin{remark}
It is a standard result of complex interpolation that the embedding $_0W \hookrightarrow W$ transfers to $_0H$, i.e. 
\begin{equation}\label{eq:contractiveembd0}
{_0H}^{\theta,p}([0,T],w_{\beta};\mathcal B)\hookrightarrow H^{\theta,p}([0,T],w_{\beta};\mathcal B)  \text{ contractively}.
\end{equation}
Moreover, for any $0 < T_1 < T_2$, 
$$H^{\theta,p}([0,T_2],w_{\beta};\mathcal B) \hookrightarrow H^{\theta,p}([0,T_1],w_{\beta};\mathcal B)  \text{ contractively}$$ via the restriction operator, and the same holds for $_0H^{\theta,p}$.
\end{remark}

\begin{theorem}[\cite{AgrestiVeraarNonLinParab}]\label{thm:equivalence_h_H}
Let $\mathcal B$ be a UMD space, $p \in (1, \infty)$, $\beta \in (-1,p-1)$, $s\in (0,1)$. If $s \neq \frac{1+\beta}{p}$, then
$$
_0H^{s,p}(\R_+,w_{\beta};\mathcal B) = \begin{cases}
\{u\in H^{s,p}(\R_+,w_{\beta};\mathcal B): u(0) = 0\}, & s>\frac{1+\beta}{p}, \\
H^{s,p}(\R_+,w_{\beta};\mathcal B), & s<\frac{1+\beta}{p},
\end{cases}
$$
isomorphically.
\end{theorem}

\begin{theorem}[\cite{LindemulderMeyriesVeraar, AgrestiVeraarNonLinParab}]
Theorem \ref{thm:equivalence_h_H} extends to finite intervals $[0,T]$. In particular, if $s\neq \frac{1+\beta}{p}$, then ${_0H}^{s,p}([0,T],w_{\beta};\mathcal B)$ is a closed subspace of $H^{s,p}([0,T],w_{\beta};\mathcal B)$.
\end{theorem} 
As a consequence the estimate $\norm{u}_{{_0H}^{s,p}([0,T],w_{\beta};\mathcal B)}\simeq \norm{u}_{H^{s,p}([0,T],w_{\beta};\mathcal B)}$ holds, where we need the condition  $u(0)=0$ if $s>\frac{1+\beta}{p}$. The theorem will usually be applied through the latter norm equivalence.
\subsection{(Stochastic) maximal regularity in weighted spaces}
let $A$ be a closed
operator on $\mathcal B$ such that $-A$ generates a strongly continuous semigroup
$P=(P(t))_{t\geq 0}$ on $\mathcal B$. We begin this subsection with a host of definitions and results taken from \cite{AgrestiVeraarStability, AgrestiVeraarNonLinParab, SimonettPruess}.

\begin{definition}[$H^\infty$-calculus for sectorial operators]
The operator $A$ is sectorial if the domain and the range of $A$ are dense in $\mathcal B$ and 
\begin{enumerate}
    \item[(i)] there exists $\phi\in (0,\pi)$ such that $\sigma(A)\subset \overline{\Sigma}_{\phi}$, where $ \Sigma_{\phi}\coloneqq\{z\in \mathbb C\,:\,|\arg z|< \phi\}$
    \item[(ii)] there exists $C_{\mathrm{res}}>0$ such that
    \begin{equation}
    \label{eq:sectorial}
    |\lambda|\norm{(\lambda-A)^{-1}}_{\mathcal L(\mathcal B)}\leq C_{\mathrm{res}}, \qquad \forall \lambda \in \mathbb C\setminus \overline{\Sigma_{\phi}}.
    \end{equation}
\end{enumerate}
Then $\omega(A)\coloneqq\inf\{\phi\in (0,\pi)\,:\,\eqref{eq:sectorial}\,\,\text{holds for some}\; C_{\mathrm{res}}>0\}$ is called the angle of sectoriality of $A$. We say that $A$ has a bounded $H^{\infty}$-calculus of angle $\phi \in (0,\pi)$ if
for all holomorphic functions $f:\Sigma_{\phi} \rightarrow  \mathbb C$ with $$\norm{\frac{f}{\min\{|z|^{\varepsilon},|z|^{-\varepsilon}\}}}_{L^\infty} < \infty$$ for some $\varepsilon > 0$, the bound 
\begin{equation}
\label{eq:H_infinite_calculus}
\norm{f(A)}_{\mathcal L(\mathcal B)}\leq \const \norm{f}_{L^{\infty}(\Sigma_{\phi})}
\end{equation} 
holds for some constant $C_{\themycounter}$ independent of $f$. Here, the operator $f(A)$ is defined through the line integral
$$
f(A)\coloneqq \frac{1}{2\pi i} \int_{\Gamma} f(z)(z-A)^{-1}\d z \in \mathcal L(\mathcal B)
$$
for a contour $\Gamma \subset \Sigma_\phi$ in an intermediate sector enclosing the spectrum. Then, $$\omega_{H^\infty}(A) \coloneqq \inf\{\phi\in (0,\pi)\,:\,\eqref{eq:H_infinite_calculus}\text{ holds for some }\const >0\}$$ is the angle of the $H^{\infty}$-calculus of $A$.
\end{definition}
\begin{remark}
It is well-known that $A$ is sectorial of angle $\phi < \frac{\pi}{2}$ if and only if $-A$ generates an analytic semigroup.
\end{remark}

\begin{example}
Importantly, second order uniformly elliptic differential operators with Dirichlet, Neumann and periodic boundary conditions have a bounded $H^\infty$-calculus, see e.g. \cite[Ex. 2.1]{AgrestiVeraarNonLinParab}.
\end{example}
\begin{definition}[Real interpolation space associated with a sectorial operator]
Let $A$ be a sectorial operator of angle $< \frac\pi2$ on $\mathcal B$. For $\theta\in(0,1)$ and $p\in (1,\infty)$, we define
$
D_A(\theta,p)
$ as the real interpolation of $\mathcal B$ and $D(A)$ equipped with the Graph norm, i.e.
$$
D_A(\theta,p)=(\mathcal B,D(A))_{\theta,p}.
$$ 
\end{definition}
\begin{proposition} \label{prop:InterpolationIdentity}
For arbitrary $1\geq \rho > \theta > 0$ and $p \in [1,\infty]$,
$$
\norm{u}_{D_A(\theta,p)}
\simeq
\norm{u}_{\mathcal B}
+
\left(
\int_0^\infty
\left(t^{\rho-\theta}\norm{A^\rho P(t)u}_{\mathcal B}\right)^p
\,\frac{\d t}{t}
\right)^{1/p},
$$
and $$
\norm{u}_{D_A(\theta,p)}
\simeq
\norm{u}_{\mathcal B}
+
\sup_{t \in (0,\infty)} t^{\rho-\theta}\norm{A^\rho P(t)u}_{\mathcal B}
$$ in the case $p = \infty$.
\end{proposition}
\begin{proof}

This characterisation of interpolation spaces for sectorial operators essentially follows from \cite[Example 7.2 and Corollary 7.3]{HaaseFunCalcInterpol}. First, the reiteration property shows that $(\mathcal B,D(A))_{\theta,p} = (\mathcal B,D(A^\rho))_{\frac\theta\rho,p}$. Inserting Example 7.2 therefore gives that \begin{equation*}
D_A(\theta,p) = \Set{u \in \mathcal B; \int_0^\infty
\left(t^{\rho-\theta}\norm{A^\rho P(t)x}_{\mathcal B}\right)^p
\,\frac{\d t}{t} < \infty}. \qedhere
\end{equation*}
\end{proof}
Hereafter, we adopt the notation $$H^{1,p}_\beta([0,T];\mathcal B) \coloneqq H^{1,p}([0,T],w_\beta;\mathcal B)$$ and the corresponding variants for $L^p, W^{k,p}$ and such.
\begin{definition}[Weighted deterministic maximal $L^p$-regularity]
Let $\mathcal B_1 \coloneqq D(A) \subset \mathcal B$. Given $p \in (1,\infty)$ and $\beta \in (-1,p-1)$, we say that $A$ has \textit{weighted deterministic maximal $L^p$-regularity} on $[0,T]$ if for every
$f\in L^p_\beta([0,T];\mathcal B)$ and $u_0\in (\mathcal B,\mathcal B_1)_{1-(1+\beta)/p,p}$
there exists a unique
$$
u\in W^{1,p}_\beta([0,T];\mathcal B)\cap L^p_\beta([0,T];\mathcal B_1)
$$
such that $u(0)=u_0$ and
$$
u'(t)+Au(t)=f(t)\quad \text{for a.e. } t\in(0,T),
$$
and moreover, there is a constant $C_{MR}>0$ independent of $f$ and $u_0$ such that
\begin{equation}
\norm{u}_{W^{1,p}_\beta([0,T];\mathcal B)}+\norm{u}_{L^p_\beta([0,T];\mathcal B_1)}
\leq C_{MR}\left(\norm{f}_{L^p_\beta([0,T];\mathcal B)}+\norm{u_0}_{(\mathcal B,\mathcal B_1)_{1-(1+\beta)/p,p}}\right).
\end{equation}
\end{definition}
\begin{remark} \label{rem:detconvmaxreg}
In particular, if $u_0 = 0$, we conclude that the convolution $$u = P \ast f \coloneqq \int_0^\cdot P(\cdot-s)f(s) \d s$$ satisfies
\begin{equation}
\norm{u}_{_0W^{1,p}_\beta([0,T];\mathcal B)}+\norm{u}_{L^p_\beta([0,T];\mathcal B_1)}
\leq C_{MR}\norm{f}_{L^p_\beta([0,T];\mathcal B)}.
\end{equation}
\end{remark}
We now introduce the stochastic integration framework necessary to apply stochastic maximal regularity. In the following, denote by $L^p_{\mathcal F}(\Omega;\gamma(L^2([0,T];U),\mathcal B))$ the closure of the $\mathcal F_t$-adapted elementary processes in $L^p(\Omega;\gamma(L^2([0,T];U),\mathcal B))$. Here, $\gamma(\cdot, \cdot)$ denotes the space of $\gamma$-radonifying operators. For UMD Banach spaces $\mathcal B$, this is the appropriate space of integrands for which the stochastic integral is well-defined. Namely, for $G \in L^p_{\mathcal F}(\Omega;\gamma(L^2([0,T];U),\mathcal B))$, the norm equivalence
\begin{equation} \label{eq:itoisometrybanach}
\norm{\int_0^\cdot G \d {W}_s}_{L^p(\Omega;C([0,T];\mathcal B))} \simeq \norm{G}_{L^p(\Omega;\gamma(L^2([0,T];U),\mathcal B))}
\end{equation}
holds by the results of \citet[Thm. 4.4]{VanNeervenVeraarStochIntUMD}.
For type $2$ Banach spaces, in particular, there exists the embedding
\begin{equation} \label{eq:stochintembedding}
L^2([0,T],\gamma(U,\mathcal B)) \hookrightarrow \gamma(L^2([0,T];U),\mathcal B),
\end{equation} 
see e.g.
\cite[Cor. 2.5]{AgrestiVeraarStability}. This lets us conclude that 
$$L^p_{\mathcal F}(\Omega\times [0,T];\gamma(U,\mathcal B)) \hookrightarrow L^p_{\mathcal F}(\Omega;\gamma(L^2([0,T];U),\mathcal B)),
$$
where the former denotes the space of $p$-integrable, progressively measurable processes with values in $\gamma(U,\mathcal B)$ \cite[Prop. 2.1]{VeraarWeisStochConv}. It was shown in detail in \cite{VanNeervenVeraarStochIntUMD} that we can localise the stochastic integral to the $\mathcal F_t$-progressively measurable integrands in $L^0(\Omega;L^p( [0,T];\gamma(U,\mathcal B)))$, where $L^0(\Omega,X)$ denotes the space of strongly measurable functions in a Banach space $X$. Now, for
$
G\in L^0_{\mathcal F}(\Omega;L^p( [0,T];\gamma(U,\mathcal B)))
$, $p \geq 2$, the stochastic convolution process is defined by
$$
(P\diamond G)(t)
\coloneqq
\int_0^t P(t-s)G_s\d {W}_s,
\qquad t \geq 0.
$$
In our framework, the following result assures well-definedness and existence of continuous modifications of the stochastic convolution process.
\begin{theorem}[\cite{VeraarWeisStochConv}] \label{thm:maxestimatestochconv}
Let $\mathcal B$ be a UMD space of type 2 and assume that $A$ has a bounded $H^\infty$-calculus of angle $< \frac{\pi}{2}$. Then, for any $G\in L^0_{\mathcal F}(\Omega;L^p( \mathbb R_+;\gamma(U,\mathcal B)))$, $P \diamond G$ has a continuous modification. Moreover, for all $p \in (0,\infty)$, 
\begin{equation} \label{eq:maxestimatestochconvo}
\norm{P \diamond G}_{L^p(\Omega;C(\mathbb R_+;\mathcal B))} \leq \const  \norm{G}_{L^p(\Omega;L^2( \mathbb R_+;\gamma(U,\mathcal B)))},
\end{equation}
where $C_{\themycounter}$ depends on $p$.
\end{theorem}

\begin{corollary}\label{cor:stopequiv}
Let $G_1, G_2$ satisfy the assumptions of Theorem \ref{thm:maxestimatestochconv}. Suppose that for some stopping time $\tau \geq 0$, $G_1 \equiv G_2$ on $[0,\tau]$, $\mathbb P$-a.s. Then
$P \diamond G_1\overset{C([0,\tau];\mathcal B)}{=} P \diamond G_2$, $\mathbb P$-a.s.
\end{corollary}

In the following, the exponential growth bound of the semigroup $P$ is defined by
$$
\omega_0(-A)
\coloneqq
\inf\left\{
\omega\in\mathbb R:
\sup_{t>0} e^{-\omega t}\norm{P(t)}_{\mathcal L(\mathcal B)}<\infty
\right\}.
$$
Note that for any $\omega>\omega_0(-A)$, the operator $\omega+A$ is sectorial \cite{AgrestiVeraarStability}, so its fractional power
$(\omega+A)^{1/2}$ is defined. We write
$
\mathcal B_{1/2}^{(\omega)} \coloneqq D\left((\omega+A)^{1/2}\right),
$ to denote its fractional domain equipped with its graph norm. 
\begin{definition} [Weighted stochastic maximal regularity]
Let $\mathcal B_1 \coloneqq D(A) \subset \mathcal B$. Given $p \in (1,\infty)$ and $\beta \in \mathbb R$, we say that $A$ has \emph{weighted stochastic maximal $L^p$-regularity} on $[0,T]$  if for some
$\omega>\omega_0(-A)$ and any $\theta \in [0,\frac12)$,
the stochastic convolution $P\diamond G$ belongs to
$$
L^p\left(\Omega;{_0H}^{\theta,p}\left([0,T],w_\beta;\mathcal B_{1/2-\theta}^{(\omega)}\right)\right),
$$
and there exists a constant $C_{p,\beta,\theta,T}>0$ such that for any $G \in L^p(\Omega\times [0,T],w_\beta;\gamma(U,\mathcal B))$,
\begin{equation}
\label{eq:finitemaxregfractional}
\norm{P\diamond G}_{L^p(\Omega;{_0H}^{\theta,p}([0,T],w_\beta;\mathcal B_{1/2-\theta}^{(\omega)}))}
\leq
C_{p,\beta,\theta,T}
\norm{G}_{L^p(\Omega\times [0,T],w_\beta;\gamma(U,\mathcal B))}.
\end{equation}
\end{definition}

The finite time-horizon weighted stochastic maximal regularity result that we shall use is the following. One obtains this result from \citet[Thm. 7.16 and Remark 7.13]{AgrestiVeraarStability} after invoking boundedness of the multiplication operator $f \mapsto (t \mapsto e^{-\varepsilon t}f(t))$ on Bessel potential spaces. This lets us reduce the claim to the case of exponentially stable semigroups generated by operators of the form $-\omega - A$.
\begin{theorem}[\cite{AgrestiVeraarStability,SimonettPruess}]
\label{thm:finitemaxregfractional}
Let $\mathcal B$ be isomorphic to a closed subspace of an $L^q$-space for some
$q\in [2,\infty)$. Assume that there exists $\omega>\omega_0(-A)$ such that
$\omega+A$ has a bounded $H^\infty$-calculus of angle $<\frac{\pi}{2}$. Let $T<\infty$ and $p\in (2,\infty)$. Then, for any $\beta\in \left(-1,\frac p2-1\right)$, $A$ has weighted  deterministic and stochastic maximal regularity.
\end{theorem}

These results are particularly useful for fixed-point arguments when combined with suitable embeddings of vector-valued Sobolev spaces.

\begin{lemma}[Endpoint mixed derivative estimate on finite intervals]
\label{lem:endpoint-mixed-derivative}
Let $\mathcal B$ be isomorphic to a closed subspace of an $L^q$-space for some
$q\in [2,\infty)$. Assume that there exists $\omega>\omega_0(-A)$ such that
$\omega+A$ has a bounded $H^\infty$-calculus of angle $<\frac{\pi}{2}$. Let $p\in(1,\infty)$, $\beta\in(-1,p-1)$ and $\vartheta\in(0,1)$. Then the weighted mixed derivative embedding
\begin{equation}\label{eq:real-line-sobolev-interpolation}
{}_0W^{1,p}\left(\R_+,w_\beta;\mathcal B\right)\cap L^p\left(\R_+,w_\beta;D(A)\right)
\hookrightarrow
{}_0H^{1-\vartheta,p}\left(\R_+,w_\beta;\mathcal B^{(\omega)}_{\vartheta}\right)
\end{equation}
is continuous and for every $T\in(0,\infty)$,
\begin{equation}\label{eq:finite-time-sobolev-interpolation}
{}_0W^{1,p}\left([0,T],w_\beta;\mathcal B\right)\cap L^p\left([0,T],w_\beta;D(A)\right)
\hookrightarrow
{}_0H^{1-\vartheta,p}\left([0,T],w_\beta;\mathcal B^{(\omega)}_{\vartheta}\right),
\end{equation}
where the embedding constant can be chosen independently of $T$.
\end{lemma}

\begin{proof}
The embedding \eqref{eq:real-line-sobolev-interpolation} follows from \citet[Subsection 5.5]{PruessSimonettBook} and then the finite time result \eqref{eq:finite-time-sobolev-interpolation} follows immediately from the properties of the extension operator $\mathcal E_T$ detailed in \citet[Prop. 2.5 (1)]{AgrestiVeraarNonLinParab}.
\end{proof}

\begin{theorem}
Let $\mathcal B$ be a UMD Banach space. Let $1<p_0\leq p_1<\infty$, let
$s_0\in(0,1)$, and let $\beta_0\in(-1,p_0-1)$ and $\beta_1\in(-1,p_1-1)$.
Assume that
$$
\frac{\beta_0}{p_0}\geq \frac{\beta_1}{p_1}
\qquad\text{and}\qquad
s_0-\frac{1+\beta_0}{p_0}
\geq
-\frac{1+\beta_1}{p_1}.
$$
Then, for every $T>0$,
\begin{equation}\label{eq:weighted-embedding-zero}
{}_0H^{s_0,p_0}([0,T],w_{\beta_0};\mathcal B)
\hookrightarrow
L^{p_1}([0,T],w_{\beta_1};\mathcal B),
\end{equation}
and the embedding constant can be chosen independently of $T$.
\end{theorem}
\begin{proof}
This essentially follows from the results of \citet[Corollary]{MeyriesVeraarSobEmbedding}, which implies the corresponding embedding on $\R$. Thus
$
H^{s_0,p_0}(\R,w_{\beta_0};\mathcal B)
\hookrightarrow
L^{p_1}(\R,w_{\beta_1};\mathcal B).
$
Application of Theorem \ref{thm:equivalence_h_H} and the extension operator $\mathcal E_T$ yields the claim.
\end{proof}

Finally, denote by $C([0,T];\mathcal B)_0$ the set of all continuous curves $u\colon[0,T] \rightarrow \mathcal B$ endowed with $\sup$-norm, where $\lim_{t \to \infty} u(t) = 0$ if $T=\infty$.

\begin{proposition}[\cite{AgrestiVeraarNonLinParab}]
\label{prop:continuousTrace}
Let $\mathcal B_1\hookrightarrow \mathcal B$ be Banach spaces. Set $\mathcal B_{1-\theta}=[\mathcal B,\mathcal B_1]_{1-\theta}$ or $\mathcal B_{1-\theta} = (\mathcal B,\mathcal B_1)_{1-\theta,r}$, where the latter denotes real interpolation for $r\in [1, \infty]$. Assume that $p\in (1,\infty)$, $\beta\in [0,p-1)$, $\theta\in (0,1)$ and $T\in (0,\infty]$. Then the following holds:
\begin{enumerate}[(i)]
\item\label{item:trace_with_weights_Xap} If $\theta>\frac{1+\beta}{p}$, then
$$
H^{\theta,p}([0,T],w_{\beta};\mathcal B_{1-\theta})\cap L^p([0,T],w_{\beta};\mathcal B_1)
\hookrightarrow C\left([0,T];(\mathcal B,\mathcal B_1)_{1-\frac{1+\beta}{p},p}\right)_0
$$
\item\label{item:trace_without_weights_Xp} If $\theta>\frac{1}{p}$, then for any $0<\varepsilon<T$
$$
H^{\theta,p}([0,T],w_{\beta};\mathcal B_{1-\theta})\cap L^p([0,T],w_{\beta};\mathcal B_1)\hookrightarrow
 C\left([\varepsilon,T];(\mathcal B,\mathcal B_1)_{1-\frac{1}{p},p}\right)_0.
$$
\end{enumerate}
Here, if $T$ ranges over $ (\eta,\infty]$ for some $\eta > 0$, then the constants in \eqref{item:trace_with_weights_Xap} and \eqref{item:trace_without_weights_Xp} depend only on $\eta$.
Furthermore, if we replace $H^{\theta,p}$ by ${_0H}^{\theta,p}$ in \eqref{item:trace_with_weights_Xap} and \eqref{item:trace_without_weights_Xp}, the constants in the embeddings can be chosen independent of $T>0$.
\end{proposition}

\begin{example}
The regularity result for stochastic convolutions and the embedding results laid out in this subsection will be used only for the periodic Laplacian on $\mathbb T^d$. It is known that $-\Delta$ has a bounded $H^\infty$-calculus of angle $<\frac\pi2$ on $L^q(\mathbb T^d)$ for every $q\in(1,\infty)$, see also \citet[Example 4.7.5]{ARENDT20021}. By the Banach space Isomorphism 
$(1-\Delta)^{\frac \sigma 2}\colon H^{\sigma,q} \rightarrow L^q$ (cf. \citet[Section 3.5.4]{ScheisserTriebel}) and Prop. 3.3.14 in \citet{PruessSimonettBook}, we find that the conjugate operator $$(1-\Delta)^{-\frac{\sigma}{2}}(-\Delta)(1-\Delta)^{\frac{\sigma}{2}} = -\Delta$$ has a bounded $H^\infty$-calculus of the same angle also on $H^{\sigma,q}$. Moreover, the definition of periodic Bessel potential spaces by means of Fourier multipliers directly implies that $H^{\sigma,q}$ is isomorphic to a closed subspace of $L^q$ (and, actually, to $L^q$ itself, as a Banach space). Thereby, the results of this section are applicable for $\mathcal B = H^{\sigma,q}$ and $A = -\Delta$.
\end{example}

\subsection{Existence of local mild solutions}

The basic point of the argument is to keep the deterministic and stochastic
components in different spaces. We therefore write $\bm z=S(\cdot)\bm z_0 +\bm u+\bm v$
and solve instead the coupled system
\begin{equation}\label{eq:split-system-power}
\left\{
\begin{aligned}
{\bm u}(t)
&=
S(t){\bm z}_0
+
\int_0^t S(t-r)F(S(r){\bm z}_0+{\bm u}_r+{\bm v}_r)\d r,\\
{\bm v}(t)
&=
\int_0^t S(t-r)\,G(S(r){\bm z}_0+{\bm u}_r+{\bm v}_r)\d {\bm W}_r.
\end{aligned}
\right.
\end{equation}
Here, $S = e^{t\Delta}$ denotes the analytic semigroup generated by the Laplacian $\Delta$ on the torus. In this decomposition, $u$ carries the initial trace, whereas $v$ is the zero-initial-trace
stochastic part. For the purpose of the fixed point argument, we specifically work with the continuous modification $
S \diamond G$ of the stochastic convolution. In the linear estimates we shall apply deterministic and stochastic maximal
regularity to the Laplacian.
Throughout this section, we assume that $q \in (d,2d)$. 

We now introduce Banach spaces of processes akin to those employed in \citet{SimonettPruess, AgrestiVeraarNonLinParab}. We will obtain candidate local solutions by means of a fixed point argument in these spaces. As specified at the beginning of the section, we will assume $q \in(d, 2d)$. Write 
$$
L^p_\beta([0,T];\mathcal B)\coloneqq L^p\left((0,T),t^\beta\d t;\mathcal B\right),
\qquad
H^{\theta,p}_\beta([0,T];\mathcal B)\coloneqq
H^{\theta,p}\left((0,T),t^\beta\d t;\mathcal B\right).
$$ 
to denote vector-valued weighted Sobolev spaces, for a Banach space $\mathcal B$ and a weight $\beta \in \mathbb R$.
\begin{definition}[Coupled weighted solution space] \label{def:solspace} Let the parameters $\delta$, $ \sigma$ be defined as in Assumption \ref{ass:nemytskiiLipschitz} and assume that $p \geq \frac{q}{\delta q-d}$. For every $r\in[\delta,1]$ we define
$$
\alpha_r\coloneqq r -\frac{d}{q} \in \left(0,\frac12\right),\quad 
\beta_r\coloneqq p\alpha_r-1 \in \left(-1, \frac{p}{2}-1\right).
$$
We introduce the deterministic and stochastic weighted maximal regularity weighted spaces
$$
\mathbb X_1(T)\coloneqq{_0H}^{1,p}_{\beta_1}\left([0,T];L^{q/2}(\mathbb T^d)\right)
\cap
L^p_{\beta_1}\left([0,T];H^{2,q/2}(\mathbb T^d)\right).
$$
and the stochastic maximal-regularity space
$$
\mathbb Y_\delta(T)\coloneqq
{_0H}^{\sigma/2,p}_{\beta_\delta}\left([0,T];H^{1,q}(\mathbb T^d)\right)
\cap
L^p_{\beta_\delta}\left([0,T];H^{1+\sigma,q}(\mathbb T^d)\right).
$$
We introduce the Banach spaces $\mathbb K_1$, $\mathbb S_\delta$ of progressively measurable random processes with values in $\mathbb X_1$ resp. $\mathbb Y_\delta$ through their norms
$$
\norm{{\bm u}}_{\mathbb K_1(T)}
\coloneqq
\norm{{\bm u}}_{L^{p}(\Omega;\mathbb X_1(T))}, \quad 
\norm{{\bm v}}_{\mathbb S_\delta(T)}
\coloneqq
\norm{{\bm v}}_{L^p(\Omega;\mathbb Y_\delta(T))}.
$$
The corresponding product space $\mathbb K_1(T) \oplus \mathbb S_\delta(T)$ will be the target space of our fixed-point iteration.
\end{definition}
\begin{remark}
Note that since $\delta > \frac{d}{q}$, $\delta q > d$, so the condition $p \geq \frac{q}{\delta q-d}$ is well-defined. This condition ensures that the weight $\beta_\delta$ is nonnegative, so that the trace embeddings \eqref{item:trace_with_weights_Xap} and \eqref{item:trace_without_weights_Xp} from Prop. \ref{prop:continuousTrace} are applicable. 
\end{remark}
\begin{remark}
Observe that the critical value $\frac{\sigma}{2} = \frac{1+\beta_\delta}{p} \Leftrightarrow \sigma = 2\alpha_\delta$ (cf. Theorems \ref{thm:equivalence_h_H} and \ref{thm:finitemaxregfractional}) can always be avoided for small enough $\sigma$.
\end{remark}
Of further importance is the family of spaces 
\begin{equation} \label{eq:WSpaceDefinition}
\mathbb W_r(T)\coloneqq L^{2p}_{\beta_r}\left([0,T];H^{r,q}(\mathbb T^d)\right)
\end{equation}
for $r \in [\delta,1]$. We will primarily utilise this space to derive estimates on the respective terms in the fixed point argument. Subsequently the Sobolev embeddings \eqref{eq:finite-time-sobolev-interpolation} and \eqref{eq:weighted-embedding-zero} let us close the estimates. We remark here the embedding 
\begin{equation}
\mathbb W_r(T) \hookrightarrow L^{\kappa}([0,T];H^{r,q}(\mathbb T^d))
\end{equation}
for any $r \in [\delta,1]$ and $4 \leq \kappa < \frac{2}{\alpha_r}$.
\begin{lemma} \label{lemma:XYembeddings}
Let $r\in[\delta,1]$ and $p \geq \frac{q}{q-d}$. 
Then there exist constants $C_X, C_Y$ dependent on $r$ but not on $T \in (0,1)$ such that for all  $({\bm u},{\bm v}) \in \mathbb X_1(T) \oplus \mathbb Y_\delta(T)$,
\begin{enumerate}[(i)]
    \item \label{item:Xembedding} $
\norm{{\bm u}}_{\mathbb W_r(T)}
\leq C_X
\norm{{\bm u}}_{\mathbb X_1(T)}$
    \item \label{item:Yembedding} $
\norm{{\bm v}}_{\mathbb W_1(T) \cap \mathbb W_\delta(T)} \leq C_Y \norm{{\bm v}}_{\mathbb Y_\delta(T)}$
\end{enumerate}
\end{lemma}
\begin{proof}
Set $\vartheta_r\coloneqq\frac12\left(r+\frac{d}{q}\right)\in(\frac dq,1).$
We apply the mixed-derivative estimate \eqref{eq:finite-time-sobolev-interpolation} to the
interpolation couple
$\mathcal B_1 = H^{2,q/2}(\mathbb T^d) \subset L^{q/2}(\mathbb T^d) = \mathcal B$, which gives
$$\norm{{\bm u}}_{{_0H}^{1-\vartheta_r,p}_{\beta_1}\left([0,T];H^{2\vartheta_r,q/2}(\mathbb T^d)\right)} \leq \const \norm{{\bm u}}_{\mathbb X_1(T)}.$$
Now
$2\vartheta_r=r+\frac{d}{q}$, 
and the Sobolev embedding at fixed scaling index implies
$$
H^{2\vartheta_r,q/2}(\mathbb T^d)
=
H^{r+d/q,q/2}(\mathbb T^d)
\hookrightarrow
H^{r,q}(\mathbb T^d).
$$
Hence
$$
{\bm u}\in
H^{1-\vartheta_r,p}_{\beta_1}\left([0,T];H^{r,q}(\mathbb T^d)\right).
$$

We next apply the Sobolev embedding \eqref{eq:weighted-embedding-zero}. The balancing condition
$$ 1-\vartheta_r-\frac{1+\beta_1}{p}
=
-\frac{1+\beta_r}{2p}
$$
is exactly
$$
1-\vartheta_r-\alpha_1
=
-\frac{\alpha_r}{2} \iff -\frac r2 + \frac{d}{2q} = -\frac12 \alpha_r,
$$
which is immediate by definition. Meanwhile, the condition $p \geq \frac{q}{q-d}$ ensures that $\frac{\beta_1}{p} \geq \frac{\beta_r}{2p}$, whence
$$
_0H^{1-\vartheta_r,p}_{\beta_1}\left([0,T];H^{r,q}(\mathbb T^d)\right)
\hookrightarrow
L^{2p}_{\beta_r}\left([0,T];H^{r,q}(\mathbb T^d)\right)=\mathbb W_r(T),
$$
which proves \eqref{item:Xembedding}. Note here that $1-\vartheta_r < 1-\frac dq = \alpha_1 = \frac{1+\beta_1}{p}$, so that the sufficient conditions are all satisfied.
It remains to prove \eqref{item:Yembedding}. Since
$\frac{\sigma}{2}>\frac{1+\beta_\delta}{2p}=\frac{\alpha_\delta}{2}
$
by Assumption \ref{ass:nemytskiiLipschitz}, we can again apply the Sobolev embedding theorem and obtain
$$
\mathbb Y_\delta(T) \hookrightarrow {_0}H^{\sigma/2,p}_{\beta_\delta}\left([0,T];H^{1,q}(\mathbb T^d)\right)
\hookrightarrow
L^{2p}_{\beta_\delta}\left([0,T];H^{1,q}(\mathbb T^d)\right).
$$
Since $H^{1,q}(\mathbb T^d)\hookrightarrow H^{\delta,q}(\mathbb T^d)$, this implies
$$
\norm{{\bm v}}_{\mathbb W_\delta(T)} \leq \norm{{\bm v}}_{L^{2p}_{\beta_\delta}([0,T];H^{1,q})}
\leq \const 
\norm{{\bm v}}_{\mathbb Y_\delta(T)}.
$$
Finally, because $\beta_1 \geq \beta_\delta$ and $T\leq 1$,
$$L^{2p}_{\beta_\delta}([0,T];H^{1,q}(\mathbb T^d)) \hookrightarrow L^{2p}_{\beta_1}([0,T];H^{1,q}(\mathbb T^d)) = \mathbb W_1(T)
$$
and therefore $\norm{{\bm v}}_{\mathbb W_1(T)}
\leq \const 
\norm{{\bm v}}_{\mathbb Y_\delta(T)}.$
\end{proof}

To make us of the full power of the maximal regularity framework, we introduce the scale of Besov spaces $B^{s}_{q,p}(\mathbb T^d)$ for $0<s < 1$, $1 < q, p < \infty$ as 
$$
B^s_{q,p}(\mathbb T^d)
\coloneqq
\left(L^q(\mathbb T^d),H^{2,q}(\mathbb T^d)\right)_{s/2,p}
=
D_\Delta\left(\frac{s}{2},p\right),
$$
where $D(\Delta) = H^{2,q}(\mathbb T^d) \subset L^q(\mathbb T^d)$.
Equivalently, 
for any $r \in (s,1]$, one obtains from Prop. \ref{prop:InterpolationIdentity} that
\begin{equation}\label{eq:semigroupcharacbesov}
\norm{{\bm u}}_{B^s_{q,p}(\mathbb T^d)}
\simeq
\norm{{\bm u}}_{L^q(\mathbb T^d)}
+
\left(
\int_0^\infty
\left(
t^{\frac{r-s}2}
\norm{(-\Delta)^\frac r2S(t){\bm u}}_{L^q(\mathbb T^d)}
\right)^p
\frac{\d t}{t}
\right)^{1/p}.
\end{equation}
The following result immediately drops out of the semigroup characterisation \eqref{eq:semigroupcharacbesov}. 
\begin{proposition} \label{prop:heatflowbesov}
Let $1<p<\infty$ and $q>d$. If $
{\bm u}_0\in B^{d/q}_{q,2p}(\mathbb T^d)$,
then $S(\cdot){\bm u}_0 \in (\mathbb W_1\cap \mathbb W_\delta)(T)$ for any finite $T > 0$. In particular, $\norm{S(\cdot){\bm u}_0}_{(\mathbb W_1\cap \mathbb W_\delta)(T)} \overset{T\to 0}{\to} 0$.
\end{proposition}

We now define our notion of solution, which depends on the truncation threshold $\lambda$. To ease notation,  let $C_X, C_Y$ be the constants from Lemma \ref{lemma:XYembeddings} and introduce the expression 
\begin{equation} \label{eq:sumnorm}
\Sigma(t,{\bm z}_0,{\bm u},{\bm v}) \coloneqq \norm{S(\cdot){\bm z}_0}_{(\mathbb W_1\cap \mathbb W_\delta)(t)}+C_X\norm{{\bm u}}_{\mathbb X_1(t)}+ C_Y\norm{{\bm v}}_{\mathbb Y_\delta(t)}.
\end{equation}
Moreover, for a random time $\sigma$, let $\mathcal T_\sigma$ denote the shift-type operator $$(\mathcal T_\sigma f)(t) \coloneqq f(t+\sigma)-S(t)f(\sigma).$$ This operator acts as a shift on convolutions, i.e. $$\mathcal T_\sigma (S\ast f) = S \ast f(\cdot +\sigma).$$ We then define 
\begin{equation} \label{eq:sumnormsigma}
\Sigma_\sigma(t, {\bm z},{\bm u},{\bm v}) \coloneqq \norm{S(\cdot){\bm z}_\sigma}_{(\mathbb W_1\cap \mathbb W_\delta)(t)}+C_X\norm{\mathcal T_\sigma {\bm u}}_{\mathbb X_1(t)}+ C_Y\norm{\mathcal T_\sigma {\bm v}}_{\mathbb Y_\delta(t)},
\end{equation}
provided that this is well-defined.
Heuristically, this measures the growth in the defined solution spaces after restarting at a stopping time $\sigma$.
\begin{definition}[Local solution]\label{def:localsol}
Let ${\bm z}_0$ be an $\mathcal F_0$-measurable random variable with values in $B^{d/q}_{q,2p}(\mathbb T^d)$ for some $p \geq 1$.
A pair $({\bm z},\tau)$ is called a \emph{local mild solution} of
\begin{equation}\label{eq:FullNoiseSPDE}
\d {\bm z}_t = \left(\Delta {\bm z}_t + F({\bm z}_t) \right)\d t + G({\bm z}_t)\d {\bm W}_t
\end{equation}
with initial datum $\bm z_0$
if, for some $T > 0$ and $\lambda > 0$, $({\bm z}_t)_{t \in [0,T]}$ is an $\mathcal F_t$-adapted, $B^{d/q}_{q,2p}(\mathbb T^d)$-valued process and $\tau \leq T$ is a monotone limit of strictly positive $(\mathcal F_t)_{t\geq 0}$-stopping times $(\tau_n)_{n \in \N}$ such that for all $n$,
\begin{enumerate}[(i)]
    \item\label{item:solpathreg} $({\bm z}_{t \wedge \tau_n})_{t \in [0,T]}$ is $\mathcal F_{t}$-adapted with weakly continuous paths
    $${\bm z}_{t \wedge \tau_n}\in C_w([0,T];B^{d/q}_{q,2p}(\mathbb T^d)) \cap L^{2}([0,T];H^{1,q}(\mathbb T^d)) \cap L^4([0,T];H^{\delta,q}(\mathbb T^d)).$$
    \item\label{item:localmildidentity} As an $L^{q/2}$-valued process, \begin{equation}\label{eq:mildlocalsol}
    {\bm z}_{t\wedge\tau_n}
    =
    S(t\wedge\tau_n){\bm z}_0
    +(S\ast F({\bm z}_{\cdot \wedge \tau_n}))(t \wedge \tau_n)
    +(S\diamond G({\bm z}_{\cdot \wedge \tau_n}))(t \wedge \tau_n),
    \end{equation}
    holds for all $t \in [0,T]$, almost surely.
    \item\label{item:sollocalregularity}
    For all finite stopping times $\sigma \leq \tau$, the relative lifetime after $\sigma$ defined by
    \begin{equation} \label{eq:localregularity}
    \tau^\lambda(\sigma) = \inf\bigg\{ 
    t \in [0,(\tau-\sigma)_+]: \Sigma_\sigma\left(t,{\bm z},S\ast F({\bm z}),S\diamond G({\bm z})\right)> \lambda 
    \bigg\} 
    \end{equation} is almost surely strictly positive on the event $\Set{\sigma < \tau}$.
\end{enumerate}
\end{definition} 
In the remaining part of the subsection, we omit reference to the spatial domain of function spaces whenever the appropriate domain is clear from the context.
\begin{remark}
Observe that $\tau^\lambda(\sigma)$ defines an $(\mathcal F^\sigma_t)_{t \geq 0}$-stopping time. Thus, by Lemma \ref{lemma:stopshift}, $\sigma + \tau^\lambda(\sigma)$ defines an $\mathcal F_t$-stopping time.
\end{remark}
\begin{remark}
Equation \eqref{eq:mildlocalsol}  is well-defined. Indeed, the regularity ${\bm z} \in L^2([0,\tau_n];H^{1,q})$ assures that $S \ast F({\bm z}) \in C([0,\tau_n];L^{q/2})$, while ${\bm z} \in L^4([0,\tau_n];H^{\delta,q})$ gives $S \diamond G({\bm z}) \in C([0,\tau_n];L^q)$.
\end{remark}
\begin{remark}
With the same methodology as in this section, one can also solve the equation with initial data in ${\bm z}_0 \in B^1_{d,p}$ and obtain weak continuity of solutions in this space. To this end, one needs to introduce an auxiliary $\delta > 1$ and repeat the arguments of this section with the corresponding weight $\alpha_\delta = \delta - 1$ and suitable local Lipschitz conditions in $H^{\delta,d}$. Then, the blow-up condition is simply $\limsup_{t \to \tau} \norm{{\bm z}_t}_{B^1_{d,2p}} < \infty$. However, we omit this analysis since it is not easily subsumed into the theorems that apply to $q > d$. Moreover, we can apply these theorems since $B^1_{d,2p} \hookrightarrow B^{d/q}_{q,2p}$ for $q > d$.
\end{remark}
\begin{remark} \label{rem:lambdaarbitrary}
If the local regularity property \eqref{eq:localregularity} is satisfied for some $\lambda > 0$, then it is also satisfied for all other $\widetilde \lambda > 0$. If $\widetilde \lambda > \lambda$, then this follows by definition.  If $\widetilde \lambda < 
\lambda$, observe that the $L^p$-norms must vanish as $t \to 0$, while positivity of the stopping time follows a fortiori.
\end{remark}
We reduce the problem to the case of vanishing initial conditions by replacing $F$, $G$ by $$\widetilde F(t,{\bm u}) \coloneqq F(S(t){\bm z}_0+{\bm u}), ~\widetilde G(t,{\bm u}) \coloneqq G(S(t){\bm z}_0+{\bm u}).$$ The goal is then to show that
$$
\begin{cases}
\mathcal K \colon {\bm u} \mapsto S \ast \widetilde F(\cdot,{\bm u}+{\bm v}) \\
\mathcal S \colon {\bm v} \mapsto S\diamond \widetilde G(\cdot,{\bm u}+{\bm v}) 
\end{cases}
$$ 
defines a contraction on $\mathbb K_1(T) \oplus \mathbb S_\delta(T)$ for small $T>0$. However, due to the growth of the nonlinear terms, this cannot be ensured without further regularisation. Following \cite{AgrestiVeraarNonLinParab}, we choose $\lambda > 0$ and introduce a truncation 
$$
\psi_{\lambda}(t,{\bm u},{\bm v}) \coloneqq  \theta_\lambda\left(\Sigma(t,{\bm z}_0,{\bm u},{\bm v})\right)
$$
for $\theta_\lambda \coloneqq \theta(x/\lambda)$ and $\Sigma$ defined as per \eqref{eq:sumnorm}, given the cut-off function $$\theta(x) = \begin{cases}
    1 & x \leq 1 \\
    2-x & 1 < x \leq 2 \\ 0 & x > 2
\end{cases}$$ We note that $\norm{\theta'_\lambda}_{L^\infty([0,\infty))} = \frac{1}{\lambda}$ and that by choice of $C_X$, $C_Y$, it holds that 
\begin{equation} \label{eq:psiconstantchoice}
\begin{aligned}
\Sigma(t,{\bm z}_0,{\bm u},{\bm v}) &= \norm{S(\cdot){\bm z}_0}_{(\mathbb W_1\cap \mathbb W_\delta)(t)}+C_X\norm{{\bm u}}_{\mathbb X_1(t)}+ C_Y\norm{{\bm v}}_{\mathbb Y_\delta(t)} \\
&\geq \norm{S(\cdot){\bm z}_0}_{(\mathbb W_1 \cap \mathbb W_\delta)(t)}+\norm{{\bm u}}_{(\mathbb W_1 \cap \mathbb W_\delta)(t)}+\norm{{\bm v}}_{(\mathbb W_1 \cap \mathbb W_\delta)(t)}.
\end{aligned}
\end{equation}Then, we instead consider the mappings
$$
\begin{cases}
\mathcal K_\lambda \colon {\bm u} \mapsto S \ast(\psi_\lambda(\cdot,{\bm u},{\bm v})\widetilde F(\cdot,{\bm u}+{\bm v})) \\
\mathcal S_\lambda \colon {\bm v} \mapsto S\diamond (\psi_\lambda(\cdot,{\bm u},{\bm v})\widetilde G(\cdot,{\bm u}+{\bm v})) 
\end{cases}
$$ 

\begin{proposition} \label{prop:Fcontraction} 
Let $F$ satisfy Assumption \ref{ass:DriftNonlin}. Then $$\norm{\mathcal K_\lambda({\bm u}_1,{\bm v}_1)-\mathcal K_\lambda({\bm u}_2,{\bm v}_2)}_{\mathbb K_1(T)} \leq \const(\lambda + \lambda^{-1} T^{\alpha_1}) \left(\norm{{\bm u}_1-{\bm u}_2}_{\mathbb K_1(T)}+\norm{{\bm v}_1-{\bm v}_2}_{\mathbb S_\delta(T)}\right)$$
\end{proposition}
\begin{proof}
We apply the deterministic maximal regularity property \ref{rem:detconvmaxreg} and find that 
$$
\begin{aligned}
&\norm{\mathcal K_\lambda({\bm u}_1,{\bm v}_1)-\mathcal K_\lambda({\bm u}_2,{\bm v}_2)}_{\mathbb X_1(T)} \\&\leq \norm{\psi_\lambda(\cdot,{\bm u}_1,{\bm v}_1)\widetilde F(\cdot, {\bm u}_1+{\bm v}_1)-\psi_\lambda(\cdot,{\bm u}_2,{\bm v}_2)\widetilde F(\cdot, {\bm u}_2+{\bm v}_2)}_{L^p_{\beta_1}([0,T];L^{q/2})}.
\end{aligned}
$$
To ease notation, let 
\begin{equation} \label{eq:drifttruncdecompo}
\Delta \widetilde F_\lambda \coloneqq \psi_\lambda(\cdot,{\bm u}_1,{\bm v}_1)\widetilde F(\cdot, {\bm u}_1+{\bm v}_1)-\psi_\lambda(\cdot,{\bm u}_2,{\bm v}_2)\widetilde F(\cdot, {\bm u}_2+{\bm v}_2).
\end{equation} Towards the contraction 
property, we follow the approach of \citet{AgrestiVeraarNonLinParab} and introduce 
the stopping times 
$$\tau^\lambda_i = \inf \Set{t \in (0,T];\Sigma(t,{\bm z}_0,{\bm u}_i,{\bm v}_i) > 2\lambda}, \quad i = 1,2$$ and decompose 
$$\norm{\Delta \widetilde F_\lambda}_{L^p_{\beta_1}([0,T];L^{q/2})} = \mathds{1}_{\tau_1 \leq \tau_2}\norm{\Delta \widetilde F_\lambda}_{L^p_{\beta_1}([0,T];L^{q/2})} + \mathds{1}_{\tau_2 < \tau_1}\norm{\Delta \widetilde F_\lambda}_{L^p_{\beta_1}([0,T];L^{q/2})}
$$
By symmetry, it suffices to demonstrate the estimate for only one of these terms. It is now standard to introduce the decomposition 
$$
\begin{aligned}
\Delta \widetilde F_\lambda & \coloneqq \psi_\lambda(\cdot,{\bm u}_1,{\bm v}_1)\left(\widetilde F(\cdot, {\bm u}_1+{\bm v}_1)-\widetilde F(\cdot, {\bm u}_2+{\bm v}_2)\right) \\ & \quad +  \widetilde F(\cdot, {\bm u}_2+{\bm v}_2)\left(\psi_\lambda(\cdot,{\bm u}_1,{\bm v}_1)-\psi_\lambda(\cdot,{\bm u}_2,{\bm v}_2)\right) \\
&  = I + II, \text{ say.}
\end{aligned}
$$
Then 
$$
\begin{aligned}
\mathds{1}_{\tau_1 \leq \tau_2}\norm{\Delta \widetilde F_\lambda}_{L^p_{\beta_1}([0,T];L^{q/2})} &\leq \mathds{1}_{\tau_1 \leq \tau_2}\norm{I}_{L^p_{\beta_1}([0,T];L^{q/2})}  + \mathds{1}_{\tau_1 \leq \tau_2}\norm{II}_{L^p_{\beta_1}([0,T];L^{q/2})} \\ 
&= \mathds{1}_{\tau_1 \leq \tau_2}\norm{I}_{L^p_{\beta_1}([0,\tau_1];L^{q/2})} + \mathds{1}_{\tau_1 \leq \tau_2}\norm{II}_{L^p_{\beta_1}([0,\tau_2];L^{q/2})}\\
&= I_1 + II_1
\end{aligned}
$$
Observe that for $\tau_1 \leq \tau_2$, $\norm{S(\cdot){\bm z}_0}_{\mathbb W_1(\tau_1)}+C_X\norm{{\bm u}_2}_{\mathbb X_1(\tau_1)}+C_Y\norm{{\bm v}_2}_{\mathbb Y_\delta(\tau_1)} \leq 2 \lambda$. Thus, by the local Lipschitz property of $\widetilde F$ and Hölder's inequality,
$$
\begin{aligned}
 I_1&\leq \const \mathds{1}_{\tau_1 \leq \tau_2}\left(\sum_{i=1,2}\norm{S(\cdot){\bm z}_0}_{\mathbb W_1(\tau_1)}+\norm{{\bm u}_i}_{\mathbb W_1(\tau_1)}+\norm{{\bm v}_i}_{\mathbb W_1(\tau_1)}\right)\norm{{\bm u}_1+{\bm v}_1-{\bm u}_2-{\bm v}_2}_{\mathbb W_1(\tau_1)}\\
& \quad + \const  T^{\alpha_1 } \norm{{\bm u}_1+{\bm v}_1-{\bm u}_2-{\bm v}_2}_{\mathbb W_1(\tau_1)} \\
&\overset{\eqref{eq:psiconstantchoice}}{\leq} \const \left(T^{\alpha_1}+\lambda\right)\left(\norm{{\bm u}_1-{\bm u}_2}_{\mathbb X_1(T)}+\norm{{\bm v}_1-{\bm v}_2}_{\mathbb Y_\delta(T)}\right) \\
\end{aligned}
$$
Since $\norm{\theta'}_{L^\infty} = \frac{1}{\lambda}$, a similar estimate gives that 
$$
\begin{aligned}
II_1 &\leq \frac1\lambda\norm{\widetilde F(\cdot,{\bm u}_2+{\bm v}_2)}_{L^p_{\beta_1}([0,\tau_2];L^{q/2})}\left(\norm{{\bm u}_1-{\bm u}_2}_{\mathbb X_1(T)}+\norm{{\bm v}_1-{\bm v}_2}_{\mathbb Y_\delta(T)}\right) \\
&\leq \frac1\lambda\left(T^{\alpha_1}+4\lambda^2\right)\left(\norm{{\bm u}_1-{\bm u}_2}_{\mathbb X_1(T)}+\norm{{\bm v}_1-{\bm v}_2}_{\mathbb Y_\delta(T)}\right).
\end{aligned}$$
Combining these estimate and taking expectations  yields the desired inequality.
\end{proof}

\begin{proposition} \label{prop:Gcontraction} 
Let $G$ satisfy Assumption \ref{ass:nemytskiiLipschitz}. Then $$\norm{\mathcal S_\lambda({\bm u}_1,{\bm v}_1)-\mathcal S_\lambda({\bm u}_2,{\bm v}_2)}_{\mathbb S_\delta(T)} \leq \const(\lambda + \lambda^{-1} T^{\alpha_\delta}) \left(\norm{{\bm u}_1-{\bm u}_2}_{\mathbb K_1(T)}+\norm{{\bm v}_1-{\bm v}_2}_{\mathbb S_\delta(T)}\right)$$
\end{proposition}

\begin{proof}
We can apply Theorem \ref{thm:finitemaxregfractional} with $D(-\Delta) = H^{\sigma+2,q}$. Then linearity of stochastic convolutions gives us that for $T \leq 1$,
$$
\begin{aligned}
&\norm{\mathcal S_\lambda({\bm u}_1,{\bm v}_1)-\mathcal S_\lambda({\bm u}_2,{\bm v}_2)}_{\mathbb S_\delta(T)} \\
&\leq \const 
\norm{\psi_\lambda(\cdot,{\bm u}_1,{\bm v}_1)\widetilde G(\cdot, {\bm u}_1+{\bm v}_1)-\psi_\lambda(\cdot,{\bm u}_2,{\bm v}_2)\widetilde G(\cdot, {\bm u}_2+{\bm v}_2)}_{L^p_{\beta_\delta}(\Omega\times[0,T];
\gamma(U,H^{\sigma,q}))}.
\end{aligned}
$$
We introduce the stopping times $\tau_i$, $i = 1,2$ and the notation $\Delta \widetilde G_\lambda = I + II$ (cf. \eqref{eq:drifttruncdecompo}). Then again
$$
\begin{aligned}
\mathds{1}_{\tau_1 \leq \tau_2}\norm{\Delta \widetilde G_\lambda}_{L^p_{\beta_\delta}([0,T];
\gamma(U,H^{\sigma,q}))} \leq I_1 + II_1
\end{aligned}
$$
with 
$$
\begin{aligned}
 I_1&\leq \const \mathds{1}_{\tau_1 \leq \tau_2}\left(\sum_{i=1,2}\norm{S(\cdot){\bm z}_0}_{\mathbb W_\delta(\tau_1)}+\norm{{\bm u}_i}_{\mathbb W_\delta(\tau_1)}+\norm{{\bm v}_i}_{\mathbb W_\delta(\tau_1)}\right)\norm{{\bm u}_1+{\bm v}_1-{\bm u}_2-{\bm v}_2}_{\mathbb W_\delta(\tau_1)}\\
& \quad + \const  T^{\alpha_\delta } \norm{{\bm u}_1+{\bm v}_1-{\bm u}_2-{\bm v}_2}_{\mathbb W_\delta(\tau_1)} \\
&\overset{\eqref{eq:psiconstantchoice}}{\leq} \const \left(T^{\alpha_\delta}+\lambda\right)\left(\norm{{\bm u}_1-{\bm u}_2}_{\mathbb X_1(T)}+\norm{{\bm v}_1-{\bm v}_2}_{\mathbb Y_\delta(T)}\right) \\
\end{aligned}
$$
The rest of the proof is analogous as well.
\end{proof}

\begin{theorem}[Local existence for integrable initial conditions]\label{thm:localexistenceL^p} For $p \in \left[\frac{q}{\delta q-d},\infty\right)$, let ${\bm z}_0 \in  L^p(\Omega;B^{ d/q}_{q,2p})$.
Then there exists a local solution $({\bm z},\tau)$ of \eqref{eq:FullNoiseSPDE} in the sense of
Definition~\ref{def:localsol} for some $\lambda > 0$. 
\end{theorem}
\begin{proof}
First, note that because of the regularity $S(\cdot){\bm z}_0 \in \mathbb W_1(T)$ (cf. Prop. \ref{prop:heatflowbesov}), the contraction properties of $\mathcal K_\lambda$ and $\mathcal S_\lambda$ imply that these maps are well-defined as maps into $\mathbb K_1(T) \oplus \mathbb S_\delta(T)$.

Fix $\lambda>0$ and $T>0$ small enough that $(\mathcal K_\lambda, \mathcal S_\lambda)$ is a contraction on $\mathbb K_1(T) \oplus \mathbb S_\delta(T)$. By the Banach fixed point theorem, it has a unique fixed point $({\bm u}^\lambda,{\bm v}^\lambda)$. Set $${\bm z}^\lambda \coloneqq S(\cdot){\bm z}_0 + {\bm u}^\lambda+{\bm v}^\lambda.$$ Then
$${\bm z}^\lambda_t
=
S(t){\bm z}_0
+\int_0^t S(t-s)\,\psi_\lambda(s,{\bm u}^\lambda,{\bm v}^\lambda)\,F({\bm z}^\lambda_s)\d s
+\int_0^t S(t-s)\,\psi_\lambda(s,{\bm u}^\lambda,{\bm v}^\lambda)\,G({\bm z}^\lambda_s)\d {\bm W}_s,
$$
in $L^p(\Omega;\mathbb W_1(T))$. Define the \emph{exit time at level $\lambda$} for the fixed point by
$$\tau^\lambda = \inf \bigg\{t \in (0,T]; \Sigma\left(t,{\bm z}_0,S\ast F({\bm z}^\lambda),S\diamond G({\bm z}^\lambda)\right)
 > \lambda \bigg\}.$$ Note that by construction, $\tau^\lambda>0$ almost surely. Moreover, for every $t<\tau^\lambda$ we have
$
\psi_\lambda(t,{\bm u}^\lambda,{\bm v}^\lambda)
=
1$ and this proves properties \eqref{item:localmildidentity} and  \eqref{item:sollocalregularity} of the candidate solution. 

Towards continuity, the stochastic maximal regularity property \eqref{eq:finitemaxregfractional} gives us that 
$$
S \diamond G \in {_0H}_{\beta_\delta}^{\sigma,p}([0,\tau^\lambda];H^{1-\sigma,q})
$$ 
since ${\bm z}^\lambda \in \mathbb W_\delta(\tau^\lambda)$. We need this additional regularity to apply Prop. \ref{prop:continuousTrace} \eqref{item:trace_with_weights_Xap}, which requires $\sigma > \frac{1+\beta_\delta}{p} = \delta - \frac dq$. Then, for $\mathcal B = H^{\sigma-1}$, $\mathcal B_1 = H^{1+\sigma}$, and deduce that 
$S \diamond G$ is continuous in 
$$
B^{1+\sigma - 2 \alpha_\delta}_{q,p} \hookrightarrow B^{\frac dq+1-\delta}_{q,p} \hookrightarrow B^{\frac dq}_{q,p}.
$$ 
After an embedding into an intermediate space (cf. \eqref{eq:finite-time-sobolev-interpolation}), Prop. \ref{prop:continuousTrace} \eqref{item:trace_with_weights_Xap} also gives 
$$\mathbb X_1(\tau^\lambda) \hookrightarrow {_0H}^{\frac12,p}_{\beta_1}([0,\tau^\lambda];H^{1,q/2}) \cap L^p([0,\tau^\lambda];H^{2,q/2})\hookrightarrow  C([0,\tau^\lambda];B^{2-2\alpha_1}_{q/2,p}) \hookrightarrow C([0,\tau^\lambda];B^{\frac dq}_{q,p})
$$ 
since $B^{2-2\alpha_1}_{q/2,p} = B^{2\frac dq}_{q/2,p} \hookrightarrow B^{\frac dq}_{q,p}$ by the Besov embedding theorems.

By the continuity properties in $B^{d/q}_{q,2p}$, it is now clear that $({\bm z}^\lambda,\tau^\lambda)$ defines a local solution in the sense of Definition \ref{def:localsol}, with $\tau_n \coloneqq \tau^\lambda$. In particular, since ${\bm z}^\lambda$ is $\mathcal F_t$-adapted and continuous, it is progressively measurable and the stopped process is $\mathcal F_{t}$-adapted. 
\end{proof}
\begin{remark} \label{remark:improvedcontinuity}
Note that if ${\bm z}_0$ takes values in $B^{d/q}_{q,r}$ for $p \leq r \leq 2p$, then ${\bm z}^\lambda \in C([0,T];B^{d/q}_{q,r})$. As $S \ast F({\bm z}^\lambda)$, $S \diamond G({\bm z}^\lambda)$ are continuous in $\mathcal B^{d/q}_{q,p}$, any higher regularity ${\bm z}_0 \in B^{d/q}_{q,r}$ for $p \leq r \leq 2p$ translates to higher regularity ${\bm z}^\lambda \in C([0,\tau^\lambda];B^{d/q}_{q,r})$. This is due to the improved continuity of $S(\cdot){\bm z}_0 \in C([0,\tau^\lambda];B^{d/q}_{q,r})$. In principle, we could then formulate our blow-up criteria and maximal solutions also in this space. This recovers a weakly continuous form of the improved continuity properties in \cite{AgrestiVeraarNonLinParab} for improved Besov fine structural indices $2<p \leq r \leq 2p$.
\end{remark}
\begin{remark}
Actually, the fixed-point argument still goes through under the weaker condition $p \geq \frac{q}{q-d}$, since Lemma \ref{lemma:XYembeddings} is still valid. However, in this case, the solution is not necessarily continuous in $B^{d/q}_{q,2p}$, since $\beta_\delta < 0$ and the trace embedding is not applicable at $t = 0$. One may still invoke the positive-time regularisation from \eqref{item:trace_without_weights_Xp}, which yields continuity of the stochastic convolution on $(0,T]$ in $B^{1+\sigma - 2\alpha_1}_{q,p}$. In particular, for $\sigma < \alpha_1$, this space is weaker than $B^{d/q}_{q,2p}$. Therefore, the gluing argument developed below does not directly extend to this weaker regime.
\end{remark}

\begin{remark} \label{rem:LocalLipschitzTraceNorm}
As in \citet{AgrestiVeraarNonLinParab}, the preceding fixed point argument continues to work if the local Lipschitz constants in the nonlinear estimates are allowed to depend locally boundedly on the critical trace norm of the arguments. More precisely, the nonlinearities are allowed to satisfy estimates of the form
$$
\norm{F({\bm z}_1)-F({\bm z}_2)}_{L^{q/2}}
\leq
L_n\norm{{\bm z}_1-{\bm z}_2}_{H^{1,q}},~
\norm{G({\bm z}_1)-G({\bm z}_2)}_{\gamma(U,H^{\sigma,q})}
\leq
L_n\norm{{\bm z}_1-{\bm z}_2}_{H^{\delta,q}},
$$
where $L_n$ and $L_n$ are uniformly bounded for $\norm{{\bm z}_1}_{B^{d/q}_{q,p}}, \norm{{\bm z}_2}_{B^{d/q}_{q,p}} \leq n$.
Observe that these constants remain harmless on the truncated region determined by $\Sigma$. Indeed, writing
$$
{\bm z}_i=S(\cdot){\bm z}_0+{\bm u}_i+{\bm v}_i,
$$
the weighted trace embeddings imply the bounds
$$
\sup_{0\leq s\leq t}\norm{{\bm u}_i(s)}_{B^{d/q}_{q,p}}
\lesssim
\norm{{\bm u}_i}_{\mathbb X_1(t)}, ~
\sup_{0\leq s\leq t}\norm{{\bm v}_i(s)}_{B^{d/q}_{q,p}}
\lesssim
\norm{{\bm v}_i}_{\mathbb Y_\delta(t)}.
$$
Hence, on every interval $[0,t]$,
$$
\sup_{0\leq s\leq t}\norm{{\bm z}_i(s)}_{B^{d/q}_{q,p}}
\lesssim
\norm{{\bm z}_0}_{B^{d/q}_{q,p}}
+
\norm{{\bm u}_i}_{\mathbb X_1(t)}
+
\norm{{\bm v}_i}_{\mathbb Y_\delta(t)}.
$$
Observe that on the set where
$
\Sigma(t,{\bm z}_0,{\bm u}_i,{\bm v}_i)\leq 2\lambda,
$
the critical trace norm of ${\bm z}_i$ is bounded by a constant depending only on $\lambda$ and the dependence of the local Lipschitz constants on the trace norm does not alter the structure of the argument. This is analogous to \citet{AgrestiVeraarNonLinParab}, where the nonlinearities can be locally Lipschitz with respect to the trace-space norm; cf. \cite{VeraarSurvey}.
\end{remark}

\begin{proposition}[Localisation] \label{prop:localisation}
Let $\mathcal F_0$-measurable initial data ${\bm z}^1_0, {\bm z}^2_0 \in L^p(\Omega;B^{d/q}_{q,2p})$ be given. Then, on $\Set{{\bm z}^1_0 \equiv {\bm z}^2_0}$, the respective local solutions given by Thm. \ref{thm:localexistenceL^p} coincide. In particular, by localisation, we obtain a local solution of equation \ref{eq:FullNoiseSPDE} for any $\mathcal F_0$-measurable ${\bm z}_0 \in L^0(\Omega;B^{d/q}_{q,2p})$.
\end{proposition}\begin{proof}
Fix $\lambda>0$ and $T>0$ as in Proposition~\ref{thm:localexistenceL^p} and let 
${\bm z}^{\lambda}_i$, $i = 1,2$ denote the respective solutions of the truncated equation.

We now show the localisation property. Put $A \coloneqq  \Set{{\bm z}^1_0 \equiv {\bm z}^2_0}\in\mathcal F_0$.
If $\mathbb P(A)=0$, the claim is vacuous, so assume $\mathbb P(A)>0$ and consider the conditional probability space
\begin{equation} \label{eq:conditionalprob}
\left(\Omega_A,\mathcal F^A,\left(\mathcal F_t^A\right)_{t \geq 0},\mathbb P^A\right)
\coloneqq
\left(A,\left\{F\cap A:\,F\in\mathcal F\},\left(\{F\cap A:\,F\in\mathcal F_t\right\}\right)_{t\geq0},\mathbb P(\,\cdot\,|A)\right).
\end{equation}
Since $A\in\mathcal F_0$, the process $\bm W$ remains a cylindrical Wiener process on this conditional space and in this space, it holds that ${\bm z}^1_0 = {\bm z}^2_0$.

On $(\Omega_A,\mathbb P^A)$ both processes ${\bm z}^{\lambda}_1$ and ${\bm z}^{\lambda}_2$ satisfy the same truncated mild equation
with the same initial datum. By abuse of notation, denote by $\psi_\lambda(t,{\bm z})$ the previously introduced cut-off function. Then, for all $t\in[0,T]$,
$$
{\bm z}_t
=
S(t) {\bm z}_0
+\int_0^t S(t-s)\psi_\lambda(s,{\bm z})F({\bm z}_s)\d s
+\int_0^t S(t-s)\psi_\lambda(s,{\bm z})G({\bm z}_s)\d {\bm W}_s,
\quad \mathbb P^A\text{-a.s.}
$$ by locality of the stochastic integral.
Since Proposition \ref{thm:localexistenceL^p} applies on the conditional space as well, we obtain
uniqueness of the fixed point. Therefore,
$
{\bm z}^{\lambda}_1 \equiv {\bm z}^{\lambda}_2
$ are indistinguishable processes w.r.t. $\mathbb P^A$.

As a consequence, localisation of solutions of the truncated equation for arbitrary measurable initial conditions is well-defined. By construction, we obtain a stopping time $\tau^\lambda(0)$ on which this defines a local solution of \eqref{eq:FullNoiseSPDE}.
\end{proof}

\subsection{Uniqueness and maximality of local mild solutions}

By our fixed point argument, the solution of the truncated equation is unique. However, proving uniqueness of local solutions requires much of the machinery developed at the beginning of the section. We cannot reduce uniqueness of local solutions to uniqueness of global solutions of the truncated equation, since it is not clear how to canonically extend a given local solution to a global solution of the truncated equation.
\begin{lemma} \label{lemma:CriticalUniqueness}
Any two local solutions $({\bm z}^i, \tau_i)$, $i = 1,2$ with equal initial conditions must be identical on $[0,\tau_1 \wedge \tau_2)$.
\end{lemma}
\begin{proof} Let two such solutions be given and set $\tau \coloneqq \tau_1 \wedge \tau_2$. Let $\sigma = \inf\{t \in [0, \tau]: {\bm z}^1 \neq {\bm z}^2\}$, which is well-defined due to weak continuity. By Corollary \ref{cor:stopequiv} and Corollary \ref{cor:StochConvStopDecomp}, a localisation shows that for $t < \tau$,
$$\begin{aligned}
\int_0^t S(t-s) G({\bm z}^i_s) \d {\bm W}_s &= S((t-\sigma)_+)\int_0^{\sigma \wedge t} S(\sigma \wedge t-s) G({\bm z}^i_s)\d {\bm W}_s \\ &\quad + \int_{t \wedge \sigma}^{t} S(t-s) G({\bm z}^i_s) \d {\bm W}_s  
\\ & =  S((t-\sigma)_+)\int_0^{\sigma \wedge t} S(\sigma \wedge t-s) G({\bm z}^i_s)\d {\bm W}_s \\ & \quad + \int_0^{(t-\sigma)_+} S((t-\sigma)_+-s) G({\bm z}^i_{\sigma+s}) \d  \widetilde {\bm W}_s,
\end{aligned}
$$
where we work with continuous versions of the respective processes and 
$
\widetilde {\bm W}_t \coloneqq  {\bm W}_{\sigma+t}-{\bm W}_\sigma,
$
denotes the shifted Wiener process. In particular, we can derive the identity 
$$
{\bm z}^i_{\sigma + t} = S(t){\bm z}_\sigma + \int_0^t S(t-s)F({\bm z}^i_{\sigma+s}) \d s + \int_0^{t} S(t-s) G({\bm z}^i_{\sigma+s}) \d  \widetilde {\bm W}_s
$$ 
for all $t \in [0,\tau^\lambda(\sigma))$. On the event $A_\sigma = \{\sigma < \tau\}$, we know that ${\bm z}^1(\sigma) = {\bm z}^2(\sigma) \in B^{d/q}_{q,2p}$ and
$$\Sigma_\sigma(\tau^\lambda(\sigma),{\bm z}^i,S \ast F({\bm z}^i),S \diamond G({\bm z}^i)) \leq \lambda$$
for some positive stopping time $\tau^\lambda(\sigma) \coloneqq \tau^{\lambda}_1(\sigma) \wedge \tau^{\lambda}_2(\sigma) \wedge (\tau-\sigma)_+$, where $\tau^\lambda_i(\sigma)$ are the respective stopping times given by \eqref{item:sollocalregularity}. In particular, this implies that  the shifted processes satisfy the integrabilities
$$\norm{{\bm z}^i_{\sigma+\cdot}}_{ (\mathbb W_1 \cap \mathbb W_\delta)(\tau^\lambda(\sigma))} \leq \lambda.$$ 
If it were the case that $\mathbb P(\sigma < \tau) > 0$, we would now run into a contradiction. Consider the Wiener process $\widetilde {\bm W}$ on the probability space 
$\Omega_\sigma \coloneqq \Omega_{A_\sigma}$ (cf. \eqref{eq:conditionalprob}). Recall that $A_\sigma = \{\sigma < \tau\}$ and observe that $A_\sigma$ is $\mathcal F_\sigma$-measurable, since for arbitrary $s \geq 0$
$$
\begin{aligned}
\Set{\sigma \leq s} \cap \Set{\sigma < \tau} 
&= \left(\Set{\sigma < \tau \leq s}\right) \cup \left(\Set{\sigma \leq s} \cap \Set{s < \tau}\right) 
\\ &= \left(\bigcup_{q \in \mathbb Q \cap [0,s]} \Set{\sigma \leq q} \cap \Set{q < \tau} \cap \Set{\tau \leq s}\right) \cup \left(\Set{\sigma \leq s} \cap \Set{s < \tau}\right) \in \mathcal F_s.
\end{aligned}
$$
Introduce the auxiliary $\mathcal F_{\sigma+t}$-stopping time $\widetilde \tau^\lambda(\sigma) = \tau^\lambda(\sigma)\wedge T_0$ for some $T_0 \ll 1$. Repeating the fixed point argument (cf. Thm. \ref{thm:localexistenceL^p}) on the interval $[0,\widetilde \tau^\lambda]$ for $\lambda,T_0 \ll 1$ small enough (recall that $\lambda$ can be chosen arbitrarily small), we rerun the contraction argument to conclude that 
$$\mathbb E_\sigma \left[\norm{{\bm z}^1_{\sigma+\cdot}-{\bm z}^2_{\sigma+\cdot}}_{(\mathbb W_1\cap \mathbb W_\delta)(\widetilde \tau^\lambda)}^p\right] \leq \const \mathbb E_\sigma \left[\norm{{\bm z}^1_{\sigma+\cdot}-{\bm z}^2_{\sigma+\cdot}}_{(\mathbb W_1\cap \mathbb W_\delta)(\widetilde \tau^\lambda)}^p\right] 
$$ with $C_{\themycounter} = \const (\lambda + \lambda^{-1}T^{\alpha_\delta}_0)< 1$. Therefore, $\mathbb E_\sigma \left[\norm{{\bm z}^1_{\sigma+\cdot}-{\bm z}^2_{\sigma+\cdot}}^p_{(\mathbb W_1\cap \mathbb W_\delta)(\widetilde \tau^\lambda)}\right]  = 0$. Here, in particular, we used the maximal regularity property on the stopped interval $[0,\widetilde \tau^\lambda]$, which can be reduced to the deterministic case for any bounded stopping time $\sigma \leq T$ by the simple estimate $$\norm{S\diamond G}_{H^{\theta,p}([0,\sigma];\mathcal B_{1/2-\theta})} \leq \norm{S\diamond \mathds{1}_{[0,\sigma]}G}_{H^{\theta,p}([0,T];\mathcal B_{1/2-\theta})} \leq \norm{G}_{L^p([0,\sigma],\gamma(U,\mathcal B))} $$ By the continuity properties of ${\bm z}^i$, $i=1,2$, it must hold that both processes are equal on the extended interval. This contradicts the definition of $A_\sigma$. We can thus conclude that $\sigma = \tau$.
\end{proof}

The next result is essentially a corollary of Thm. \ref{thm:maxestimatestochconv}, since the condition on $u$ implies $L^2$-regularity in time of $G({\bm u})$.
\begin{proposition}
Let $\tau \leq T$ be a bounded $\mathcal{F}_t$-adapted stopping time
and $({\bm u}_t)_{t \in [0,T]} \subset L^q$ be a $\mathcal F_{t}$-progressively measurable process. Consider the event 
$$E_\tau = \left\{{\bm u} \in  L^{4}([0,\tau];H^{\delta,q})\right\}$$ and suppose that $\mathbb P(E_\tau) = 1$. Then there exists a continuous modification of 
$$
S \diamond G({\bm u}) \in C([0,\tau];L^q).
$$
\end{proposition}

\begin{corollary} \label{cor:well-defined}
By localisation, we can infer that even if $\mathbb P(E_\tau) < 1$, then
$$
S \diamond G({\bm u}) \in C([0,\tau];L^q).
$$
is well-defined on the event $E_\tau$.
\end{corollary}
In the following, the functional \begin{equation} \label{eq:blowupfunctional}
\Xi_\sigma(t,{\bm z}) \coloneqq  \norm{{\bm z}}_{L^\infty([\sigma,t];B^{d/q}_{q,2p})}+ \norm{{\bm z}}_{L^{2}([\sigma,t];H^{1,q})} + \norm{{\bm z}}_{L^4([\sigma,t];H^{\delta,q})}
\end{equation} measures blow-up of solutions.
\begin{remark}\label{rem:simplifiedblowup}
For $\delta < \frac{1}{2}+\frac{d}{2q}$, the usual Gagliardo-Nirenberg inequality for Sobolev spaces implies $$\mathfrak X \coloneqq L^\infty([a,b];B^{d/q}_{q,2p}) \cap L^2([a,b];H^{1,q}) \hookrightarrow L^4([a,b];H^{\delta,q})$$ and thus blow-up only depends on the $\mathfrak X$-norm.
\end{remark}
\begin{proposition}
Let $({\bm z}, \tau)$ be a given local solution according to Definition \ref{def:localsol}. Then, for any stopping time $\sigma \leq \tau$ there exists a stopping time $\widetilde \tau \geq \tau$ and a local solution $\widetilde {\bm z} \in C_w([0,\widetilde \tau];B^{d/q}_{q,2p})$ such that on the event 
$$
E_\sigma \coloneqq \Set{\sigma < \tau < T, ~\lim_{t \to \tau}\Xi_\sigma(t,{\bm z}) < \infty}
$$ it almost surely holds that
$\widetilde \tau > \tau$ and $\widetilde {\bm z} \equiv z$ on $[0,\tau)$.
\end{proposition}
\begin{proof}
Observe that we can reduce the claim to the case $\sigma = 0$. Namely, on 
$$
E_\sigma  = \{\sigma < \tau,~{\bm z} \in L^\infty([\sigma,\tau];B^{d/q}_{q,2p}) \cap  L^{2}([\sigma,\tau];H^{1,q}) \cap L^4([\sigma,\tau];H^{\delta,q})\},$$  property \eqref{item:solpathreg} of local solutions implies that $${\bm z} \in L^\infty([0,\tau];B^{d/q}_{q,2p}) \cap  L^{2}([0,\tau];H^{1,q}) \cap L^4([0,\tau];H^{\delta,q}).$$
We now argue that  the limit ${\bm z}_\tau \coloneqq L^{q/2}-\lim_{t \nearrow \tau}{\bm z}_t$ exists in $L^{q/2}$ and 
\begin{equation}\label{eq:terminalvalue}
{\bm z}_\tau = S(\tau){\bm z}_0 + \int_0^\tau S(\tau-s)F({\bm z}_s)\d s + \int_0^\tau S(\tau-s)G({\bm z}_s) \d {\bm W}_s,
\end{equation}
with all integrals well-defined.  Well-definedness of the stochastic convolution is a consequence of Corollary \ref{cor:well-defined}, where in particular we make use of Corollary \ref{cor:stopequiv} to justify the well-definedness of the stochastic integral on the stopped interval. 
Meanwhile, it is standard to show that for $F \in L^1([0,T];L^{q/2})$, $S \ast F$ is continuous in $L^{q/2}$. By a localisation with respect to $\tau_n$, the definition of a local solution shows that for $t < \tau$, 
$${\bm z}_t = S(t){\bm z}_0 + \int_0^t S(t-s) F({\bm z}_s) \d s + \int_0^t S(t-s) G({\bm z}_s) \d {\bm W}_s.
$$
Note that the right hand side is in $C([0,\tau];L^{q/2})$ and that therefore, ${\bm z}_{\tau_n} \rightarrow {\bm z}_{\tau}$. Hence, we
derive existence of the desired limit expression. Then the weak 
$B^{d/q}_{q,2p}$-continuity of ${\bm z}$ on $[0,\tau)$ and reflexivity of the Besov space imply that ${\bm z} \in C_w([0,\tau];B^{d/q}_{q,2p})$ with $\mathcal F_\tau$-measurable terminal value $\mathds{1}_{E_\sigma}{\bm z}_\tau \in B^{d/q}_{q,2p}$. Measurability in particular follows since the stopped processes $t \mapsto \mathds1_{t \leq \tau_n}{\bm z}_{t \wedge \tau_n}$ are adapted and càglad, hence progressively measurable. Thus, their pointwise limit $\mathds1_{t < \tau}{\bm z}_{t \wedge \tau}$ must be progressively measurable as well. 

Define the shifted process
$
\widetilde {\bm W}_t\coloneqq {\bm W}_{\tau+t}-{\bm W}_\tau,
$
and note that by the strong Markov 
property, this defines a Wiener process
w.r.t. $(\mathcal F^\tau_t)_{t\geq 0}$. The same fixed point 
argument as before 
produces a local solution $(\bar {\bm z},\bar \tau^\lambda)$ with initial condition 
$\bar {\bm z}_0 \coloneqq \mathds 1_{E_\sigma} {\bm z}_\tau$ driven by $\widetilde {\bm W}$. Note that 
$$\bar \tau^\lambda = \inf \left\{t \geq 0:  \Sigma\left(t,\bar {\bm z}_0,S\ast F(\bar {\bm z}),S\diamond G(\bar {\bm z})\right)>\lambda\right\} \wedge T_0
$$
for some $T_0 \ll 1$ and in particular, that $\bar {\bm z} \in C([0,\bar \tau^\lambda];B^{d/q}_{q,2p})$.
Define the extension 
$$\widetilde {\bm z}_t \coloneqq \mathds{1}_{t<\tau} {\bm z}_t + \mathds{1}_{t \geq \tau } \mathds{1}_{(t-\tau)_+ \leq \bar \tau^\lambda}\mathds{1}_{E_\sigma}\bar {\bm z}_{(t-\tau)_+}
$$ and observe that it is left-continuous and adapted (cf. Lemma \ref{lemma:stopshift}). 
Moreover, since $E_\sigma$ is $\mathcal F_\tau$-measurable, $\mathds1_{E_\sigma} \bar \tau^\lambda$ is again a stopping time w.r.t. $(\mathcal F^\tau_t)_{t \geq 0}$ and thus $\widetilde \tau \coloneqq \tau + \mathds1_{E_\sigma} \bar \tau^\lambda$ is an $\mathcal F_t$-stopping time (cf. \eqref{item:shift}).

We construct an approximating sequence of $\widetilde \tau$ and verify the local regularity property \eqref{eq:localregularity}. Let 
$$\widetilde \tau_k \coloneqq \inf \left\{t \in [0,\tau] : \Xi_\sigma(t,{\bm z}) > k \right\} \wedge (\tau + \mathds{1}_{E_\sigma} \bar \tau^\lambda)
$$ By continuity of ${\bm z}$ and the definition of $E_\sigma$, it 
must be that $\widetilde \tau_k \to \tau$ on $E^c_\sigma$. On the other hand, on $E_\sigma$, we conclude $\widetilde \tau_k = \widetilde \tau$ for large enough $k$ since the extension is continuous on 
$[\tau, \widetilde \tau]$. To verify the local regularity, choose any stopping time 
$\widetilde \sigma \leq \widetilde \tau$ and note that on $\Set{\widetilde \sigma < \widetilde \tau}$, 
$$\widetilde \tau^\lambda(\widetilde \sigma) \geq \mathds{1}_{\widetilde \sigma < \tau}\tau^\lambda(\widetilde \sigma \wedge \tau) + \mathds{1}_{\widetilde \sigma \geq \tau} (\widetilde \tau-\widetilde \sigma) > 0,$$ where $\widetilde\tau^\lambda(\widetilde \sigma)$ is the relative $\Sigma$-lifetime past $\widetilde \sigma$ of $\widetilde {\bm z}$, and $\tau^\lambda$ is the corresponding lifetime of the original process ${\bm z}$. It is left to verify that this 
process satisfies \eqref{eq:mildlocalsol}. On $\Set{\tau >\widetilde \tau_k \wedge t}$, by definition
$
\widetilde {\bm z}_{\widetilde \tau_k \wedge t} = {\bm z}_{\widetilde \tau_k \wedge t}$ and localisation yields that this continuous process satisfies the semigroup identity in $L^{q/2}$. Meanwhile, on $\Set{\tau \leq \widetilde \tau_k \wedge t \leq \widetilde \tau}$, it holds that 
$\widetilde {\bm z}_{\widetilde \tau_k \wedge t} = \bar {\bm z}_{t \wedge \widetilde \tau_k-\tau}$ and so
$$ 
\begin{aligned}
\widetilde {\bm z}_{t \wedge \widetilde \tau_k} &= S(t \wedge \widetilde \tau_k-\tau){\bm z}_\tau + \int_0^{t \wedge \widetilde \tau_k-\tau}S(t \wedge \widetilde \tau_k-\tau-s)F(\bar {\bm z}_s) \d s \\ &\qquad+\int_0^{t \wedge \widetilde \tau_k-\tau}S(t \wedge \widetilde \tau_k-\tau-s)G(\bar {\bm z}_s) \d \widetilde {\bm W}_s \\
&= S(t \wedge \widetilde \tau_k){\bm z}_0 + \int_0^\tau S(t\wedge \widetilde \tau_k-s) F({\bm z}_s) \d s + \int_{\tau}^{t \wedge \widetilde \tau_k}S(t\wedge \widetilde \tau_k-s) F(\bar {\bm z}_{s-\tau}) \d s \\
& \quad + S(t\wedge \widetilde \tau_k - \tau) \int_0^\tau S(\tau-s) G({\bm z}_s) \d {\bm W}_s + \int_0^{t \wedge \widetilde \tau_k-\tau}S(t \wedge \widetilde \tau_k-\tau-s)G(\bar {\bm z}_s) \d \widetilde {\bm W}_s \\
&= S(t \wedge \widetilde \tau_k){\bm z}_0 + \int_{0}^{t \wedge \widetilde \tau_k}S(t\wedge \widetilde \tau_k-s) F(\widetilde {\bm z}_s) \d s + \int_{0}^{t \wedge \widetilde \tau_k}S(t\wedge \widetilde \tau_k-s) G(\widetilde {\bm z}_s) \d {\bm W}_s
\end{aligned}
$$ 
yields the localised equation
\eqref{eq:mildlocalsol} on $[\tau,\widetilde \tau]$ and thus, we can glue the processes. To obtain the last equation, we implicitly localised Corollary \ref{cor:StochConvStopDecomp} to the event $E_\sigma$, which poses no problem since $E_\sigma \in \mathcal F_\tau$. 
\end{proof}

\begin{theorem}[Maximal local solution and blow-up criterion] \label{thm:maximallocalsolution}
Assume the conditions of Theorem \ref{thm:localexistence}. Then there exists a unique maximal local solution $(z,\tau_{\max})$ of \eqref{eq:FullNoiseSPDE} in the following sense:
\begin{enumerate}
\item (Maximality) If $(\widetilde {\bm z},\widetilde\tau)$ is any other local solution, then $\widetilde\tau\leq \tau_{\max}$ a.s. and
$\widetilde {\bm z}_t={\bm z}_t$ for all $t<\widetilde\tau$ a.s.
\item (Blow-up criterion) Let $\Xi_\sigma(t,{\bm z})$ be defined as in \eqref{eq:blowupfunctional}. Then, for any stopping time $\sigma \leq \tau_{\mathrm{max}}$, 
\begin{equation}
\limsup_{t\nearrow \tau_{\max}} \Xi_\sigma(t,{\bm z}) =\infty.
\end{equation}
almost surely on the event $\{\sigma < \tau_{\max}<T\}$,
\end{enumerate}
\end{theorem}
\begin{proof}
Fix ${\bm z}_0$ and let $\mathcal S({\bm z}_0)$ be the set of all local solutions $({\bm z},\theta)$ with initial value ${\bm z}_0$ in the sense of Definition~\ref{def:localsol}. Let
$$
\tau \coloneqq \esssup_{(z,\theta)\in\mathcal S({\bm z}_0)} \theta.
$$
Then $\tau$ is an $(\mathcal F_t)$-stopping time and by the properties of the essential supremum, we can choose an increasing sequence of stopping times $(\tau_n)_{n \geq 1}$ with $\tau_n \overset{n \to 
\infty}{\to} \tau$, $\mathbb P$-a.s. The verification of the continuity properties and semigroup identities satisfied by the gluing of the corresponding local solutions ${\bm z}^{(n)}$ is done analogously as before.

Now let $\sigma\leq \tau$ be any stopping time. We verify the local regularity property \eqref{eq:localregularity}. On $\{\sigma<\tau\}$ we have $\sigma<\tau_n$ for some $n$ since $\tau_n\rightarrow \tau$. At the same time, ${\bm z} \equiv {\bm z}^{(n)}$ on $[\sigma,\tau_n]$, hence for $t\in[0,(\tau_n-\sigma)_+]$,
$$
\Sigma_\sigma(t, {\bm z},S\ast F({\bm z}), S\diamond G({\bm z}))=\Sigma_\sigma(t,{\bm z}^{(n)},S\ast F({\bm z}^{(n)}), S\diamond G({\bm z}^{(n)}))
$$
Therefore, on $\{\sigma<\tau_n\}$,  the fact that $({\bm z}^{(n)},\tau_n)$ satisfies the exit-time condition implies that
$$
\begin{aligned}
\tau^\lambda(\sigma)
&=\inf\{t\leq \tau-\sigma:\ \Sigma_\sigma(t, {\bm z},S\ast F({\bm z}), S\diamond G({\bm z}))>\lambda\} \\
&\geq 
\inf\{t\leq \tau_n-\sigma:\ \Sigma_\sigma(t,{\bm z}^{(n)},S\ast F({\bm z}^{(n)}), S\diamond G({\bm z}^{(n)})) >\lambda\} = \tau^\lambda_{n}(\sigma) 
>
0.
\end{aligned}
$$
Since $\{\sigma<\tau\}=\bigcup_{n\geq 1}\{\sigma<\tau_n\}$ a.s., this yields $\tau^\lambda(\sigma)>0$ on $\{\sigma<\tau\}$. Hence $({\bm z},\tau)$ is a local solution in the sense of Definition~\ref{def:localsol}.

To see maximality, let $(\bar {\bm z},\bar\tau)\in\mathcal S({\bm z}_0)$ be any local solution. Then $\bar\tau\leq\tau$ a.s. by definition of $\tau$ as an essential supremum. On $[0,\bar\tau]$ both $\bar{\bm z}$ and ${\bm z}$ solve the same equation with the same initial condition, so pathwise uniqueness implies ${\bm z}_t=\bar {\bm z}_t$ for all $t< \bar\tau$ a.s. Thus $({\bm z},\tau)$ dominates every other local solution and is maximal.

Now, the blowup property is easy to see, as finitude of the stopping time on a set of positive probability would imply that we can extend the maximal solution on this set, thus contradicting maximality.
\end{proof}

\begin{proposition}
It almost surely holds that for all ${\bm u} \in H^{1,2} \cap L^{\frac{q}{q-2}}$ and $t \in [0,T]$ and $n \in \mathbb N$
\begin{equation} \label{eq:variationalidentity}
\langle {\bm z}_{t \wedge \tau_n}, \bm u \rangle = \langle {\bm z}_0, \bm u\rangle + \int_0^{t \wedge \tau_n}-\langle \nabla {\bm z}_{s},\nabla {\bm u} \rangle+\langle F({\bm z}_s), \bm u \rangle \d s + \langle \bm u, \int_0^{t \wedge \tau_n}G({\bm z}_s) \d {\bm W}_s\rangle.
\end{equation}
\end{proposition}
\begin{proof}
Let $u \in C^\infty(\mathbb T^d)$ and $t > 0$ be fixed. Then 
\begin{equation} \label{eq:prelimvariational}
\begin{aligned}
\int_0^{t \wedge \tau_n} \langle {\bm z}_{s}, \Delta \bm u \rangle \d s &= \int_0^{t \wedge \tau_n} \bigg( \langle S(s){\bm z}_0, \Delta \bm u\rangle+\int_0^{s} \langle S(s-r)F({\bm z}_{r}), \Delta \bm u \rangle \d r \\ &\qquad + \int_0^{s} \langle \Delta \bm u, S(s-r) G({\bm z}_r)\d {\bm W}_r \rangle \bigg) \d s.
\end{aligned}
\end{equation}
Since $\Delta$ generates an analytic semigroup on $L^2$, it is well-known that $\Delta \int_0^t S(s) \d s = S(t)-I$, where $I$ denotes the identity operator on $L^2$. Thus $$\int_0^{t \wedge \tau_n} \langle S(s){\bm z}_0, \Delta \bm u\rangle \d s = \langle \Delta \int_0^{t \wedge \tau_n} S(s) \d s\, {\bm z}_0, \bm u \rangle = \langle S(t \wedge \tau_n){\bm z}_0-{\bm z}_0,\bm u\rangle.$$
Now, we apply the deterministic and stochastic Fubini formulas to obtain the desired identity.
$$
\begin{aligned}
\int_0^{t \wedge \tau_n}\int_0^{s} \langle S(s-r)F({\bm z}_{r}), \Delta \bm u \rangle \d r\d s 
&= \int_0^{t \wedge \tau_n} \langle \Delta \int_{r}^{t \wedge \tau_n} S(s-r) \d s \, F({\bm z}_r), \bm u \rangle \d r 
\\ &= \int_0^{t \wedge \tau_n} \langle (S(t \wedge \tau_n-r) - I)F({\bm z}_r), \bm u \rangle \d r.
\end{aligned}
$$
Importantly, we know that $F({\bm z}) \in L^1([0,\tau_n];L^{q/2})$ for $\bm u \in L^\frac{q}{q-2}$, hence these identities are well defined (recall that $q > d \geq 2$). Similarly,
$$
\begin{aligned}
\int_0^{t \wedge \tau_n}\int_0^{s}\langle \Delta \bm u, S(s-r) G({\bm z}_r) \d {\bm W}_r\rangle\d s &= \int_0^{t \wedge \tau_n} \langle \bm u,  \Delta \int_{r}^{t \wedge \tau_n} S(s-r) \d s\, G({\bm z}_r) \d {\bm W}_r\rangle \\ &= \int_0^{t \wedge \tau_n} \langle \bm u, (S(t \wedge \tau_n-r) - I)G({\bm z}_r) \d {\bm W}_r\rangle.
\end{aligned}
$$
The stochastic identity in particular can be justified by applying the stochastic Fubini theorem to $S \diamond (\mathds{1}_{\Set{\cdot \leq \tau_n}} G({\bm z}_\cdot))$ first, and then inserting the stopped time $t \wedge \tau_n$.
Plugging these identities into \eqref{eq:prelimvariational} and, 
rearranging the right hand side and 
integrating by parts on the left-hand side 
then gives \eqref{eq:variationalidentity}. 

Since $H^{1,2}  \cap L^{\frac{q}{q-2}}$ is separable, we can generalise this variational identity to apply for all $\bm u \in H^{1,2} \cap L^{\frac{q}{q-2}}$ simultaneously.
By weak continuity, this extends to all $t \in [0,T]$.
\end{proof}
\begin{remark}
It is desirable that ${\bm z}$ almost surely has paths in $$L^2([0,\tau_n];H^{1,2}) \cap L^\infty([0,\tau_n];L^{\frac{q}{q-2}}),$$ since then one can test ${\bm z}$ against itself and apply Itô-type formulas. This is given, since $B^{d/q}_{q,2p} \hookrightarrow H^{d/q-\varepsilon,q}$ for any $\varepsilon > 0$, so the Sobolev embedding theorems in particular show that $B^{d/q}_{q,2p} \hookrightarrow L^{\frac{q}{q-2}}$. 
\end{remark}
\section{Applications}
First, we shortly present a simple application of the well-posedness result. Let $d \geq 2$ and consider the equation $$\d z(t,\bm x) = \Delta z(t,\bm x) + |\nabla z(t,\bm x)|^2 +z^2(t,\bm x)\d W^Q(t,\bm x),~t \geq 0, \bm x \in \mathbb T^d$$ with $z(0,\bm x) = z_0(\bm x) \in B^{d/q}_{q,p}$ for $q \in (d,2d)$. Here, $W^Q = R\, W$ for some cylindrical Wiener process $W$ on a Hilbert space $U$ and $R \in \gamma(U,H^{\sigma_0,q})$ for $\sigma_0 > 0$, so $W^Q$ is a Wiener process on $H^{\sigma_0,q}(\mathbb T^d)$. In this case, the gradient nonlinearity clearly satisfies Assumption \ref{ass:nemytskiiLipschitz}. Meanwhile, local Lipschitzness of the dispersion coefficient follows from the Prop. \ref{prop:noisecoeffbounds} below. 

As $B^{d/q}_{q,p} \hookrightarrow B^{d/q}_{q,m}$ for any $m \geq p$, Theorem \ref{thm:localexistence} implies existence of a maximal local solution: Choose $m \geq \frac{q}{\delta q - d}$ for $\sigma_0 > \delta - d/q$. Then there exists a local solution of the above equation with initial datum $z_0$ and weakly continuous trajectories in $B^{d/q}_{q,2m}$. 

To emphasize the improvement on previous results in the maximal regularity framework, let $q \to 2d$, so $d/q \to \frac12$. Then $1-d/q \to \frac12$, so in the previous framework, the noise strength $\sigma_0 \geq \frac12 \geq 1-d/q$ is necessary, asymptotically. In contrast, we only require $\sigma_0 > 0$, so asymptotically, we gain almost half a derivative in permissible roughness. This is nearly optimal; since even in the additive noise case, we cannot expect higher roughness of the stochastic convolution than $W^Q \in L^q$: this is the minimum regularity required for existence of at least one derivative. 

\subsection{Application to phase-field models of moving boundary problems} \label{sec:AppliedSec}
We now apply the theory of the previous section to the phase-field model introduced in \cite{ABS}, given by
\begin{equation}\label{eq:StrongPhaseFieldABS}
\begin{cases}
\d \phi_t  = \left(\Delta \phi_t + g(\phi_t,c_t) +  \Psi(\phi_t,c_t) |\nabla \phi_t|\right) \d t \\
\d c_t =  \left(\Delta  c_t + \frac{\nabla \phi_t}{\phi_t} \nabla c_t + f(\phi_t,c_t) \right)\d t + b(\phi_t,c_t) \d W^Q_t
\end{cases}, \qquad \text{on } \mathbb T^d.
\end{equation}
Here, $W^Q$ is a $H^{\sigma,q}$-valued $Q$-Wiener process, i.e. $W^Q \coloneqq R W$ for a cylindrical Wiener process $W$ on some separable Hilbert space $U$ and $R \in \gamma(U,H^{\sigma,q})$. 

\begin{remark}
This subsection is essentially a shortened version of Section 2.3 of the author's dissertation. Therein, the interested reader can in particular find some of the proofs that are omitted in the present work. 
\end{remark}

\begin{remark}
The considerations of this section also apply to several other examples from \cite[Section 5]{Ich2}.
\end{remark}

The motivation for the theory developed in the preceding section is that we can cast this type 
of equation into the form \eqref{eq:FullNoiseSPDE} by means of a Cole-Hopf type transform. Assume that $\phi > 0$ almost surely and set $$z \coloneqq \log \frac 1\phi = - \log \phi.$$ Then, \eqref{eq:StrongPhaseFieldABS} transforms to 
\begin{equation}
\begin{cases}
\d z_t = \left(-\frac{\Delta \phi}{\phi} - \frac{g(\phi_t,c_t)}{\phi_t} - \Psi(\phi_t,c_t) \left \lvert \frac{\nabla \phi_t}{\phi_t} \right \rvert\right) \d t \\
\d c_t =  \left(\Delta  c_t + \frac{\nabla \phi_t}{\phi_t}\nabla c_t + f(\phi_t,c_t) \right)\d t +b(\phi_t,c_t) \d W^Q_t
\end{cases},
\end{equation}
After closing the equation in $z = -\log \phi$, $\nabla z = -\frac{\nabla \phi}{\phi}$, we obtain
\begin{equation}\label{eq:StrongColeHopf} 
\begin{cases}
\d z_t  = \left(\Delta z_t - |\nabla z_t|^2 +  h(z_t,c_t) + \Phi(z_t,c_t) |\nabla z_t|\right) \d t \\
\d c_t =  \left(\Delta  c_t - \nabla z_t \nabla c_t + k(z_t,c_t) \right)\d t + \beta(z_t,c_t) \d W^Q_t
\end{cases},
\end{equation}
where $$ h(z,c)\coloneqq -\frac{g(e^{-z},c)}{e^{-z}}, \quad \Phi(z,c) \coloneqq - \Psi(e^{-z},c),\quad  k(z,c) \coloneqq f(e^{-z},c), \quad \beta(z,c) \coloneqq b(e^{-z},c).$$ 
We impose the following assumptions in the coefficients:
\begin{assumption} \label{ass:ColeHopfGLipschitz}
For $u, v \in H^{1,q}$, it holds that $$\norm{\frac{g(u_1,v_1)}{u_1}-\frac{g(u_2,v_2)}{u_2}}_{L^{\frac q 2}} \leq C_\infty \left(u_1,u_2,v_1,v_2\right)\left(\norm{u_1-u_2}_{L^q}+\norm{v_1-v_2}_{L^q}\right)$$
for some constant $C_\infty$ dependent on $\norm{u_1}_{L^\infty},\norm{u_2}_{L^\infty},\norm{v_1}_{L^\infty},\norm{v_2}_{L^\infty}$.
\end{assumption}
\begin{remark}
On the torus, this in particular applies to $g(x) = x(1-x)(x-a)$, $a \in \mathbb R$.
\end{remark}

\begin{assumption} \label{ass:PsiCoeff}
The nonlinearity $\Psi$ is of the form $$\Psi(\phi,c) = \sum_{i=1}^\ell \Psi_i(\phi,c) ,$$ for some finite collection of  nonlinearities $\Psi_i \colon L^\infty \oplus {L}^\infty \rightarrow L^\infty$ that satisfy the Lipschitz property
$$|\left(\Psi_i(\phi_1, c_1) - \Psi_i(\phi_2,c_2)\right)(\bm x)| \leq L_R\left(|\phi_1(\bm x)- \phi_2(\bm x)|+|c_1(\bm x)-c_2(\bm x)| + \norm{\phi_1-\phi_2}_{L^2}+ \norm{c_1- c_2}_{{L}^2}\right)$$ whenever $\norm{\phi_1}_{L^\infty}, \norm{\phi_2}_{L^\infty}, \norm{c_1}_{{L}^\infty},\norm{c_2}_{{L}^\infty} \leq R$.
\end{assumption}

\begin{assumption}\label{ass:NemytskiiAssumption}
The nonlinearities $g,  f$ should generally correspond to Nemytskii-type operators with dependence on nonlocal properties of inputs. Due to the conditions we impose on solutions and initial conditions, it suffices to specify their behaviour on $L^\infty$. Thus, we only assume that for each $R > 0$, there exists a constant $L_R$ such that  $f \colon L^\infty \oplus {L}^\infty \rightarrow {L}^\infty$ admits the bound $$|( f(\phi_1, c_1)- f(\phi_2,c_2))(\bm x)|\leq L_R\left(|\phi_1(\bm x)-\phi_2(\bm x)|+|c_1(\bm x)- c_2(\bm x)| + \norm{\phi_1-\phi_2}_{L^q}+ \norm{c_1-c_2}_{L^q}\right),$$
$\!\d {\bm x}$-almost everywhere whenever $\norm{\phi_1}_{L^\infty}, \norm{\phi_2}_{L^\infty}, \norm{ c_1}_{{L}^\infty},\norm{ c_2}_{{L}^\infty} \leq R$.
\end{assumption}
\begin{assumption} \label{ass:Invariance}
Let real numbers $K, L$ with $K <L$ be given and introduce the spaces $$\mathcal X_{\phi} = \left\{\phi \in L^\infty(\mathbb T^d) :  \phi(\bm x) \in [0,1],~ \d {\bm x}\text{-a.s.} \right\}, ~\mathcal X_{ c} = \left\{c \in {L}^\infty(\mathbb T^d) :  c(\bm x) \in [L,K],~\d {\bm x}\text{-a.s.} \right\}.$$ We assume that  whenever $\phi \in  \mathcal{X}_\phi$ and  $c \in \mathcal X_{ c}$, then $$f(\phi,  c) \mathds{1}_{\Set{c \equiv L}} \geq 0 \text{ and } f(\phi, c) \mathds{1}_{\Set{c \equiv K}} \leq 0.$$ and $$g(\phi,  c) \mathds{1}_{\Set{\phi \equiv 0}}  = 0 \text{ and } g(\phi,  c) \mathds{1}_{\Set{\phi \equiv 1}} \leq 0.$$ 
\end{assumption}
\begin{remark}
These conditions on $f$ and $g$ are modelled after nonlocal reaction terms appearing in the system introduced in \cite{ABS}, see also \cite[Example 5.1]{Ich2}.
\end{remark}

\begin{assumption} \label{ass:CriticalNemytskiiDispersion}
The dispersion coefficient $b$ has locally bounded second-order partial derivatives.
\end{assumption}

\begin{remark} \label{rem:CriticalNewNemytskiiExplained}
This specific form of $b$ will be quite tame on bounded solutions, as the truncation $\widetilde \beta$ satisfies the local Lipschitz assumption \ref{ass:nemytskiiLipschitz} and additionally the bound $$\norm{\widetilde \beta(u,v)}_{\mathcal L(H^{\sigma,q},H^{\sigma_0,q})} \leq \const\left(1+\norm{u}_{H^{\frac{d}{q},q}}+\norm{v}_{H^{\frac{d}{q},q}}\right),$$ see \eqref{item:NemytskiiBound} in Prop. \ref{prop:nemytskiibounds} with $k=1$. The solution theory developed in this section would however still apply for nonlinearities $\beta$ with \begin{equation} \label{eq:CriticalAlternativeNemytskiiBound}
\norm{\widetilde \beta(u,v)}_{\mathcal L(H^{\sigma,q},L^q)} \lesssim L_n\norm{u}_{H^{\delta,q}}\norm{v}_{H^{\delta,q}}+\norm{u}^2_{H^{\delta,q}}+\norm{v}_{H^{\delta,q}}+1,
\end{equation} 
where $L_n$ is uniform on $\norm{u}_{B^{d/q}_{q,2p}} \leq n$ (cf. Remark \ref{rem:LocalLipschitzTraceNorm}). This expression avoids any quadratic terms in the second variable, corresponding to $c_t$. With such an affine linear bound in $v$, we are still able to leverage deterministic estimates on $z$ to restrict growth of $c_t$ and thereby rule out blow-up.
\end{remark}

\begin{remark}
On another note, one could in principle choose a host of other conditions that make our solution theory applicable to $\widetilde h$. For example, one can use smooth truncations instead of the Lipschitz truncations employed above. In that case, the Assumption \ref{ass:ColeHopfGLipschitz} can be relaxed to 
$$
\begin{aligned}
\norm{\frac{g(u_1,v_1)}{u_1}-\frac{g(u_2,v_2)}{u_2}}_{L^{\frac q 2}} \leq C_\infty \left(u_1,u_2,v_1,v_2\right)\left(\norm{u_1-u_2}_{H^{1,q}}+\norm{v_1-v_2}_{H^{1,q}}\right),
\end{aligned}
$$ and consequently $\widetilde h$ satisfies $$
\begin{aligned}
&\norm{\widetilde h(z_1,c_1)-\widetilde h(z_2,c_2)}_{L^{\frac q 2}}
\leq \const \left(\norm{z_1}_{H^{1,q}}+\norm{z_2}_{H^{1,q}}\right)\left(\norm{z_1-z_2}_{H^{1,q}}+\norm{c_1-c_2}_{H^{1,q}}\right).    
\end{aligned}
$$ 
However, to assure boundedness of solutions, we then require a different invariance conditions to account for the smooth truncation, such as $g(x,c) \leq 0$ for all $x \in [1,1+\varepsilon]$ for some $\varepsilon > 0$ and arbitrary $c \in \mathcal K$. (cf. Assumption \ref{ass:Invariance}). 
\end{remark}

We will restrict ourselves to solutions with values in the spaces $\mathcal X_\phi,\mathcal X_c$ (cf. Assumption \ref{ass:Invariance}). 
Due to the well-behaved nature of Nemytskii operators on critical Sobolev spaces, we will be able to prove existence of solutions for any spatial noise regularity $\sigma > 0$. This improves on the conditions of \citet{AgrestiVeraarNonLinParab}, where the minimal possible noise regularity is $\sigma \geq 1-d/q$, for initial data with Besov regularity $B^{d/q}_{q,p}$, $p \in [1,\infty)$.  

We further exploit that the equation governing the dynamics of $\phi$ is deterministic: we can thus additionally apply the previously sketched Kato-type approach of \cite{BENARTZI2002343}. This is useful in several senses; for one, this helps us avoid some more involved estimates in Besov spaces. On the other hand, we obtain control on expectations of the bilinear nonlinearity $\nabla z \nabla c$. We then obtain the following theorem.
\begin{theorem} \label{thm:phasefieldglobalstrongex}  Let the coefficients of \eqref{eq:StrongPhaseFieldABS} satisfy assumptions \ref{ass:ColeHopfGLipschitz} to \ref{ass:CriticalNemytskiiDispersion}. Let $W^Q$ take values in $H^{\sigma,q}$ for some arbitrary $\sigma > 0$ and $q \in (d,2d)$. Let $$p > \frac{1}{\sigma \wedge \alpha_1}$$ and $z_0, c_0 \in L^0(\Omega;B^{d/q}_{q,2p})$ be $\mathcal F_0$-measurable. Then, if $z_0 \geq 0$, $c_0 \in \mathcal X_c$ almost surely, there exists a unique global-in-time mild and variational solution $z,c \in C([0,T];B^{d/q}_{q,2p})$ of \eqref{eq:StrongColeHopf} with initial data $z_0, c_0$ in the sense of Definition \ref{eq:mildlocalsol}. These solutions satisfy the additional regularity 
$$
z,c \in L^4([0,T];H^{1,q}).
$$
Setting $\phi \coloneqq e^{-z}$,  the pair
$(\phi, c)$ is an analytically weak and mild solution of \eqref{eq:StrongPhaseFieldABS} with $\log \frac{1}{\phi_0} = z_0$. In particular, 
$$
\phi \in C([0,T];B^{d/q}_{q,2p}) \cap C((0,T];H^{1,q})
$$
with $\sup_{t\in (0,T]} t^{\frac12 - \frac{d}{2q}} \norm{\phi}_{H^{1,q}} < \infty$ and
$\phi_t \in [0,1]$, $c_t \in \mathcal X_c$ for all $t \in [0,T]$ almost surely. Moreover, $(\phi,c)$ is the unique solution with these regularities. 
\end{theorem}
\begin{remark} \label{remark:MorreyRemark}
Since Morrey's inequality fails at the critical Besov index, this in particular means that we can allow for unbounded initial values for $z_0$. Translating this back, we find that $\phi_0$ may vanish on a countable set of points. Thus, $\nabla \log \phi$ can indeed exhibit singular behaviour at $t=0$, since moreover, the derivative required by the unweighted singular term might not exist at $t=0$.
\end{remark}

As mentioned, we can restrict our analysis to the truncated system
\begin{equation}\label{eq:StrongColeHopfTrunc} 
\begin{cases}
\d z_t  = \left(\Delta z_t - |\nabla z_t|^2 +  \widetilde h(z_t,c_t) + \widetilde \Phi(z_t,c_t) |\nabla z_t|\right) \d t \\
\d c_t =  \left(\Delta  c_t - \nabla z_t \nabla c_t + \widetilde k(z_t,c_t) \right)\d t + \widetilde \beta(z_t,c_t) \d W^Q_t,
\end{cases}
\end{equation}
whenever $z_0 \geq 0$, $c_0 \in \mathcal X_c$ almost surely. Here, $\widetilde h$ is defined by $$\widetilde h(z,c) \coloneqq h(z\vee0, L \vee c \wedge K) = -\frac{g(e^{-z \vee 0}, \,L\vee c \wedge K)}{e^{-z \vee 0}},$$
while $\widetilde \Phi(z,c) \coloneqq - \widetilde \Psi(e^{-z},c),$ and $\widetilde k(z,c) \coloneqq \widetilde f(e^{-z},c)$ are defined using the truncations $$\widetilde \Psi(e^{-z },c) \coloneqq \Psi(e^{-z \vee 0}, L \vee c \wedge L), \quad \widetilde f(e^{-z},c) \coloneqq f(e^{-z \vee 0}, L \vee c \wedge K).$$ 
The definition of $\widetilde \beta$ is of a slightly different pattern, to ensure differentiability. We define $$\widetilde \beta(z,c) \coloneqq \widetilde b(e^{-\eta(z)},c_t)$$ for some smooth cut-off with $\eta(x) \equiv x$ on $[0,\infty)$ and $\eta(x) \equiv 0$ for $x < -1$ (or some other negative threshold).

\begin{lemma} \label{lemma:phasefieldlocalstrongex}
Let $W^Q$ take values in $H^{\sigma,q}$ for some arbitrary $\sigma > 0$ and $q \in (d,2d)$. Let $$p > \frac{1}{\sigma \wedge \alpha_1}$$ and $z_0, c_0 \in L^0(\Omega;B^{d/q}_{q,2p})$ be $\mathcal F_0$-measurable. Then a maximal local solution of \eqref{eq:StrongColeHopfTrunc} with initial data $z_0, c_0$ exists (in the sense of Definition \ref{def:localsol}).
\end{lemma}
\begin{proof}
To elucidate the condition $p > \frac 1 \sigma$, note that for fixed $\sigma \leq 1-\frac{d}{q}$, we need to choose $\delta > \frac dq$ with $\sigma > \delta - \frac{d}{q}$. Then, $p \geq \frac{q}{\delta q-d} = \frac{1}{\delta - d/q}$ for some such $\delta$ if and only if $p > \frac{1}{\sigma}$. On the other hand, if $\sigma > 1-\frac{d}{q}$, then we can simply apply the condition $p \geq \frac{q}{q-d}= \frac{1}{\alpha_1}$, where $\alpha_1 \coloneqq 1-\frac{d}{q}$ was defined in the preceding section.

To apply Theorem \ref{thm:localexistence}, note that the gradient nonlinearities evidently satisfy Assumption \ref{ass:DriftNonlin}. Since $z \mapsto e^{-z \vee 0}$ is globally bounded and Lipschitz,
\begin{enumerate}[(1)]
    \item Assumption \ref{ass:ColeHopfGLipschitz} implies that $\widetilde h$
    is Lipschitz as a map $\widetilde h \colon L^q \oplus L^q \rightarrow L^{q/2}$
    \item $\widetilde \Phi$ is Lipschitz 
    \item $\widetilde k$ is Lipschitz.
\end{enumerate} 
Moreover, since $\widetilde \beta$ is bounded with bounded first- and second derivatives, Proposition \ref{prop:noisecoeffbounds} below shows that $\widetilde \beta$ satisfies the local Lipschitz property in $H^{\delta,q}$ for any $\delta > \frac{d}{q}$. Thus the claim follows
\end{proof}
\begin{proposition} \label{prop:noisecoeffbounds}
Assume that $g \colon \mathbb R^m \rightarrow \mathbb R$ has bounded second order partial derivatives. For any $\sigma_0 > \sigma$, let $R \in \gamma(U,H^{\sigma_0,q})$ be a covariance operator. Then the dispersion coefficient $$G({\bm u}) \coloneqq g({\bm u}) \circ R$$ satisfies Assumption \ref{ass:nemytskiiLipschitz} for any $\delta > \frac dq$.
\end{proposition} 
\begin{proof}
This is a direct corollary of Proposition \ref{prop:nemytskiibounds} and the Sobolev embeddings.
\end{proof}
\begin{remark}
Prop. \ref{prop:nemytskiibounds} actually implies that if $g$ has bounded $k$-th order partial derivatives, then
$$\norm{G({\bm u})-G({\bm v})}_{\gamma(U,H^{\sigma,q})} \leq \const (1+\norm{{\bm u}}^{k-1}_{H^{\delta,q}}+\norm{{\bm v}}^{k-1}_{H^{\delta,q}})\norm{{\bm u}- {\bm v}}_{H^{\delta,q}}$$
for any $\delta \in \left(\frac dq, 1\right)$ and  $\sigma > (k-1)\left(\delta-\frac dq\right)$. It is plausible that the theory developed in the previous section still applies in this case, after appropriate adaptations of exponents and weights in the solution space $\mathbb Y_\delta$. 
\end{remark}
\begin{proposition} \label{prop:nemytskiibounds}
Assume that $g$ has bounded $k$-th order partial derivatives for $k \geq 2$. Let $1\geq \delta > \frac dq$ and ${\bm u}, {\bm v} \in H^{\delta,q}(\mathbb T^d)$. The Nemytskii operator defined by $g$ (cf. Assumption \ref{ass:nemytskiiLipschitz}) satisfies
\begin{enumerate}[(i)]
    \item \label{item:NemytskiiLipschitzEst}
    $$\norm{g({\bm u})-g({\bm v})}_{H^{\delta,q}} 
        \leq \const (1+\norm{{\bm u}}^{k-1}_{H^{\delta,q}}+\norm{{\bm v}}^{k-1}_{H^{\delta,q}})\norm{{\bm u}-{\bm v}}_{H^{\delta,q}}$$
    \item \label{item:NemytskiiOperatorLipschitz} For $0 \leq \sigma < \sigma_0 \leq \frac{d}{q}$, $$\norm{g({\bm u})-g({\bm v})}_{\mathcal L(H^{\sigma_0,q}, H^{\sigma,q})} \leq \const (1+\norm{{\bm u}}^{k-1}_{H^{\delta,q}}+\norm{{\bm v}}^{k-1}_{H^{\delta,q}})\norm{{\bm u}-{\bm v}}_{H^{\delta,q}}.$$ 
    \item \label{item:NemytskiiBound} For any $0 \leq \sigma < \sigma_0< \frac{d}{q}$ and $k \geq 1$,    $$\norm{g({\bm u})}_{\mathcal L(H^{\sigma_0,q},H^{\sigma,q})} \leq \const\left(1+\norm{{\bm u}}^{k-1}_{H^{\frac{d}{q},q}}\right)\left(1+\norm{{\bm u}}_{H^{\frac{d}{q},q}}\right).$$
\end{enumerate}
for some constant factors dependent on $\delta,q, \sigma, \sigma_0, g$ and $d$, but not on ${\bm u},{\bm v}$.
\end{proposition}

\subsection{Regularity properties of the transformed phase-field}
With existence and uniqueness of local solutions settled, we can prove global-in-time existence by controlling blow-up of each equation individually. This stands in contrast to more difficult joint estimates. For this reason, the theory developed in the previous section simplifies establishing global solutions of the system. However, a drawback of the framework is the increased complexity of working in Besov spaces. To circumvent some of the difficulties, we follow \citet{BENARTZI2002343} and introduce the spaces 
$$\mathcal W_{r}(T) = \{u \in C((0,T]; H^{1,r}): \sup_{t \in (0,T]} t^{\frac{\alpha_1}2}\norm{u(t)}_{H^{1,r}}  < \infty\},~ r \in [1,\infty)$$ where $\alpha_1$ is defined as in Definition \ref{def:solspace}. Further, we set $$\mathbb Z_q(T) \coloneqq \mathbb X_1(T) \cap \mathcal W_q(T).$$
\begin{remark}
$\mathcal W_q$ is the $p=\infty$ variant of the previously introduced space $\mathbb W_1$, see also Prop. \ref{prop:InterpolationIdentity}.
\end{remark}
\begin{proposition} \label{prop:zmoreregularity}
Let $z_0 \in L^0(\Omega;B^{d/q}_{q,2p})$ be $\mathcal F_0$-measurable and take values in a compact subset $\mathcal A \Subset B^{d/q}_{q,2p}$. Given a measurable process $c \colon \Omega \times [0,\tau] \rightarrow \mathcal X_c$, there exists a 
deterministic $t_0 > 0$ and unique mild solution $z$ of 
\begin{equation} \label{eq:decoupledcolehopf}
\d z_t  = \left(\Delta z_t - |\nabla z_t|^2 + \widetilde h(z_t,c_t) + \widetilde\Phi(z_t,c_t) |\nabla z_t|\right) \d t
\end{equation}
with $$z \in L^\infty(\Omega;\mathbb Z_q(\tau_0)),$$ where $\tau_0 = \tau \wedge t_0$. In particular, for $(z,c)$ given by Lemma \ref{lemma:phasefieldlocalstrongex}, $z \in \mathbb Z_q(\tau_0)$. Moreover, if $z_0\geq 0$, $\mathbb P$-a.s., then $z_t \geq 0$ for all $t \in [0,t_0 \wedge \tau)$, $\mathbb P$-a.s.
\end{proposition}
\begin{proof}  The proof combines the fixed-point argument of \citet[Prop. 2.4]{BENARTZI2002343} with our fixed point argument in $\mathbb X_1$. Namely, the inequality $$\norm{z}_{\mathbb Z_q(t)} \leq \norm{S(\cdot)z_0}_{\mathbb Z_q(t)} + \const (T^{\alpha_1}+\norm{z}_{\mathbb Z_q(t)})\norm{z}_{\mathbb Z_q(t)},$$ can be derived analogously to our previous estimates (resp. the estimates in \cite{BENARTZI2002343}). Since $\norm{S(\cdot)z_0}_{\mathbb Z_q(t)}< \infty $ is bounded uniformly for a fixed $z_0 \in B^{d/q}_{q,2p} \hookrightarrow B^{d/q}_{q,\infty}$ by Prop. \ref{prop:InterpolationIdentity}, density of e.g. $H^{1,q} \subset B^{d/q}_{q,2p}$ shows that $$\norm{S(\cdot)z_0}_{\mathbb Z_q(t)} \to 0.$$ This can be extended to  $z \in \mathcal A \Subset B^{d/q}_{2p}$ by compactness. Thus, we find that for a deterministic small time $t_0$, the solution map is an endomorphism on a small ball \begin{equation} \label{eq:zsmallball}
B_R(0) \subset \mathbb Z_q(t)
\end{equation} with $R \ll 1$. It is now analogous to our fixed point argument to show that this map defines a contraction on said ball. In particular, by the regularity imposed by the solution space $\mathbb Z_q$, the resulting process must agree with the solution obtained in Lemma \ref{lemma:phasefieldlocalstrongex}. Nonnegativity of $z$ now follows as in the proof of Theorem \ref{thm:c_global_existence}. 
\end{proof}

\begin{proposition} \label{prop:ExponentialContinuity}
Let $0 < \mathfrak s < 1$ be arbitrary. Let $\eta$ be smooth, Lipschitz with bounded first derivative and $\eta(x) \equiv x$ on $[0,\infty)$, $\eta(x) \equiv 0$ on $(-\infty,-1]$. Then, for all $z \in \mathcal B^{\mathfrak s}_{q,p}$,$$
\|e^{-\eta(z)}\|_{B^\mathfrak s_{q,p}} \leq C\left(1+\|z\|_{B^\mathfrak s_{q,p}}\right).
$$ In particular, the map
$$
B^{\mathfrak s}_{q,p} \ni z \mapsto e^{-\eta(z)} \in B^\mathfrak s_{q,p}
$$
is continuous. Thus, if $X \Subset B^\mathfrak s_{q,p}$ and every element of $X$ is nonnegative, then
$
\{e^{-z}: z \in X\} \Subset B^\mathfrak s_{q,p}.
$
\end{proposition}
\begin{proof}
For $\mathfrak s < 1$, we may use the standard first-difference characterisation of Besov spaces on $\mathbb T^d$:
$$
\norm{{u}}_{B^\mathfrak s_{q,p}}
\simeq
\norm{{u}}_{L^q}
+
\left(
\int_{|{\bm h}|<1}
\left(
|{\bm h}|^{-\mathfrak s}\norm{\tau_{\bm h} u-u}_{L^q}
\right)^p
\frac{\d {\bm h}}{|{\bm h}|^d}
\right)^{1/p},
$$
for the periodic shift $\tau_{\bm h} u({\bm x}) \coloneqq u({\bm x}+{\bm h})$. Since $\norm{e^{-\eta(z)}}_{L^q} \leq e^1\cdot|\mathbb T^d|^\frac{1}{q}$, we only need to control the Besov seminorm to show boundedness. Indeed, for every ${\bm h}$,
$$
|e^{-\tau_{\bm h} \eta(z)}-e^{-\eta(z)}|
\leq
|\tau_{\bm h}z-z|
$$
almost everywhere, and hence
$
\norm{\tau_{\bm h} e^{-\eta(z)}-e^{-\eta(z)}}_{L^q}
\leq
\norm{\tau_{\bm h} z-z}_{L^q}.
$
Therefore, 
$$
\norm{e^{-\eta(z)}}_{B^s_{q,p}} \leq \const \left(1+ \norm{z}_{B^s_{q,p}}\right).
$$
Continuity now follows either by a blackbox theorem which implies continuity of bounded nonlinear operators \cite[Theorem 3, Section 5.5.2]{RunstSickelBesov}. Alternatively, one can circumvent the boundedness proof deduce continuity from real interpolation of nonlinear operators, since the composition operator is continuous on $H^{1,q}$ (ibid, Prop. 3, Section 2.5.4). The compactness property is a direct implication of either route.
\end{proof}
\begin{proposition} \label{prop:exponentialregularity}
Let $z$ denote the process obtained in Prop. \ref{prop:zmoreregularity}. Then $e^{-z} \in \mathcal W_q(\tau_0)$.
\end{proposition}
\begin{proof}
This follows from nonnegativity of $z \in C((0,\tau_0);H^{1,q})$.
\end{proof}
It is now standard to derive the following property.
\begin{proposition} \label{prop:exponentialweaksol}
On $[0,\tau_0]$, $e^{-z}$ is an analytically weak solution of \ref{eq:DecoupledMembEq}.
\end{proposition}
\subsection[Regularity properties of]{Regularity properties of the phase-field}
\begin{lemma}
\label{lemma:initialtimephisol}
Let $\phi_0 \in L^\infty(\Omega;B^{d/q}_{q,p})$, $q > d$, be $\mathcal F_0$-measurable and take values in a compact subset $\mathcal A \Subset B^{d/q}_{q,p}$. Given a measurable $c \colon [0,\tau] \rightarrow \mathcal X_c$, there exists a deterministic small $t_1 > 0$ and a unique mild solution $\phi$ of \begin{equation}\label{eq:DecoupledMembEq}
\partial_t \phi_t =  \Delta \phi_t + \widetilde g(\phi_t,c_t) +\widetilde\Psi(\phi_t, c_t) |\nabla \phi_t|,
\end{equation}
such that $$\phi \in L^\infty(\Omega;\mathcal W_q(\tau_1) \cap C([0,\tau_1];B^{d/q}_{q,p})),$$
where $\tau_1 = \tau \wedge t_1$.
\end{lemma}
\begin{proof}
This proof is similar to the proof of Prop. \ref{prop:zmoreregularity} and therefore omitted. 
\end{proof}
\begin{lemma}
\label{lemma:positive-time-phisol}
Let $\phi_0 \in H^{1,
q}$, $q > d$, be given. Given a measurable $c \colon [0,\tau] \rightarrow \mathcal X_c$, there exists a unique mild solution $\phi \in C([0,\tau];H^{1,q})$ of \begin{equation}
\partial_t \phi_t = \Delta \phi_t + \widetilde g(\phi_t,c_t) +\widetilde\Psi(\phi_t, c_t) |\nabla \phi_t|.
\end{equation}
with 
$$
\norm{\phi}_{C([0,\tau];H^{1,q})}\leq \const \norm{\phi_0}_{H^{1,q}}.
$$
\end{lemma}
\begin{proof}
We utilise the Banach fixed point theorem. Consider the operator $$\mathcal T \colon C([0,T];H^{1,q}) \rightarrow C([0,T];H^{1,q}), ~ u \mapsto S(\cdot) u_0 + \int_0^\cdot S(\cdot-s)(\widetilde g(u_s,c_s)+\widetilde \Psi(u_s,c_s)|\nabla u_s|)\d s.$$ Then the assumed properties on the coefficients imply that 
$$\norm{(\mathcal T u)_t}_{L^q} \leq \norm{u_0}_{L^q} + \int_0^t\const (\norm{u_s}_{L^q}+\norm{\nabla u_s}_{L^q}) \d s \leq \norm{u_0}_{L^q} + \const t \sup_{0 \leq s \leq t} \norm{u_s}_{H^{1,q}}$$
and, due to analyticity of the heat semigroup, 
$$\norm{\nabla (\mathcal T u)_t}_{L^q} \leq \norm{\nabla u_0}_{L^q} + \int_0^t\const (t-s)^{-\frac12} (\norm{u_s}_{L^q}+\norm{\nabla u_s}_{L^q}) \d s \leq \norm{\nabla u_0}_{L^q} + \const t^{\frac12 } \sup_{0 \leq s \leq t} \norm{u_s}_{H^{1,q}}.$$ From this we can conclude that 
\begin{equation} \label{eq:MembEqMapBounded}
\norm{\mathcal T u}_{C([0,T];H^{1,q})} \leq \norm{u_0} + \const (T+ T^{\frac12}) \norm{u}_{C([0,T];H^{1,q})}.
\end{equation} Now fix $\phi_0 \in H^{1,q}$ as initial condition. Then \eqref{eq:MembEqMapBounded} shows that for small enough $T > 0$ and $R = 2\norm{\phi_0}_{H^{1,q}}$, the solution map $\mathcal T$ maps $$B_R(0) \subset C([0,T];H^{1,q})$$ into itself. Now 
a similar estimate as above yields that for some small $\epsilon > 0$, 
$$
\norm{\mathcal T(u-v)}_{C([0,T];H^{1,q})} \leq \const (T+T^{\frac12}+RT^{\frac12-\frac{q}{2(q+\epsilon)}})\norm{u-v}_{C([0,T];H^{1,q})}.
$$ Note that the constant $C_{\themycounter}$ depends on the radius $R \geq \norm{u_0}_{H^{1,q}}$. The scaling 
$T^{\frac12-\frac{q}{2(q+\epsilon)}}$ stems from the hypercontractivity and Sobolev estimate 
$$
\begin{aligned}
\norm{\nabla S(t-s) |\nabla u|(\widetilde \Psi(u,c_t)-\widetilde\Psi(v,c_t))}_{L^q} &\leq \const (t-s)^{-\frac12} \norm{S\left(\frac{t-s}{2}\right) |\nabla u|(\widetilde \Psi(u,c_t)-\widetilde\Psi(v,c_t))}_{L^q}
\\&\leq \const (t-s)^{-\frac12 -\frac{q}{2(q+\epsilon)}} \norm{\nabla u}_{L^q}\norm{u-v}_{L^{q+\epsilon}}
\\  &\leq \const  R(t-s)^{-\frac12 -\frac{q}{2(q+\epsilon)}} \norm{u-v}_{H^{1,q}}
\end{aligned}$$ 
for $\epsilon$ small enough to employ the Sobolev embedding theorem. Again, choosing small enough $T_0$, we conclude that $\mathcal T$ is a contraction mapping on $B_R(0)$. 

Now, we can repeat the type of argument that yielded \eqref{eq:MembEqMapBounded} and apply a Grönwall lemma for singular kernels (e.g. \citet[Ch. 5, Lemma 6.7]{PazySemigroupsPDE}) to obtain a global bound on $\norm{\phi_t}_{H^{1,q}}$ in terms of the Mittag-Leffler function, i.e. we know that on $[0,\tau]$, solutions of \eqref{eq:DecoupledMembEq} are bounded uniformly. Since any bounded solution can be extended uniquely by the fixed-point argument, we can conclude solutions exist on the entire interval $[0,\tau]$ and any two solutions with bounded paths in $H^{1,q}$ must be identical.
\end{proof}

\begin{corollary} \label{cor:whole-interval-phisol}
Assume the setting of Lemma \ref{lemma:initialtimephisol}. Then there exists a unique mild and variational solution on the entire interval $[0,\tau]$. Moreover, \begin{equation}
\norm{\phi}_{\mathcal W_r(\tau)} \leq \begin{cases}
\norm{\phi}_{\mathcal W_r(\tau_1)} & \tau < t_1 \\
(1+\const t^{-\alpha_1/2}_1)\norm{\phi}_{\mathcal W_r(t_1)} & \tau \geq t_1.
\end{cases} \qedhere
\end{equation}
\end{corollary}
\begin{proof}
Glueing together the solutions obtained in the preceding two lemmata shows existence of solutions on the entire interval. The regularity properties of the solution make it standard to show that this is a variational solution as well. Moreover, since $t_1$ is deterministic, we find that \begin{equation}
\norm{\phi}_{\mathcal W_r(\tau)} \leq \begin{cases}
\norm{\phi}_{\mathcal W_r(\tau_1)} & \tau < t_1 \\
(1+\const t^{-\alpha_1/2}_1)\norm{\phi}_{\mathcal W_r(t_1)} & \tau \geq t_1.
\end{cases} \qedhere
\end{equation}
\end{proof}
The next result follows analogously as in \cite[Thm. 2.24]{Pardoux2021}.
\begin{proposition}
If $0 \leq \phi_0 \leq 1$, $\d {\bm x}$-a.s., then $0 \leq \phi_t \leq 1$ for all $t\in [0,\tau]$.  
\end{proposition}
The following lower bound, derived in \cite[Prop 4.17]{Ich2}, will be useful in our analysis.
\begin{proposition} \label{prop:Subsolbound}
Suppose that $\phi_0\geq 0$ and $\phi_0 \not \equiv 0$. Let $\phi_t \in L^2([0,T];H^{1,2}) \cap L^\infty([0,T];L^2)$ denote a variational solution of \eqref{eq:DecoupledMembEq}. Then, for all $\delta > 0$, there exists $\kappa_\delta > 0$ with
$$\inf_{\delta \leq t \leq T} \essinf_{x \in \mathbb T^d} \phi(t,x) > e^{-M_1T}\kappa_\delta \int \phi_0 \d {\bm x},$$ where $M_1$ denotes the Lipschitz constant of $g$ on $[0,1]$, and in particular the function $\phi$ has full support on $\mathbb T^d$ for all positive times.
\end{proposition}
\subsection{Globality of solutions}
\begin{proposition} \label{prop:zphiidentification}
Let $z$ be the local mild solution of \eqref{eq:decoupledcolehopf} with $0 \leq z_0 \in B^{d/q}_{q,2p}$. Restrict the initial data on the subset of functions with $$\int \phi_0 \d x > \varepsilon > 0$$ for some fixed $\varepsilon > 0$. Then $z = \log \frac{1}{\phi}$, where $\phi$ is the unique mild solution of \eqref{eq:DecoupledMembEq} and thereby 
\begin{equation} \label{eq:zfinalregularity}
z \in L^\infty(\Omega;C([0,\tau];B^{d/q}_{q,2p}) \cap \mathcal W_q(\tau))
\end{equation}with $z_t \geq 0$ for all $t \in [0,\tau]$.
\end{proposition}
\begin{proof}
We derived in Prop. \ref{prop:exponentialweaksol}, that $\bar \phi \coloneqq e^{-z}$ is a variational solution of \eqref{eq:DecoupledMembEq}. By the derived regularities, it is standard to obtain a mild representation of $\bar \phi$ on $[0,\tau_0]$. By Prop. \ref{prop:ExponentialContinuity}, we additionally know that $\bar \phi_0 = e^{-z_0} \in B^{d/q}_{q,2p}$. Thus, the mild representation combined with the fact that $\bar \phi \in \mathcal W_q(\tau_0) \hookrightarrow L^4([0,\tau_0];H^{1,q})$ shows that
$$\bar \phi = \underbrace{S(\cdot) \bar \phi_0}_{\in C([0,\tau];B^{d/q}_{q,2p})} + \underbrace{S \ast\left(\widetilde g( \bar \phi,c)+\widetilde\Psi( \bar \phi,c)|\nabla  \bar \phi|\right)}_{\in C([0,\tau_0];H^{1,q})} \in C([0,\tau_0];B^{d/q}_{q,2p}).$$
Evidently, the singleton $\phi_0$ is compact and we can conclude that $e^{-z}$ must be equal to the mild solution given by Cor. \ref{cor:whole-interval-phisol}. 

For $\tau \leq t_0 \wedge t_1$, this yields the claim.  The fact that $z \geq 0$ then follows immediately from $ \phi \leq 1$. If $\tau > t_0 \wedge t_1$, let $\sigma_n = \inf \{s \geq 0: \phi_s \neq e^{-z_s}\} \wedge \tau_n,$ where $(\tau_n)_{n \in \N}$ denotes the localising sequence of $\tau$. Naturally, $$\phi_{\sigma_n} = e^{-z_{\sigma_n}} \in B^{d/q}_{q,2p},\quad z_{\sigma_n}\in B^{d/q}_{q,2p} \text{ with }z_{\sigma_n} \geq 0.$$ By a repetition of the procedure from the first paragraph, it must be the case that $\sigma_n = \tau_n$, and it follows that $z \equiv \log \frac{1}{\phi}$ on $[0,\tau]$.

Finally, we prove the stochastic bound \eqref{eq:zfinalregularity}. On $[0,\tau_0]$, Prop. \ref{prop:zmoreregularity} already implies the claim. If $\tau > t_0$, we can invoke Prop. \ref{prop:ExponentialContinuity} to deduce that $e^{-z}$ takes values in a compact set. Then, by virtue of the deterministic lower bound given by Prop. \ref{prop:Subsolbound}, the regularities of $\phi$ transfer to $z$. Namely, we know that $$\frac{1}{\phi_t} > e^{M_1T}\kappa_{t_0}^{-1}\varepsilon^{-1}.$$ We note that $t_0$ only depends on the compact set of initial conditions. Thus,
$$\sup_{t \in [t_0,\tau]}\norm{\nabla z_t}_{L^q} = \sup_{t \in [t_0,\tau]}\norm{\frac{\nabla \phi_t}{\phi_t}}_{L^q} \leq \const \sup_{t \in [t_0,\tau]} \norm{\nabla \phi_t}_{L^q}  \leq C_{\themycounter} t^{-\alpha_1}_0 \norm{\phi}_{\mathcal W_q(\tau)}\in L^\infty(\Omega) $$ and by similar reasoning, we obtain continuity of $z$ in $H^{1,q} \hookrightarrow B^{d/q}_{q,2p}$ on $[t_0,\tau]$. We can conclude that \eqref{eq:zfinalregularity} is valid.
\end{proof}
\begin{corollary}
Assume the setting of the previous proposition. Suppose that $\tau < T$. Then $$\lim_{t \to \tau} \norm{c}_{L^\infty([0,t];B^{d/q}_{q,2p}) \cap L^4([0,t];H^{1,q})} = \infty.$$
\end{corollary}
\begin{proof}
Since the preceding estimates exclude blow-up in $z$, it must be that $c$ diverges. 
\end{proof}
\begin{proposition} 
Again, assume the setting of the preceding propositions. As $h \to 0$, we obtain the uniform the convergence  
\begin{equation}
\norm{z_{t+\cdot}}_{L^\infty(\Omega;(\mathbb W_1 \cap \mathbb W_\delta)(h))} \to 0
\end{equation}
for any fixed $t \geq 0$, and in particular, 
\begin{equation} \label{eq:uniformvanishingt0}
\norm{z}_{L^\infty(\Omega;(\mathbb W_1 \cap \mathbb W_\delta)(h))} \to 0,
\end{equation}
with the obvious truncation at time $\tau$ if $h, t+h > \tau$.
\end{proposition}
\begin{proof}
Let $t > 0$. Then $$\norm{z}_{L^\infty(\Omega;C([t,\tau];H^{1,q}))} \leq t^{-\alpha_1/2} \norm{z}_{L^\infty(\Omega;\mathcal W_q(\tau))}$$ and therefore 
\begin{equation} \label{eq:uniformsmallnessweighted}
\norm{z_{t+\cdot}}_{L^\infty(\Omega;\mathbb W_1(h))} \leq  t^{-\alpha_1} \norm{z}_{L^\infty(\Omega;\mathcal W_q(\tau))} \left(\int_0^h s^{p\alpha_1-1} \d s\right)^\frac1{2p} \leq \const \left(\frac{h}{t}\right)^{\alpha_1/2}\norm{z}_{L^\infty(\Omega;\mathcal W_q(\tau))}
\end{equation}
Choosing $h$ small relative to $t$ yields the claim. The estimate for $\mathbb W_\delta(h)$ is analogous. In the special case that $t = 0$, the unique local solution of \eqref{eq:decoupledcolehopf} vanishes in $\mathbb Z_q(h) \coloneqq \mathbb X_1(h) \cap \mathcal W_q(h)$ as $h \to 0$. This follows by choosing $R$ small enough in \eqref{eq:zsmallball}. Hence, setting $\tau_h \coloneqq \tau \wedge h$ gives
$$
\norm{z}_{L^\infty(\Omega;(\mathbb W_1 \cap \mathbb W_\delta)(\tau_h) \cap \mathcal W_q(\tau_ h))} < \varepsilon$$
for any $\varepsilon > 0$ and  $h \ll 1$. 
\end{proof}
\begin{theorem} \label{thm:c_global_existence}
Assume the setting of Lemma \ref{lemma:phasefieldlocalstrongex}. Assume additionally that $c_0 \in L^p(\Omega;B^{d/q}_{q,2p})$ and that $0 \leq z_0 \in B^{d/q}_{q,2p}$ takes values in a compact set with $$\int e^{-z_0} \d x > \varepsilon > 0$$ for some fixed $\varepsilon > 0$. Then the solution $(z,c)$ of \eqref{eq:StrongColeHopfTrunc} must be global with $$\mathbb E\left[\norm{c}^p_{(\mathbb W_1 \cap \mathbb W_\delta)(T) \cap C([0,T];B^{d/q}_{q,2p})}\right] < \infty.$$ Further, if $c_0 \in \mathcal X_c$ $\d \mathbb P$-a.s., then $c_t \in \mathcal X_c$ for all $t \in [0,T]$, $\d \P$-almost surely and in particular, $(z,c)$ solve equation \eqref{eq:StrongColeHopf}.
\end{theorem}
\begin{proof}
The proof is similar in concept to how we controlled $\phi$ in Lemmas \ref{lemma:initialtimephisol} and \ref{lemma:positive-time-phisol}. We will first prove that the stopping time $\tau^\lambda(0)$ has a deterministic lower bound using the previously derived bounds on $z$, and then improve this iteratively to a global bound. However, in contrast to Prop. \ref{prop:zmoreregularity}, we do not need to repeat the fixed point argument itself.

We first find bounds for $c$ on $[0,\tau_h]$ with $\tau_h \coloneqq \tau \wedge h$, by a very similar splitting procedure as in the proof of Lemma \ref{thm:localexistenceL^p}. For technical reasons, we localise the proof using stopping times $\tau^n_h = \tau^n\wedge \tau \wedge h$, where $\tau^n$ is the relative lifetime after $0$ with $\lambda = n$, cf. Def. \ref{def:localsol} \eqref{item:sollocalregularity}. Let $$\mathfrak X_n \coloneqq (\mathbb W_1 \cap \mathbb W_\delta)(\tau^n_h) \cap C([0,\tau^n_h];B^{d/q}_{q,2p}), \quad \mathfrak X \coloneqq (\mathbb W_1 \cap \mathbb W_\delta)(\tau_h) \cap C([0,\tau_h];B^{d/q}_{q,2p}).$$ Recall that $$c = S(\cdot)c_0 + S\ast \left(\nabla z \cdot \nabla c + \widetilde k(z,c)\right) + S\diamond \widetilde \beta$$
on $[0,\tau)$. Then the maximal regularity theorems show that
$$
\begin{aligned}
\mathbb E\left[\norm{c}^p_{\mathfrak X_n}\right] &\leq \mathbb E\left[\norm{S(\cdot)c_0}^p_{\mathfrak X}\right] + \mathbb E \left[\norm{S \ast(\nabla z \nabla c)}^p_{\mathbb X_1(\tau^n_h)}\right]  \\
&\quad + \mathbb E \left[\norm{S \ast\widetilde k(z,c)}^p_{\mathbb X_1(\tau^n_h)}\right] +  \mathbb E\left[\norm{S \diamond \widetilde \beta(z,c)}^p_{\mathbb Y_\delta(\tau^n_h)}\right]\\
&\leq \const  \mathbb E\left[\norm{c_0}^p_{B^{d/q}_{q,2p}}\right] + \const \mathbb E\left[\left(\norm{z}_{(\mathbb W_1\cap \mathbb W_\delta)(\tau^n_h)} +(\tau^n_h)^{\alpha_\delta}\right)^p\norm{c}^p_{\mathbb W_1(\tau^n_h)}\right] \\
&\quad + \const \mathbb E\left[ (\tau^n_h)^{2p\alpha_1}\right] +  \const \mathbb E\left[\norm{z}^{p}_{\mathbb W_\delta(\tau^n_h)} \right]\\
& \leq \const\left( \mathbb E\left[\norm{c_0}^p_{B^{d/q}_{q,2p}}\right]  + h^{2p\alpha_1} + \norm{z}^{2p}_{L^\infty(\Omega;\mathbb W_\delta(\tau^n_h))} \right) \\
&\quad+ \const\underbrace{\left(\norm{z}^p_{L^\infty(\Omega;(\mathbb W_1 \cap \mathbb W_\delta)(\tau^n_h))}+h^{p\alpha_\delta}\right)}_{\overset{h \to 0}{\to}0}\mathbb E\left[\norm{c}^p_{\mathfrak X_n}\right]
\end{aligned}
$$
where we used $\widetilde k(z,c) \in L^\infty(\Omega \times [0,T]\times \mathbb T^d)$, the bound \eqref{item:NemytskiiBound} from Prop. \ref{prop:nemytskiibounds} and the trace embedding \eqref{item:trace_with_weights_Xap} from Prop. \ref{prop:continuousTrace}. 
As already depicted above, the vanishing property \eqref{eq:uniformvanishingt0} implies that
$$\norm{z}^p_{L^\infty(\Omega;(\mathbb W_1 \cap \mathbb W_\delta)(\tau^n_h))}+h^{p\alpha_\delta}
$$
tends to $0$ uniformly in $h$. 
Fix now $h\ll1$ small enough. Then we can rearrange the inequality to see that, uniformly in $n$,  
$\mathbb E\left[\norm{c}^p_{\mathfrak X_n}\right]  < \infty$, and the limit $n \to \infty$ shows that $$\mathbb E\left[\norm{c}^p_{\mathfrak X}\right]  < \infty.$$
Since $\mathbb W_1(\tau_h) \hookrightarrow L^4([0,\tau_h];H^{1,q})$, we can conclude that $\tau_h$ cannot be a maximal stopping time and thus $\tau > h$. 

The main consequence we can apply \eqref{eq:uniformsmallnessweighted} now for $t=h$, and repeat the above procedure indefinitely for some $h'$ which is small relative to $h$ (since the bound only improves as $t$ gets larger). This directly implies that $\tau = T$. In particular, we find that $$\mathbb E\left[\norm{c}^p_{(\mathbb W_1 \cap \mathbb W_\delta)(T) \cap C([0,T];B^{d/q}_{q,2p})}\right] < \infty.$$
Due to the criticality of the nonlinearity, it is somewhat more cumbersome to prove boundedness of $c$. Namely, it is not directly possible to emulate the strategy employed in e.g. \cite[Thm. 2.24]{Pardoux2021}: the
critical nonlinearity is not absorbed by the dissipativity of the Laplacian. 
This would require the exact borderline integrability $\nabla z \in L^{\frac{2}{\alpha_1}}([0,T];H^{1,q})$, which is not assured.

To finish the proof, we show that 
$$\tau_{\mathcal X} =\inf\left\{t \geq 0: c_t \notin \mathcal X_c\right\} \wedge T = T.$$
Assume the contrary. Note that $c_{\tau_{\mathcal X}} \in \mathcal X_c$ and $c_{\tau_{\mathcal X}} \in L^p(\Omega;B^{d/q}_{q,2p})$. Without loss of generality, let $\tau_{\mathcal X} = 0$. For this type of initial condition, we can consider solutions $c^\theta$ of the doubly truncated equation
\begin{equation} \label{eq:CriticalDecoupledChemTruncGradEq}
\d c^\theta_t =  \left(\Delta  c^\theta_t - \psi_\lambda(t,z)\nabla^\theta z_t \nabla c^\theta_t + \psi_\lambda(t,z)\widetilde k(z_t,c^\theta_t) \right)\d t +\psi_\lambda(t,z)\widetilde \beta(z_t,c^\theta_t) \d W^Q_t
\end{equation} 
where $\psi_\lambda(t,z)$ is the truncation employed in the proof of Theorem \ref{thm:localexistenceL^p} and $\nabla^\theta$ is the truncated gradient defined by \begin{equation} \label{eq:GradientCutoff}
\nabla^\tau u \coloneqq \begin{cases}
\nabla u & \lvert \nabla u \rvert \leq \tau \\
\tau \frac{\nabla u}{\lvert \nabla u \rvert}
&\text{else.}\end{cases}    
\end{equation}
By the first part of this proof, it is evident that global solutions of this equation exist for any $\theta \in (0,\infty]$ even for $\psi_\lambda \equiv 1$. 
Moreover, it now quickly follows by the methods of \cite[Thm 2.24]{Pardoux2021} that $ c^\theta_t \in \mathcal{X}_c $ for all $t \in [0,T]$, due to the invariance Assumption \ref{ass:Invariance}. Let $c^\lambda$ denote the solution of \eqref{eq:CriticalDecoupledChemTruncGradEq} for $\theta = \infty$, so $c^\lambda \equiv c$ on $\tau^\lambda$. Analogous reasoning as in fixed point proofs (see Props. \ref{prop:Fcontraction} and \ref{prop:Gcontraction}) shows that 
$$\begin{aligned}
\mathbb E\left[\norm{c^\theta-c^\lambda}^p_{(\mathbb W_1\cap \mathbb W_\delta)(h)}\right] &\leq \const (\norm{z}^p_{L^\infty(\Omega;(\mathbb W_1\cap \mathbb W_\delta)(h))}+h^{\alpha_\delta}+\lambda+\lambda^{-1}h^{\alpha_\delta})\mathbb E\left[\norm{c^\theta-c^\lambda}^p_{(\mathbb W_1\cap \mathbb W_\delta)(h)}\right] \\
&\quad + \const \mathbb E\left[\norm{c^\lambda}^p_{(\mathbb W_1)(h)} \norm{\nabla z - \nabla^\theta z}^p_{(\mathbb W_1\cap \mathbb W_\delta)(h)} \right].
\end{aligned}$$
Here, the need for a localisation stems from the dispersion coefficient, which is only locally Lipschitz. Since $\norm{\nabla z - \nabla^\theta z}^p_{(\mathbb W_1\cap \mathbb W_\delta)(h)} \in L^\infty(\Omega)$ vanishes $\mathbb P$-almost surely as $\theta \to \infty$, the Lebesgue DCT shows that for small enough $h$ that $$\mathbb E\left[\norm{c^\theta-c^\lambda}^p_{(\mathbb W_1\cap \mathbb W_\delta)(h)}\right]  \leq \const \mathbb E\left[\norm{c^\lambda}^p_{(\mathbb W_1)(h)} \norm{\nabla \phi - \nabla^\theta \phi}^p_{(\mathbb W_1\cap \mathbb W_\delta)(h)} \right] \overset{\theta \to \infty}{\to} 0.$$
Therefore, it must be that $c^\lambda \in \mathcal X_c$, $\d t \otimes \d \mathbb P$-almost surely. By continuity of $c \equiv c^\lambda$, this extends to any $t \in [0,\tau^\lambda]$. Therefore, $\tau_{\mathcal X}$ cannot be maximal and we reach a contradiction.
\end{proof}

\begin{remark}
In the special case $\delta = 1$, $p = \frac{q}{q-d}$, the weight $\beta_1 = 0$ vanishes so $$\mathbb X_1(h) \hookrightarrow \mathbb W_1(h) = L^{2p}([0,h];H^{1,q}).$$ In this case, one can equivalently use the Agresti--Veraar blow up criterion \eqref{eq:AgrestiVeraarBlowup}. For $p > \frac{q}{q-d}$, the weight is strictly positive, but a sliding window approach on intervals $[a,a+h]$ for $a > 0$ should still recover most of the argument above using their criterion. The initial-time estimate $\tau > h>0$ requires a different argument, since the weight at $t= 0$ implies $$\mathbb W_1(h) \not \hookrightarrow L^{2p}([0,h];H^{d/q+1/p,q}).$$ This can probably be circumvented by a localisation on sets $\{\tau > a\}$ for $a > 0$, since $$\mathbb W_1(h)\hookrightarrow L^{2p}([a,h];H^{d/q+1/p,q})$$ for any $ a \in (0,h)$. So for $\delta = 1$, one could in principle also work with the Agresti-Veraar condition, provided one can identify the respective notions of solutions.
\end{remark}

We are now ready to extend this to any measurable initial conditon with $z_0 \geq 0$ almost surely.
\begin{proof}[Proof of Theorem \ref{thm:phasefieldglobalstrongex}]
Let $z_0, c_0$ be measurable with $z_0 \geq 0$. By Ulam's theorem, finite measures on polish spaces can be exhausted by compact sets and thus we can choose increasing sequences of compact sets $Z_n, C_n \subset B^{d/q}_{q,2p}$ such that $\mathbb P(z_0 \in Z_n, c_0 \in C_n) \to 1$. Introduce moreover $$Z_{n, \varepsilon} = \left\{ z_0 \in Z_n: \int e^{-z_0} \d x > \varepsilon \right\}$$
Let $\varepsilon_n \to 0$ be chosen so that $\mathbb P(z_0 \in Z_{n,\varepsilon_n}) > \frac{n-1}{n}\mathbb P(z_0 \in Z_n)$ and set $$z^{(n)}_0 \coloneqq \mathds{1}_{Z_{n,\varepsilon_n}}(z_0)z_0, \quad c^{(n)}_0 \coloneqq \mathds{1}_{C_n}(c_0)c_0.$$ By the preceding result, global probabilistically strong, mild solutions of \eqref{eq:StrongColeHopf} exist for with initial condition $(z^{(n)}_0, c^{(n)}_0)$. By Lemma \ref{lemma:CriticalUniqueness}, localised to the event $\{z_0 \in Z_{n,\varepsilon_n} , c_0 \in C_n\}$ (cf. Prop. \ref{prop:localisation}), these global solutions must coincide with the local solution with initial condition $(z_0, c_0)$ whenever $z_0 \in Z_n$ and $c_0 \in C_n$. Taking $n \to \infty$ now shows that $(z,c)$ is a global solution of \eqref{eq:StrongColeHopf}. Moreover, setting $\phi = e^{-z}$ gives a solution of \eqref{eq:StrongPhaseFieldABS}. 

At last, we show uniqueness of solutions $(\phi,c)$ of the original equation. Let initial data $\phi_0 \in \mathcal X_\phi,c_0 \in \mathcal X_c$ with $\log \phi_0,c_0 \in B^{d/q}_{q,2p}$ be given. Then, by our previous proofs, there exists a unique solution $(z,c)$ of the transformed system. Now, let $(\bar \phi, \bar c)$ denote a mild solution of the phase-field system with the specified regularities, with initial values $\phi_0,c_0$. In particular, it is a local mild solution in the sense of Def. \ref{def:localsol}. We note that $\phi_t \in \mathcal X_\phi, c_t \in \mathcal X_c$ for all $t\in [0,T]$ follows as before.

Let $\bar z$ now denote the solution of \eqref{eq:decoupledcolehopf} with random parameter $\bar c$ and initial value $\log \frac{1}{\phi_0}$. As in the previous sections, we can identify $\bar z = \log \frac{1}{\bar \phi}$ with the appropriate regularities we derived. This in particular implies that $(\bar z, \bar c)$ is a local mild solution of \eqref{eq:StrongColeHopf} and therefore, $z = \bar z, c = \bar c$.
\end{proof}
\section*{Acknowledgement}
Amjad Saef is supported by Deutsche Forschungsgemeinschaft (DFG, German Research
Foundation) under Germany's Excellence Strategy - The Berlin Mathematics
Research Center MATH+ (EXC-2046/1, project ID: 390685689).

\appendix
\section*{Appendix}
\section{Gluings of stochastic convolutions}
\begin{lemma} \label{lemma:stopshift}
Let $\sigma$ be a stopping time,
and denote by $(\mathcal F^\sigma_t)_{t \geq 0} \coloneqq (\mathcal F_{\sigma+t})_{t \geq 0}$ the corresponding shifted 
filtration. Let $G \colon \Omega \times \R_{\geq0} \rightarrow E$ denote an $(\mathcal F^\sigma_t)_{t \geq 0}$-progressively measurable process with values in a topological vector space $E$ and let $\tau(\sigma) \geq 0$ denote a stopping time w.r.t. $(\mathcal F^\sigma_t)_{t \geq 0}$. Then 
\begin{enumerate}[(i)]
    \item\label{item:shift} The process $\widetilde G_t \coloneqq \mathds1_{\{t \geq \sigma\}} G_{(t-\sigma)_+}$ is $\mathcal F_t$-progressively measurable and
    \item\label{item:stoptime} $\sigma + \tau(\sigma)$ is an $\mathcal F_t$-stopping time.
\end{enumerate}
\end{lemma}

\begin{proof}
We first prove that $\widetilde G_t$ is $\mathcal F_t$-progressively measurable as an $E$-valued random variable. Fix $t \geq 0$. Let $\mathcal P_t$ denote the progressive $\sigma$-algebra on $[0,t] \times \Omega$ associated with $(\mathcal F_s)_{s \in [0,t]}$, and let $\mathcal P_t^\sigma$ denote the progressive $\sigma$-algebra on $[0,t] \times \Omega$ associated with $(\mathcal F^\sigma_s)_{s \in [0,t]}$. For $A \subseteq [0,t] \times \Omega$, define
$$
\Gamma_t(A)
\coloneqq
\left\{ (s,\omega) \in [0,t] \times \Omega : \sigma(\omega) \leq s \text{ and } (s-\sigma(\omega),\omega)\in A \right\}.
$$
Set $\mathcal D_t
\coloneqq
\left\{A\in \mathcal P_t^{\sigma} : \Gamma_t(A) \in \mathcal P_t \right\}.$
We claim that $\mathcal D_t=\mathcal P_t^\sigma$. Under this assumption, progressively measurability of  $G$ implies that
$$
A_B
\coloneqq
\left\{( s,\omega )\in [ 0,t]\times\Omega : G_s (\omega)\in B \right\} \in \mathcal D_t
$$
for any Borel set $B \in \mathcal B(E)$. By definition, it follows that
$
\Gamma_t( A_B )
\in\mathcal P_t.
$
But then
$$
\left\{ (s,\omega)\in [ 0,t]\times\Omega : \widetilde G_s(\omega)\in B \right\}
=
\Gamma_t(A_B)
\cup
\left( \Gamma_t([0,t]\times \Omega)^c\cap\left\{ 0 \in B \right\} \right)\in\mathcal P_t.
$$
Hence $\widetilde G_t$ is $\mathcal P_t$-measurable. Since $t\geq 0$ was arbitrary, $\widetilde G$ is $\left( \mathcal F_t \right)_{t \geq 0}$-progressively measurable.

It is left to prove that $\mathcal D_t = \mathcal P^\sigma_t$. First, we show that this set defines a $\sigma$-algebra. By definition, $\Gamma_t(\emptyset) = \emptyset \in \mathcal P_t$. Further,
$$
\Gamma_t([0,t]\times \Omega)
=
\left\{ (s,\omega)\in[0,t]\times\Omega : \sigma(\omega)\leq s \right\}.
$$
Consider the process
$
H_s(\omega)\coloneqq \mathds 1_{\{\sigma(\omega)\leq s\}}$, $s\in[0,t].
$
For every $s\in[0,t]$, the random variable $H_s$ is $\mathcal F_s$-measurable since
$
\{H_s=1\}=\{\sigma\leq s\}\in\mathcal F_s.
$
Moreover, for every $\omega\in\Omega$, the path
$
s\mapsto  H_s(\omega)=\mathds 1_{\{\sigma(\omega)\leq s\}}
$
is càdlàg. Hence $H$ is progressively measurable. Therefore
$$
\Gamma_t([0,t]\times \Omega)
=
\left\{ (s,\omega)\in[0,t]\times\Omega : H_s(\omega)=1 \right\}\in\mathcal P_t.
$$
Further, $\Gamma_t$ commutes with countable unions and
$
\Gamma_t(A \setminus B)
=
\Gamma_t(A) \setminus \Gamma_t(B).
$
It follows that $\mathcal  D_t$ is closed under complements and countable unions. Thus, it remains to check that $\mathcal D_t$ contains a generating class of $\mathcal P_t^{\sigma}$. A standard generating class for $\mathcal P_t^{\sigma}$ is given by sets of the form 
$$
\Set{(a,b ]\times A':
 0 \leq a < b \leq t \text{ and } A' \in \mathcal F_{\sigma+a}} \cup \Set{\Set{0 }\times A':  A' \in \mathcal F_\sigma}.
$$
All we need to show is that this generating class belongs to $\mathcal D_t$. In this case, $\mathcal D_t$ must coincide with $\mathcal P^\sigma_t$. Let first $A=\{0\}\times A'$ with $A'\in\mathcal F_\sigma$. Then
$$
\Gamma_t(A)
=
\left\{ (s,\omega)\in[0,t]\times\Omega : \omega\in A' \text{ and } \sigma(\omega)=s \right\}.
$$
Define
$
H_s(\omega)\coloneqq \mathds 1_{A'}(\omega)\mathds 1_{\{\sigma(\omega)\leq s\}}$ and $
K_s(\omega)\coloneqq \mathds 1_{A'}(\omega)\mathds 1_{\{\sigma(\omega)< s\}}.
$
For every $s\in[0,t]$, one has
$
\{H_s=1\}=A'\cap\{\sigma\leq s\}\in\mathcal F_s
$
by the defining property of $\mathcal F_\sigma$, and
$$
\{K_s=1\}
=
A'\cap\{\sigma<s\}
=
\bigcup_{q\in\mathbb Q\cap[0,s)} A'\cap\{\sigma\leq q\}\in\mathcal F_s.
$$
Thus both $H$ and $K$ are adapted. Moreover, for every $\omega\in\Omega$, the path $s\mapsto H_s(\omega)$ is càdlàg and the path $s\mapsto K_s(\omega)$ is left-continuous. Hence both $H$ and $K$ are progressively measurable. Since
$$
\mathds 1_{\Gamma_t(A)} = H-K,
$$
it follows that $\Gamma_t(A)\in\mathcal P_t$, so $A\in\mathcal D_t$. If $A=\left( a,b \right]\times A'$ with $0\leq a<b\leq t$ and $A'\in\mathcal F_{\sigma+a}$, then
$$
\Gamma_t(A)
= 
\left\{ (s,\omega)\in[0,t]\times\Omega : \omega\in A' \text{ and } \sigma(\omega)+a<s\leq \sigma(\omega)+b \right\}.
$$
Define
$
H_s(\omega)\coloneqq \mathds 1_{A'}(\omega)\mathds 1_{\{\sigma(\omega)+a< s\}}$ and $
K_s(\omega)\coloneqq \mathds 1_{A'}(\omega)\mathds 1_{\{\sigma(\omega)+b< s\}}$.
Since $A'\in\mathcal F_{\sigma+a}$, we find by definition that both $H$ and $K$ are adapted. As before, progressively measurability follows. Finally,
$
\mathds 1_{\Gamma_t(A)} = H-K,
$
and therefore $\Gamma_t(A)\in\mathcal P_t$. Thus $A\in\mathcal D_t$. This completes the proof that $\mathcal D_t=\mathcal P_t^\sigma$.

We can now reduce \eqref{item:stoptime} to \eqref{item:shift}. Let $\tau=\tau(\sigma)$ be a stopping time with respect to the filtration $(\mathcal F^\sigma_t)_{t \geq 0}$ and once again define the progressively measurable process
$
H_s \coloneqq \mathds 1_{\{\tau(\sigma)\leq s\}}.
$
By applying part \eqref{item:shift} to $H$, we find that
$$
\widetilde H_t
\coloneqq
\mathds 1_{\{t\geq \sigma\}} H_{(t-\sigma)_+}
$$
is $\left(\mathcal F_t\right)_{t \geq 0}$-adapted. Since $\tau(\sigma) \geq 0$, we have
$$
\widetilde H_t
=
\mathds 1_{\{t\geq \sigma\}}\mathds 1_{\{\tau(\sigma)\leq t-\sigma\}}
=
\mathds 1_{\{\sigma+\tau(\sigma)\leq t\}}.
$$
Therefore
$
\{\sigma+\tau(\sigma)\leq t\}\in\mathcal F_t.$
This proves that $\sigma+\tau(\sigma)$ is a stopping time with respect to $\left(\mathcal F_t\right)_{t \geq 0}$.
\end{proof}

\begin{corollary} \label{cor:operatorstopshift}
If $G$ satisfies the conditions of Thm. \ref{thm:maxestimatestochconv}, then so does $\widetilde G$. In particular, $\widetilde G$ is stochastically integrable with respect to $W$.
\end{corollary}
\begin{proof}
Naturally, pathwise Lebesgue integrability transfers to the shifted process and thus $\widetilde G \in L^2(\mathbb R_+;\gamma(U,\mathcal B))$. The necessary measurability properties now follow by the preceding lemma.
\end{proof}

\begin{proposition}
Assume the setting of Corollary \ref{cor:operatorstopshift}. Then, for their respective continuous modifications,
$$\int_0^t G_s \d \widetilde { W}_s = \int_0^{\sigma+t} \widetilde G_s \d { W}_s$$ holds for all $t \geq 0$, $\P$-almost surely.
\end{proposition}
\begin{proof}
Fix $t>0$ and set $\widetilde {W}_s \coloneqq {W}_{\sigma+s}-{W}_{\sigma}$ for $s\geq 0$.
By the strong Markov property, $\widetilde {W}$ is a $U$-cylindrical Wiener process with
respect to the shifted filtration $(\mathcal F_{\sigma+s})_{s\geq 0}$.

We first verify the identity in the real-valued case for fixed $t \geq 0$. For elementary progressive processes with values in $U^\ast \cong \gamma(U,\mathbb R)$, the identity follows quickly. Let $0 = t_0 < t_1 < \dots <t_k = t$ and $\xi_m \in \mathcal F^\sigma_{t_m}$ for $m = 1,\dots,k-1$. Then $G = \sum_{m=1}^{k-1}\mathds{1}_{(t_m,t_{m+1}]} \xi_m$
satisfies 
$$
\begin{aligned}
\int_0^t G_s \d \widetilde {W}_s &= \sum_{m=1}^{k-1} \xi_m(\widetilde {W}_{t_{m+1}}-\widetilde {W}_{t_{m}}) \\
&= \sum_{m=1}^{k-1} \xi_m( {W}_{\sigma+t_{m+1}}- {W}_{\sigma +t_{m}}) 
\\ 
&= \int_0^\infty \mathds{1}_{\sigma \leq s \leq \sigma + t}  \sum_{m=1}^{k-1} \mathds{1}_{(t_m, t_{m+1}]}(s-\sigma) \xi_m \d {W}_s.
\end{aligned}
$$
This class of elementary integrands is dense in $\mathcal F^\sigma_t$-progressively measurable processes in $L^2(\Omega \times[0,t];U^\ast)$. Let $G^{(n)} \overset{n \to \infty}{\to} G$ be such an elementary approximating sequence and define
$
\widetilde G^{(n)}_s \coloneqq 
\mathbf 1_{\left\{ \sigma \leq s \leq \sigma+t \right\}}
G^{(n)}_{(s-\sigma)_+}.
$
By a change of variables,
$$
\int_0^\infty \norm{\widetilde G^{(n)}_s-\widetilde G_s}^2_{U^\ast}\,\mathrm dr
=
\int_{\sigma}^{\sigma+t}
\norm{ G^{(n)}_{(s-\sigma)_+}-G_{(s-\sigma)_+}}_{U^\ast}^2 \d s
=
\int_0^t \norm{G^{(n)}_s-G_s}^2_{U^\ast} \d s.
$$
and observe that the right-hand side tends to $0$. After taking expectations, the Itô isometry shows that 
\begin{equation} \label{eq:pointwiseshift}
\int_0^t G_s \d \widetilde {W}_s
\overset{L^2(\Omega)}{=}
\int_0^{\sigma+t} 
\widetilde G_s \d {W}_s.
\end{equation}
At last, by localisation, this extends to to any measurable $G$ with paths in $L^2(\R_{\geq 0};U^\ast)$. 

With this at hand, we tackle fully vector-valued integrals. Let 
$$
I_t
\coloneqq
\int_0^t G_s  \d \widetilde {W}_s
, \quad 
J_t
\coloneqq
\int_0^{\sigma+t}\widetilde G_s \d {W}_s
$$
We prove $I=J$ almost surely by duality. Since $\mathcal B$ is a separable UMD Banach space, Theorem 5.9 in \cite{VanNeervenVeraarStochIntUMD} implies that for every $x^\ast\in \mathcal B^\ast$,
\begin{equation} \label{eq:dualequality}
\left\langle I_t,x^\ast \right\rangle
=
\int_0^t x^\ast G_s \d\widetilde {W}_s=
\int_0^{\sigma+t}x^\ast\widetilde G_s \d {W}_s = \left\langle J_t,x^\ast \right\rangle.
\end{equation}
Since $\mathcal B$ is separable, there exists a countable subset of $\mathcal B^\ast$ which separates points of $\mathcal B$ and \eqref{eq:dualequality} holds for all these functionals on a set of full probability. Therefore, $I_t = J_t$ almost surely. Now, both stochastic integrals have continuous modifications, and thus \eqref{eq:pointwiseshift} extends to all $t\geq 0$.
\end{proof}

\begin{proposition} \label{prop:randomcommutator} Let $G \colon \Omega \times \R_{\geq 0} \rightarrow \gamma(U,\mathcal B)$ satisfy the assumptions of Thm. \ref{thm:maxestimatestochconv}and let $\sigma$ be an $\mathcal F_t$-stopping time. Then the continuous modifications of the stochastic integrals satisfy the identity
$$
P(t -\sigma \wedge t)\int_0^{\sigma \wedge t} P(\sigma \wedge t-r) G_r\d {W}_r = \int_0^{\sigma \wedge t} P(t-r) G_r \d {W}_r
$$
for all $t \geq 0$, $\P$-almost surely.
\end{proposition}
\begin{proof}
We prove the identity for fixed $t \geq 0$. Consider the processes
$$X^t_s \coloneqq P(t-s \wedge t) \int_0^{s \wedge t} P(s \wedge t-r) G_r\d {W}_r$$ 
and 
$$Y^t_s \coloneqq  \int_0^{s \wedge t} P(t-r) G_r \d {W}_r.$$
For fixed $t$, both processes possess almost surely continuous modifications. Continuity of  $X^t_s$ follows from continuity of the stochastic convolution and strong continuity of the semigroup. Meanwhile, $Y^t_s$ is a well-defined stochastic integral and thereby continuous \cite{VanNeervenVeraarStochIntUMD}. For fixed $s > 0$, the processes are almost surely equal and continuity implies indistinguishability of the modifications. In particular, we can infer $X^t_\sigma = Y^t_\sigma$ for each $t \geq 0$. But now an analogous argument shows that $t \mapsto X^t_\sigma$ and $t \mapsto Y^t_\sigma$ have continuous modifications and coincide on a dense set, whence the continuous modifications are indistinguishable.
\end{proof}

\begin{corollary} \label{cor:StochConvStopDecomp}
Let $G, H$ satisfy the conditions of Proposition \ref{prop:randomcommutator}, where $G$ is $\mathcal F_t$-progressive and $H$ is $\mathcal F^\sigma_t$-progressive. Again, let $\sigma$ be a stopping time. Then the progressively measurable process 
\begin{equation} \label{eq:patchwork}
P(t-t \wedge \sigma)\int_0^{t \wedge \sigma}P(t\wedge \sigma-s) G_s \d {W}_s + \mathds{1}_{t \geq \sigma}\int_0^{(t-\sigma)_+} P((t-\sigma)_+ -s)H_s \d \widetilde {W}_s
\end{equation} 
is a modification of 
$$\int_0^t P(t-s)\left(\mathds1_{\{s \leq \sigma\}}G_s + \mathds1_{\{s \geq \sigma\}}H_{(s-\sigma)_+}\right) \d {W}_s.
$$ 
In particular, continuous modifications of these processes are indistinguishable.
\end{corollary}
\begin{proof}
We fix $t \geq 0$. By linearity and the definition of $\widetilde H$,  
$$\int_0^t P(t-s)\left(\mathds1_{\{s \leq \sigma\}}G_s + \mathds1_{\{s \geq \sigma\}}H_{(s-\sigma)_+}\right) \d {W}_s = \int_0^{\sigma \wedge t} P(t-s) G_s \d {W}_s + \int_0^t P(t-s)\widetilde H_s \d {W}_s
$$ and we can immediately apply Proposition \ref{prop:randomcommutator} to the first term. To finish the proof, observe that the stochastic convolutions $t \mapsto \int_0^t P(t-s) H_s \d \widetilde {W}_s$ and $t \mapsto \int_0^t P(t-s) \widetilde H_s \d {W}_s$ are pathwise continuous and progressively measurable w.r.t. the respective filtrations. Naturally, it follows that the randomly shifted path $t \mapsto \int_0^{\sigma+t} P(\sigma+t-s) \widetilde H_s\d {W}_s$ is continuous as well.
Further, for each $t \geq 0$, $$\int_0^t P(t-s) H_s \d \widetilde {W}_s = \int_\sigma^{\sigma+t} P(\sigma+t-s) \widetilde H_s\d {W}_s$$ and a fortiori, the processes are indistinguishable and satisfy the assumptions of Lemma \ref{lemma:stopshift}. Therefore, $$\mathds{1}_{t \geq \sigma}\int_0^{(t-\sigma)_+} P((t-\sigma)_+ -s)H_s \d \widetilde {W}_s = \mathds1_{t \geq\sigma}\int_0^t P(t-s) \widetilde H_s \d {W}_s = \int_0^t P(t-s) \widetilde H_s \d {W}_s,$$ where the equality follows by Lemma 2.5 in \cite{VanNeervenVeraarEvoEq}.
\end{proof}

\bibliography{references}
\end{document}